\documentclass[UTF-8,reqno]{amsart}
\usepackage{enumerate}
\usepackage{mhequ}
\usepackage[margin=1.2in]{geometry}
\usepackage{amssymb,url,color, booktabs,nccmath}
\usepackage{mathrsfs}
\usepackage{enumitem}
\usepackage{graphicx}
\usepackage{tikz}
\usepackage{amsthm}
\usepackage{color}
\usepackage[colorlinks=true]{hyperref}
\hypersetup{
    linkcolor=blue,          
    citecolor=red,        
    filecolor=blue,      
    urlcolor=cyan
}

\definecolor{darkergreen}{rgb}{0.0, 0.5, 0.0}

\numberwithin{equation}{section}
\def\theequation{\arabic{section}.\arabic{equation}}
\newcommand{\be}{\begin{eqnarray}}
\newcommand{\ee}{\end{eqnarray}}
\newcommand{\ce}{\begin{eqnarray*}}
\newcommand{\de}{\end{eqnarray*}}
\newtheorem{theorem}{Theorem}[section]
\newtheorem{lemma}[theorem]{Lemma}
\newtheorem{remark}[theorem]{Remark}
\newtheorem{definition}[theorem]{Definition}
\newtheorem{proposition}[theorem]{Proposition}
\newtheorem{Examples}[theorem]{Example}
\newtheorem{corollary}[theorem]{Corollary}
\newtheorem{assumption}{Assumption}[section]
\newtheorem*{theorem*}{Theorem}
\newtheorem*{remark*}{Remark}

\newcommand{\assign}{:=}

\newcommand{\mathd}{\mathrm{d}}

\newcommand{\nocomma}{}

\def\eps{\varepsilon}

\def\[{{\Big[}}
\def\]{{\Big]}}
\def\<{{\langle}}
\def\>{{\rangle}}
\def\({{\Big(}}
\def\){{\Big)}}

\def\bx{{\mathbf{x}}}
\def\tr{\mathrm {tr}}

\def\dif{{\mathord{{\rm d}}}}

\def\min{{\mathord{{\rm min}}}}

\def\={&\!\!=\!\!&}

\def\1{{\mathbf{1}}}

\def\geq{\geqslant}
\def\leq{\leqslant}

\def\div{\mathord{{\rm div}}}

\def\eps{\varepsilon}

\def\[{{\Big[}}
\def\]{{\Big]}}
\def\<{{\langle}}
\def\>{{\rangle}}
\def\({{\Big(}}
\def\){{\Big)}}

\def\bx{{\mathbf{x}}}
\def\tr{\mathrm {tr}}

\def\dif{{\mathord{{\rm d}}}}

\def\min{{\mathord{{\rm min}}}}

\def\={&\!\!=\!\!&}
\def\bt{\begin{theorem}}
\def\et{\end{theorem}}
\def\bl{\begin{lemma}}
\def\el{\end{lemma}}
\def\br{\begin{remark}}
\def\er{\end{remark}}
\def\bx{\begin{Examples}}
\def\ex{\end{Examples}}
\def\bd{\begin{definition}}
\def\ed{\end{definition}}
\def\bp{\begin{proposition}}
\def\ep{\end{proposition}}
\def\bc{\begin{corollary}}
\def\ec{\end{corollary}}

\def\geq{\geqslant}
\def\leq{\leqslant}

\def\div{\mathord{{\rm div}}}

\def\<{\langle} \def\>{\rangle}

\allowdisplaybreaks

\begin{document}

\title{quantitative particle approximations  of stochastic 2D Navier-Stokes equation}

\author{Yufei Shao}
\address[Y.Shao]{School of Mathematics and Statistics, Beijing Institute of Technology, Beijing 100081, China
}
\email{yufeishao@bit.edu.cn}

\author{Xianliang Zhao}
\address[X. Zhao]{ Academy of Mathematics and Systems Science,
	Chinese Academy of Sciences, Beijing 100190, China; Fakult\"at f\"ur Mathematik, Universit\"at Bielefeld, D-33501 Bielefeld, Germany}
\email{xzhao@math.uni-bielefeld.de}

\begin{abstract}
	In this article, we investigate an interacting particle system featuring random intensities, individual   noise, and environmental   noise, commonly referred to as stochastic point vortex model. The model serves as an approximation for the stochastic 2-dimensional Navier-Stokes equation. We establish a quantitative mean-field convergence for the stochastic 2-dimensional Navier-Stokes equation in the form of relative entropy. To address challenges posed by environmental noise, random intensities, and  singular kernel, we compare relative entropy for conditional distributions, employing technology of disintegration and the relative entropy method developed by Jabin and Wang in \cite{jabin2018quantitative}.

\end{abstract}


\date{\today}
\thanks{ Y.S. is grateful
	to the financial supports by  National Key R\&D Program of China (No. 2022YFA1006300) and the financial supports of the NSFC (No. 12271030). X.Z. is supported by the DFG through the CRC 1283 “Taming uncertainty and profiting from randomness and low regularity in analysis, stochastics and their applications”.}
\maketitle

\setcounter{tocdepth}{1}

\tableofcontents

\section{Introduction}
\subsection{Motivation and main results}
The goal of this work is to show  a quantitative mean-field result that derives  the vorticity formulation  of the stochastic 2-dimensional Navier-Stokes equation:
\begin{equation}\label{eqt:ns}
	\mathd   v=\(\Delta   v-K*  v \cdot \nabla   v\)\mathd t-\sum_{k=1}^{d}\(\sigma_k\cdot\nabla   v \)\circ \mathd W^{k}_t  , \quad (t,x)\in [0,T]\times \mathbb{T}^2,  \quad   v(0,\cdot)=  v_0.
\end{equation} 
Here $\Delta$ denotes the Laplace operator on the two dimensional torus $\mathbb{T}^2=[-\pi,\pi]^{2},$ 
 $T>0$   and   $K$ stands for  the Biot-Savart law on $\mathbb{T}^2$, namely 
\begin{align}
	K(x)=\sum_{m=(m_1,m_2)\in \mathbb{Z}^2\setminus{\{0\}}}\frac{im^{\perp }}{|m|^2}e^{im\cdot x}, \quad m^{\perp }=(m_2,-m_1),\quad x=(x_1,x_2)\in \mathbb{T}^2.
\end{align}
The  random  environmental perturbation in \eqref{eqt:ns}, commonly referred  as transport type noise,  is  described by  smooth divergence-free vector fields $\{\sigma_k, k=1,\cdots ,d\}$  for $d \in \mathbb{N}$  and independent standard 1-dimensional Brownian motions $\{W^{k}, k=1,\cdots ,d\}$ on $\mathbb{T}^2.$   Here $\circ$ denotes the stochastic integral in the Stratonovich sense.

The derivation of macroscopic models from interacting particle systems is a fascinating and active research topic in mathematical physics. It simplifies the complexities of many-body systems considerably. The basic idea is that, as the number of particles increases to infinity, the macroscopic model can effectively describe the universal properties of particles. In this article, we derive the stochastic 2-dimensional Navier-Stokes equation \eqref{eqt:ns} from the stochastic vortex model \eqref{eqt:vortex} with random intensities, individual   noise, and common environmental   noise, given as:
\begin{equation}\label{eqt:vortex}
	X_i(t)=X_i(0)+\frac{1}{N}\sum_{j\neq i}\int_0^t\xi_jK(X_i(s)-X_j(s))\mathd s+\sqrt{2}B_i(t)+\sum_{k=1}^{d}\int_0^t\sigma_k(X_i(s))\circ\mathd W^k_s.
\end{equation}
 Here $\{W^k ,k =1,\cdots, d\}$  denote independent 1-dimensional Brownian motions  on torus $\mathbb{T}^2,$ and $\{B_i,i=1,\cdots,N\}$ are independent 2-dimensional Brownian motions  on torus $\mathbb{T}^2,$  on probability space $(\Omega,\mathcal{F},\mathbb{P}),$ for $d,N  \in \mathbb{N}.$ Additionally,  a class of i.i.d $\mathbb{R} \times \mathbb{T}^2$ valued random variables $\{(\xi_i, X_i(0)), i=1,\cdots,N\}$ exist  on the same probability space $(\Omega,\mathcal{F},\mathbb{P}).$ We denote the identical distributions for $\{X_i(0),i=1,\cdots,N\}$ and $\{\xi_i,i=1,\cdots,N\}$ by $ \mathcal{L}(X(0))\text{ and } \mathcal{L}(\xi)$ respectively. We also suppose the  mutual independence among $\{X_i(0),i=1,\cdots,N\},$ $\{\xi_i,i=1,\cdots,N\},$ $\{B_i,i=1,\cdots,N\},$ and $\{W^k,k=1,\cdots, d\}.$   Let $(\mathcal{F}_t)_{t\in [0,T]}$ denote the normal filtration  generated by
1-dimensional Brownian motions $\{W^k, k=1,\cdots,d\},$ 2-dimensional Brownian motions $\{B_i,i=1,\cdots,N\}$ and $\{(\xi_i, X_i(0)), i=1,\cdots,N\}.$

In the interacting particle system  \eqref{eqt:vortex},  $\sum_{k=1}^{d}\int_0^t\sigma_k(X_i(s))\circ\mathd W^k_s$ represents the common environmental   noise shared by all particles. The individual noise for $i$-th  point vortex is represented by $B_i,$ as well as the i.i.d random variables  $\{\xi_j,j=1,\cdots, N\}$   signify the intensities/magnitudes of the $j$-th point vortex at the position  $X_j.$ 
The interaction between each pair of particles is of the scale $\frac{1}{N}$, causing the impact of any single particle on another to become negligible, as $N\rightarrow \infty.$ This is referred to as {\em weakly interacting}.

The mean field limit for the first-order systems, exemplified by \eqref{eqt:vortex} with $\sigma_k=0, k=1,\cdots,d, $ has been extensively studied over the last decade. In our setting, the very first mathematical investigation can be traced back to McKean in \cite{mckeanpropagation}. Additionally, classical mean field limits from  Newton dynamics towards Vlasov Kinetic PDEs are well-documented in \cite{dobrushin1979vlasov,braun1977vlasov,hauray2007n,lazarovici2016vlasov}  along with a comprehensive review in \cite{jabin2014review}. Recent advancements in the mean field limit for systems, such as interacting SDEs/ODEs with general singular interaction kernels, have been noteworthy. This includes results that specifically focus on the point  vortex model \cite{osada1986propagation,fournier2014propagation} as well as more recent convergence results involving general singular kernels, as seen in  \cite{duerinckx2016mean,jabin2018quantitative, bresch2020mean,serfaty2020mean,rosenzweig2020mean,nguyen2021mean,lacker2021hierarchies,bresch2022new,rosenzweig2023global}.  Further references in these works provide a more comprehensive development of the mean field limit.

Our model falls within the category of the point vortex model. The point vortex approximation towards the determinstic 2-dimensional Navier-Stokes/Euler equations, i.e, $\sigma_k=0,k=1,\cdots,d,$ has garnered significant interest since the 1980s, as documented in  \cite{osada1986propagation,fournier2014propagation,jabin2018quantitative,guillin2021uniform}.  
Notably, a quantitative estimate of propagation of chaos has been established  by Jabin and Wang in \cite{jabin2018quantitative} by evolving  the relative entropy. This work  approximates the stochastic two-dimensional Navier-Stokes equation on the torus by developing Jabin and Wang's approach. Recently, Wang and Feng in \cite{feng2023quantitative} generalized  results in \cite{jabin2018quantitative} to the whole space through detailed analysis of the regularity of solutions to the two-dimensional Navier-Stokes equation. Likewise, our results may also be extended to the whole space, which will be the focus of future work.

It is natural to consider that interacting particle systems are influenced by environmental noise, inspired by mean field games (see e.g. \cite{carmona2018meanfield}),  and the phenomenon called regularization by noise (see e.g.  \cite{flandoli2011random}).
Two types of results emerge in the case of environmental noise. Firstly, the mean field limit equation becomes a stochastic partial differential equation, where stochasticity reflects the random environment that persists in the limit. Another intriguing result is that the mean field limit equation is still a deterministic partial differential equation by simultaneously rescaling the space covariance of the noise with the increasing number of particles. 

In the first case, Coghi and Flandoli \cite{coghi2016propagation}  studied the mean-field problem  for Lipschitz interactions, and recently   Nguyen, Rosenzweig, and Serfaty  \cite{nguyen2021mean}  studied the singular kernel cases.
 When the kernel is  Biot-Savart law, Rosenzweig   \cite{rosenzweig2020mean1} obtained a quantitative convergence of the stochastic Euler equation by the modulated energy method. 
 Additionally, we mention the result \cite{flandoli2021point}  by Flandoli and Luo, in which they approximated the stationary solution of the stochastic 2-dimensional Navier-Stokes equation by the point vortex model with common   noise.
The difference between the stochastic point vortex model \eqref{eqt:vortex} we used and the stochastic point vortex  model used in \cite{rosenzweig2020mean1} and \cite{flandoli2021point} lies in the introduction of individual   noise $\{B_i,i=1,\cdots,N\},$ which is the reason for the appearance of $\Delta  v$  describing the viscosity of the fluid in our final mean field limit equation \eqref{eqt:ns}, i.e. the stochastic 2-dimensional Navier-Stokes equation.  While the Laplacian in the stochastic 2-dimensional Navier-Stokes equation in \cite{flandoli2021point} comes from the rescaled Stratonovich type common   noise. Similiar Stratonovich type common noise without rescaling ($\sum_{k=1}^{d}\sigma_k(X_i(t))\circ\mathd W^k_t$ in this paper)  are completely preserved as the noise term in the mean field limit equation in our work (i.e. $\sum_{k=1}^{d}\(\sigma_k\cdot\nabla   v \)\circ \mathd W^{k}_t$)  and  \cite{rosenzweig2020mean1}'s work. Additionaly, we also mention \cite{coghi2020regularized}  for a scaling limit result for  stochastic 2-dimensional Euler equations on point vortices with regularized Biot-Savart kernel.

In the second case, we highlight the following  results. In \cite{flandoli2021mean},  Flandoli and Luo  approximate the deterministic 2-dimensional Navier-Stokes equation under a suitable scaling of the   noise. In \cite{guo2023scaling}, Guo and Luo obtained the scaling limit of moderately interacting particle systems with suitably smoothed singular interactions, including regularized version of Biot-Savart kernel and repulsive Poisson kernel.  

The primary contribution of this paper lies in establishing a convergence rate from the particle system $\eqref{eqt:vortex}$ to the stochastic 2-dimensional  Navier-Stokes equation $\eqref{eqt:ns}$. To the best of our knowledge, this is the first quantitative convergence result for stochastic 2-dimensional Navier-Stokes equation. 
To articulate our results, we introduce the following random distributions on $\mathbb{T}^{2N}\times\mathbb{R}^N.$ 

Let $F^N_t(\mathd x^N,\mathd z^N)$ denote the conditional distribution of $(X^N(t),\xi^N)$ described by  interacting particle system \eqref{eqt:vortex} given the environmental noise $\{W^k, k=1,\cdots, d\},$ i.e.
\begin{equation}
	F^N_t(\mathd x^N,\mathd z^N)\assign \mathcal{L}((X^N(t),\xi^N)|\mathcal{F}_{T}^{W})(\mathd x^N,\mathd z^N), 	\nonumber
\end{equation}
where $X^N$ is the probabilistically strong  solution to the interacting particle system \eqref{eqt:vortex} and $(\mathcal{F}_t^W)_{t\in [0,T]}$  denotes the natural filtration generated by $\{W^k, k=1,\cdots, d\}.$
Let  $\bar{F}^N_t(\mathd x^N,\mathd z^N)$ defined by
\begin{align*}
	\bar{F}^N_t(\mathd x^N,\mathd z^N)\assign\frac{1}{\mathbb{E}[\xi_1]^N}   v^{\otimes N}_t(\mathd x^N)\mathcal{L}(\xi)^{\otimes N}(\mathd z^N),
\end{align*}
where $(  v_t)_{t\in [0,T]}$ is  the  probabilistically strong  solution to the stochastic 2-dimensional Navier-Stokes  equation \eqref{eqt:ns} with the initial data \begin{align}\label{initial data}
	  v_0(\mathd x)\assign\mathbb{E}[\xi_1]\mathcal{L}(X(0))(\mathd x).
\end{align} 
We finally establish a quantitative mean-field convergence concerning the normalized relative entropy (see the definition for the normalized relative entropy in Section \ref{entropy}) of $F^N_t$ and $\bar{F}^N_t,$ denoted as $H_N(F_t^N|\bar{F}_t^N)$, as stated in Theorem \ref{thm:entropy}.
\begin{theorem}\label{thm:entropy}
Assume that
	\begin{enumerate}
	
	\item The probability measure $\mathcal{L}(X(0)) $  on $\mathbb{T}^2$ has a  density $\bar{\rho}_0\in H^{3}(\mathbb{T}^2)$ and there exists a constant $a>0$ such that $\underset{x\in \mathbb{T}^2 }{\inf}\bar{\rho}_0>a$,
	\item The probability measure $\mathcal{L}(\xi)$ on $\mathbb{R}$  has a density with compact support and   $\int_{\mathbb{R}}x\mathcal{L}(\xi)(\mathd x)>0.$
\end{enumerate} 
It holds that
 \begin{align*}
 	\sup_{[0, T]}H_N(F_t^N|\bar{F}^N_t)\leq \frac{1}{N}\exp\bigg(C_{\xi,0}\int_{0}^{T}(\|  v_t\|_{H^4}^2 +1)\mathd t\bigg)\quad\mathbb{P}-a.s..
 \end{align*}

In particular,
\begin{align*}
	\mathbb{E}\[\exp\bigg(-C_{\xi,0}\int_{0}^{T}(\|  v_t\|_{H^4}^2 +1)\mathd t\bigg)	\sup_{[0, T]}H_N(F_t^N|\bar{F}^N_t)\] \leq \frac{1}{N},
\end{align*}
where $H_N(\cdot|\cdot)$ denotes the nomalized relative entropy functional, $C_{\xi,0}$ is a positive deterministic constant depending on $\|v_0\|_{L^2(\mathbb{T}^2)},\underset{x\in \mathbb{T}^2 }{\inf}\bar{\rho}_0$ and $\mathcal{L}(\xi),$
and  $(  v_t)_{t\in [0,T]}$ is the probabilistically strong solution to the stochastic 2-dimensional Navier-Stokes  equation \eqref{eqt:ns} with the initial data $  v_0=\mathbb{E}[\xi_1]\bar{\rho}_0$  in the sense of Definition \ref{defnss}.
\end{theorem}

By applying the Csis\'zar-Kullback-Pinsker inequality, we directly arrive at the following quantitative estimate for the total variation distance between weighted $k$-magrinal (see the definition in Notations below) conditional distribution for the interacting particle system \eqref{eqt:vortex} and $k$-tensorized law of the probabilistically strong solution to the stochastic 2-dimensional Navier-Stokes equation \eqref{eqt:ns}. 

\begin{corollary}\label{cor:total}Under the same assumptions  in Theorem \ref{thm:entropy},  the following estimates hold for all $k\leq N$: 
	\begin{align*}
		\sup_{[0, T]}	\left\| \frac{1}{\mathbb{E}[\xi_1]^k}\int_{\mathbb{R}^k}  \prod_{i=1}^k z_i F_t^{N,k} \(\mathd x^k, \mathd z^k\)-\frac{1}{\mathbb{E}[\xi_1]^k}  v_t^{\otimes k}\right\|_{TV}\leq \frac{C_{k,\xi}}{\sqrt{N}}\exp\bigg(C_{\xi,0}\int_{0}^{T}(\|  v_t\|_{H^4}^2 +1)\mathd t\bigg)\quad\mathbb{P}-a.s.,
	\end{align*}
and
\begin{align*}
	\mathbb{E}\bigg[\exp\bigg(-C_{\xi,0}\int_{0}^{T}(\|  v_t\|_{H^4}^2 +1)\mathd t\bigg)\sup_{[0, T]}	\left\| \int_{\mathbb{R}^k}\frac{1}{\mathbb{E}[\xi_1]^k}  \prod_{i=1}^k z_i F_t^{N,k} \(\mathd x^k, \mathd z^k\)-\frac{1}{\mathbb{E}[\xi_1]^k}  v_t^{\otimes k}\right\|_{TV}\bigg] \leq \frac{C_{k,\xi}}{\sqrt{N}},
\end{align*}
where  $F_t^{N,k}$ denotes $k$-marginal  of $F_t^{N},$  $\|\cdot \|_{TV}$ denotes the total variational norm, $C_{\xi,0}$  and $(  v_t)_{t\in [0,T]}$ are the same as in Theorem \ref{thm:entropy} and $C_{k,\xi}$ is a positive deterministic constant  depends on $k$ and $\mathcal{L}(\xi).$

\end{corollary}
A notable distinction between our work and the investigations addressing the (stochastic) Navier-Stokes (Euler) equation in the environmental noise case mentioned earlier,  is their initiation from the standpoint of empirical measures of particle system \eqref{eqt:vortex}, leading to the derivation of the law of large numbers type results, i.e. formally \begin{align*}
	\mu_N(t)\assign\frac{\sum_{i=1}^{N}\xi_i\delta_{X_i(t)}}{N}\rightharpoonup  v_{t},\quad \text{as }N\rightarrow\infty,
\end{align*} 
where $\delta_{x},x\in\mathbb{T}^2$ denotes the Dirac measure at $x$  and $(  v_{t})_{t\in[0,T]}$ is the solution to the mean field limit equation \eqref{eqt:ns}.
Our approach begins with the conditional distribution of particle system \eqref{eqt:vortex}, culminating in a  {\em conditional propagation of chaos} type result, i.e. Theorem \ref{thm:entropy} and Corollary \ref{cor:total}.
Furthermore, we derive the convergence rate of the conditional probability measure of $(X^N,\xi^N)$ from the interacting particle system \eqref{eqt:vortex}.  
\subsection{Methodology and difficulties.}
The results in the paper can be decomposed into three steps: step 1,  well-posedness of the mean field equation \eqref{eqt:ns},  step 2,  well-posedness of the particle system \eqref{eqt:vortex} and step 3, establishing a quantitative mean-field convergence.

In the initial step, our goal is to obtain the well-posedness of the mean field equation \eqref{eqt:ns}.
There have been established results regarding  it, e.g. \cite{brzezniak2016existence}, \cite{flandoli2019rho}, \cite{galeati2023weak}. We will not only establish  well-posedness of  \eqref{eqt:ns}  but also derive regularity  estimates (see \eqref{regu-1} and \eqref{regu-2} below), which are crucial for the application of the relative entropy method.
It is worth mentioning that the required estimates of the solutions are obtained through applying the Lagrangian approach (e.g. \cite{marchioro2012mathematical}, \cite{brzezniak2016existence}), as well as carefully applying the Biot-Savart law  $K$ written as $\div V$ with a potential $2\times2$ matrix $V$ satisfying $V_{\alpha\beta}\in L^{\infty}(\mathbb{T}^2),\alpha ,\beta=1,2,$ introduced in Section \ref{sec:relaen} and the divergence free property for  Biot-Savart law $K.$  

Additionally, we establish a connection between stochastic 2-dimensional Navier-Stokes equation \eqref{eqt:ns} and the conditional McKean-Vlasov equations \eqref{eqt:ncopy} below:
\begin{equation}\label{eqt:ncopy}
	Y_i(t)=X_i(0)+\int_0^tK*  v_s(Y_i(s))\mathd s+\sqrt{2}B_i(t)+\sum_{k=1}^{d}\int_0^t\sigma_k(Y_i(s))\circ\mathd W^k_s, \quad i=1,\cdots,N,
\end{equation}
where $(  v_t)_{t\in [0,T]}$ is the probabilistically strong solution to the stochastic 2-dimensional Navier-Stokes  equation \eqref{eqt:ns} with the initial data $  v_0=\mathbb{E}[\xi_1]\mathcal{L}(X(0))(\mathd x)$  in the sense of Definition \ref{defnss}.
  The connection between  equations \eqref{eqt:ns} and   \eqref{eqt:ncopy} is given by
  \begin{align}
  	\mathbb{P}(\|  v_t-\tilde{  v}_t\|_{H^{-(2+  \eps )}}=0, \forall t\in [0,T])=1,\quad \forall   \eps \in (\frac{1}{2},1),
  \end{align}
  where $\tilde{  v}_t(\mathd x)$ is the conditional distribution of $Y_i(t)$ given the environmental noise, i.e. $$\tilde{  v}_t(\mathd x)\assign\int_{\mathbb{R}} z \mathcal{L}((Y_i(t),\xi_i)| \mathcal{F}^{W}_T)(\mathd x\mathd z)=\mathbb{E}[\xi_i]\mathcal{L}((Y_i(t))| \mathcal{F}^{W}_T)(\mathd x),$$ Here the final equality is a result of the independence $\xi_i \perp\mathcal{F}^W_T$ introduced in  our setting.
  This connection is established by an uniqueness result regarding the linearized stochastic 2-dimensional Navier-Stokes equation.   
Due to the connection between \eqref{eqt:ns} and \eqref{eqt:ncopy}, our final result, i.e. Theorem \ref{thm:entropy}, can be then understood as follows: as the number of particles approaches infinity, the particles within the weakly interacting particle system \eqref{eqt:vortex} tend to behave  like  some identical distributed particle $Y_i$ from \eqref{eqt:ncopy}.

In the subsequent step,  in order to ensure the comprehensiveness of our investigation, we provide a proof of the well-posedness of the particle system \eqref{eqt:vortex} by adopting established methods, as exemplified in works such as \cite{fontbona2007paths} and \cite{flandoli2011full}.

In the final step, we aim to establish a quantitative mean-field convergence. There are three primary challenges in studying mean-field convergence, including the singularity of Biot-Savart law, environmental   noise, random intensities. We will develop the relative entropy method introduced by Jabin and Wang \cite{jabin2018quantitative} to deal with the singularity problem.
In the case without environmental noise, researchers commonly compare the relative entropy between the distribution of particles described by \eqref{eqt:vortex} and the solution of the mean field limit equation \eqref{eqt:ns}. However, when environmental noise is introduced,  it is natural to compare the conditional distribution for the interacting particle system \eqref{eqt:vortex} with the solution to the mean field limit equation \eqref{eqt:ns}. During this process, we encounter two main difficulties. The first challenge arises in calculating the time evolution of the relative entropy, i.e. Lemma \ref{lem:entro-2}. The usual method  requires a certain regularity in time of the solution to the mean limit equation \eqref{eqt:ns} (commonly requiring $W^{1,\infty}$).  However, due to the introduction of environmental noise and the time regularity of Brownian motion being $C^{\frac{1}{2}-}$, the regularity in time of the solution to \eqref{eqt:ns} is at most $C^{\frac{1}{2}-}$. Therefore, this method cannot be applied.  Our idea for this challenge is to  obtain a time evolution of the regularized version at first and  utilize the weak convergence of the conditional distribution of the solution to \eqref{eqt:approvortex} to achieve the final result, which requires the spatial regularity of the solution to \eqref{eqt:ns} and the property of Fisher information (see Lemma \ref{fournier3.3}).  The second difficulty, arising from random intensities $\xi^N,$ necessitates considering the relative entropy of the joint distribution on $\mathbb{R}^N\times\mathbb{T}^{2N}$. By using disintegration and  property of relative entropy \eqref{con:Frho}, we transform the relative entropy of the joint distribution on $\mathbb{R}^N\times\mathbb{T}^{2N},$ i.e. $H_N(F_t^N|\bar{F}_t^N),$  into an integral of the relative entropy of the distribution on $\mathbb{T}^{2N},$ with respect to the distribution of intensities $\xi^N.$ 
By this connection, we obtain the estimate  for the time evolution of $H_N(F_t^N|\bar{F}_t^N)$ by calculating the estimate  for the time evolution of the relative entropy between the  distributions on $\mathbb{T}^{2N}$ and the intensities $\xi^N$ are regarded as constants during this process. Through two critical large-deviation type estimates (see Lemma \ref{lemma jw1} and Lemma \ref{lemma jw2}), we arrive at the final result, i.e. Theorem \ref{thm:entropy}.
\newline

\paragraph{\textbf{Organization of the paper:}}
 Section \ref{sec:pre} provides an introduction to relevant definitions, assumptions, and auxiliary results that are used in the subsequent proofs. The main sections of the paper are Section \ref{sec:conditionalmv} and Section \ref{sec:quant}. In Section \ref{sec:conditionalmv}, we study the stochastic 2-dimensional Navier-Stokes equation \eqref{eqt:ns} and the conditional McKean-Vlasov equation \eqref{eqt:ncopy}, as well as their connections. Moving on to  Section \ref{sec:quant}, we derive a quantitative mean-field convergence by developing the relative entropy method  from  \cite{jabin2018quantitative}. The proofs of well-posedness for  stochastic 2-dimensional Navier-Stokes equation \eqref{eqt:ns} and a regularity result of  \eqref{eqt:ns} are presented in Section \ref{sec:ns}. For simplicity, well-posedness of the interacting system \eqref{eqt:vortex} is detailed in Appendix \ref{appendix:vortex}. We  also provide a disintegration result including a short proof in Appendix \ref{sec:disinte} for the sake of completeness.
\newline
\paragraph{\textbf{Notations:}}

 \begin{enumerate}
	\item Bracket notations: The bracket $\<\cdot, \cdot\>$  denotes integrals when the space and underlying measure are clear from the context. We use a similar bracket $[\cdot ,\cdot]_t$ to denote quadratic variations between local martingales at time $t$. 
	\item  Filtration notations: The notation $\mathcal{F}_1\vee \mathcal{F}_2$ stands for the $\sigma$-algebra generated by $\mathcal{F}_1\cup \mathcal{F}_2$. Given a stochastic process $X_t, t\in [0,T]$ and random variable $\xi$, we use $(\mathcal{F}_t^X)_{t\in [0,T]}((\mathcal{G}_t^X)_{t\in [0,T]})$ to denote the natural (normal) filtration 
	generated by $X$ and $\mathcal{F}^{\xi}$ to denote the $\sigma$-algebra generated by $\xi.$ 
	In particular, we  use $(\mathcal{F}_t^W)_{t\in [0,T]}$ to denote the natural filtration generated by independent  1-dimensional Brownian motions $\{W^{k}, k=1,\cdots ,d\}$ and $(\mathcal{F}_t^{\xi^N,W})_{t\in [0,T]}$ to denote the natural filtration  generated by independent  1-dimensional Brownian motions $\{W^{k}, k=1,\cdots ,d\}$ and $\xi^N$. We also use $(\mathcal{F}_t)_{t\in [0,T]}$ to denote the normal filtration generated by
	1-dimensional Brownian motions $\{W^k, k=1,\cdots,d\},$ 2-dimensional Brownian motions $\{B_i,i=1,\cdots,N\}$ and $(\xi_i, X_i(0))_{i=1,\cdots,N}.$ Finally, given a Polish space $E,$ we use $\mathbb{B}(E)$ to denote the Borel $\sigma$-algebra of $E$.
	\item  Distribution notations: Given a  Polish space $E$ and probability space $(\Omega,\mathcal{F},\mathbb{P}),$ we say $Q(\mathd x)(\omega)$ is a random meausre on $E,$ if $Q(\mathd x)(\omega)$ is a function of two variables $\omega\in \Omega$ and $A\in\mathbb{B}(E),$ satisfying that $Q(\mathd x)(\omega)$ is a $\mathcal{F}-$measurable  in $\omega$ for fixed $A$ and there exists a null set $N\in \mathcal{F}$ such that $Q(\mathd x)(\omega)$ is a measure in $A$ for fixed $\omega\in N^c.$   Given $\sigma$-algebra $\mathcal{F}$ and and random variable $X,$ we use $\mathcal{L}(X),$ $\mathcal{L}(X|\mathcal{F})$ to denote the distribution of $X$ and conditional distribution of $X$ with respect to  $\mathcal{F}.$  In particular, for convention, we may denote the distribution by its density function when the distribution has a density function. Given a symmetric probability measure $\rho^N$ on $E^N$ where $E$ is a Polish space, the $k$-marginal $\rho^{N,k}$ is a probability measure on $E^k$ defined by $\int_{E^{N-k}}\rho^N(\mathd x_1\cdots \mathd x_{N})$ where $k\leq N.$ Finally, we use $\mathcal{P}(E)$ to denote the space of probability measure on $E.$ 
	\item  Independence and Conditional independence: Given three $\sigma$-algebras $\mathcal{F}_1$, $\mathcal{F}_2$, and $\mathcal{F}_3,$  we use $\mathcal{F}_1\perp\mathcal{F}_2|\mathcal{F}_3$ to  indicate that $\mathcal{F}_1$ and $\mathcal{F}_2$ are conditionally independent given $\mathcal{F}_3.$ $\mathcal{F}_1\perp\mathcal{F}_2$ indicates that $\mathcal{F}_1$ and $\mathcal{F}_2$ are mutually independent. 
	\item  Product space and Product function:  Given two measure spaces $(\Omega_1,\mathcal{A}_1,\mu_1)$ and $(\Omega_2,\mathcal{A}_2,\mu_2),$  we denote their product space as $\Omega_1 \times \Omega_2,$ the product $\sigma$-algebra  as $\mathcal{A}_1 \times \mathcal{A}_2,$  and the product measure  as $\mu_1 \times \mu_2.$ For $k\in \mathbb{N},$ we use $(\Omega_1^{\otimes k},\mathcal{A}_1^{\otimes k},\mu_1^{\otimes k})$ to denote the $k$-product measure space for  the measure space $(\Omega_1,\mathcal{A}_1,\mu_1).$ Given a function $f(x),x\in E$ on Polish space $E$ and $k\in \mathbb{N},$ the $k$-tensorized function $f^{\otimes k}(x^k)$ is defined by $\prod_{i=1}^{k}f(x_i),$ where $x^k=(x_1,\cdots,x_k)\in E^k.$
\end{enumerate}

Throughout this paper, we use $C$ to denote universal constants, and we indicate relevant dependencies using subscripts when necessary. We use the notation $a\lesssim b$ if there exists a universal constant
$C > 0$ such that $a\leq Cb.$
We also use some notations about matrices and vectors to simplify the computations, e.g., $\nabla f$  denotes
$\begin{pmatrix}
	\partial_{1} f \\ \partial_{2} f
\end{pmatrix}$, 
$\nabla^{ j} f$ denotes $\nabla\otimes(\nabla^{j-1}f)^T,$ where $A^T$ denotes the transpose of the matrix $A.$ 
We work mostly with Sobolev spaces $H^{k}(\mathbb{T}^2)$ and the norm of Sobolev space $H^{k}(\mathbb{T}^2), k\in \mathbb{N}$, is defined as 
\begin{align*}
	\|f\|_{H^{k}(\mathbb{T}^2)}^2\assign \sum_{0\leq j\leq k}\|\nabla^{j} f \|_{L^2(\mathbb{T}^2)}^2.
\end{align*}
For simplicity, we will also use a similar notation $\mathbb{H}^i, i\in \mathbb{N}$, to denote the $L^2$-norm of $i$-th order derivatives in Section \ref{sec:ns}, i.e, 
\begin{align}
	\|f\|_{\mathbb{H}^i}\assign \|\nabla^{ i}f\|_{L^2(\mathbb{T}^2)}.\nonumber
\end{align}
Finally, $A:B$ denotes $\sum_{\alpha,\beta}A_{\alpha\beta}B_{\beta\alpha}$, where $A$ and $B$ are matrices. 
\section{Preliminary} \label{sec:pre}
In this section, we introduce the assumptions and present several auxiliary results. 
\subsection{Assumptions} \label{sec:assump}
We fix $T>0$ and $d\geq 1$ in this paper.
\begin{assumption}\label{initial}
	We make the following  assumptions on the law for the initial values and intensities: 
	\begin{enumerate}
		
		\item \label{assumption1}The probability measure $\mathcal{L}(X(0)) $  on $\mathbb{T}^2$ has a  density $\bar{\rho}_0\in H^{3}(\mathbb{T}^2)$ and there exists a constant $a>0$ such that $\underset{x\in \mathbb{T}^2 }{\inf}\bar{\rho}_0>a$. 
		\item  The probability measure $\mathcal{L}(\xi)$ on $\mathbb{R}$  has a density with compact support and   $\int_{\mathbb{R}}x\mathcal{L}(\xi)(\mathd x)>0$. 
	\end{enumerate} 
\end{assumption}
By the assumption for the initial data $\mathcal{L}(X_0),$    the solution  $(  v_t)_{t\in[0,T]}$ to the  stochastic 2-dimensional Navier-Stokes equations \eqref{eqt:ns}  satisfies 
\begin{enumerate}
	\item It holds almost surely that  for all $t\in [0,T],$ \begin{align*}
		\underset{x\in \mathbb{T}^2 }{\sup}  v_t\leq \underset{x\in \mathbb{T}^2 }{\sup}  v_0,  \quad  \underset{x\in \mathbb{T}^2 }{\inf}   v_t\geq  \underset{x\in \mathbb{T}^2 }{\inf}   v_0,\quad\|  v_t\|_{L^2}\leq \|  v_0\|_{L^2},\quad \mathbb{P}-a.s..
	\end{align*} 
	\item  It holds that 
	\begin{align*}
		\mathbb{E} \int_0^T \|  v_t\|_{H^4}^2\mathd t <\infty ,\quad\mathbb{E}[\sup_{t\in [0,T ]}\|  v_t\|_{H^2}^2]< \infty.
	\end{align*}
	
\end{enumerate}

In the Section \ref{sec:relaen}, we will see that the regularity of  $(  v_t)_{t\in[0,T]}$ and boundedness of intensity  are  required to apply the relative entropy method.
\begin{assumption}\label{assumption2}
	For the vector fields $\{\sigma_k, k=1,\cdots,d\}$ on $\mathbb{T}^2,$ we assume that 
	$\{\sigma_k, k=1,\cdots,d\}, d \in \mathbb{N},$ are smooth and  divergence free.
\end{assumption}

We emphasize that  the divergence-free property of  $\{\sigma_k, k=1,\cdots,d\}$ is crucial to  well-posedness of the particle system \eqref{eqt:vortex} and is useful to obtain uniform estimates for the solutions  of  the stochastic 2-dimensional Navier-Stokes   equation \eqref{eqt:ns}.

 We restate the setting we finally work in, which is defined as follows:
\begin{assumption}\label{assumptiona}
	Suppose that there exist  independent 1-dimensional Brownian motions $\{W^k,k=1,\cdots, d\}, d\geq 1,$ and 2-dimensional Brownian motions $\{B_i,i=1,\cdots,N\}$ on torus $\mathbb{T}^2,$ on probability space $(\Omega,\mathcal{F},\mathbb{P})$ and a class of i.i.d $\mathbb{R} \times \mathbb{T}^2$ valued random variables $\{(\xi_i, X_i(0)), i=1,\cdots,N\}$ exist  on the same probability space $(\Omega,\mathcal{F},\mathbb{P})$ satisfying  $\mathcal{L}(X_i(0))= \mathcal{L}(X(0))$ and $\mathcal{L}(\xi_i)= \mathcal{L}(\xi).$  We also assume  mutual independence among $\{X_i(0),i=1,\cdots,N\},$ $\{\xi_i,i=1,\cdots,N\},$ $\{B_i,i=1,\cdots,N\},$ and $\{W^k,k=1,\cdots, d\}.$ We define $(\mathcal{F}_t)_{t\in [0,T]}$ as the normal filtration generated by
	1-dimensional Brownian motions $\{W^k,k=1,\cdots, d\},$ 2-dimensional Brownian motions $\{B_i,i=1,\cdots,N\}$ and $\{(\xi_i, X_i(0)), i=1,\cdots,N\}.$

\end{assumption}
It is easy to construct stochastic basis $(\Omega,\mathcal{F},(\mathcal{F}_t)_{t\in [0,T]},\mathbb{P})$ and Brownian motions satisfying Assumption \ref{assumptiona}.

\subsection{Function spaces}
We work mostly with Sobolev spaces $H^{k}$, and the norm of Sobolev space $H^{k}$, $k\in \mathbb{N}$, is defined as 
\begin{align*}
	\|f\|^2_{H^{k}}\assign \sum_{0\leq j\leq k}\|\nabla^{j} f \|^2_{L^2}.
\end{align*} 
We also use some results on Besov spaces. We use $(\Delta_{i})_{i\geq -1}$ to denote the Littlewood-Paley blocks for a dyadic partition of unity. Besov spaces  $B^{\alpha}_{p,q}$ 
on the torus with  $\alpha\in \mathbb{R}$ and $1\leq p,q\leq \infty$,  are defined as the completion of $C^{\infty}$ with respect to the norm
\begin{equation*}
	\|f\|_{B^{\alpha}_{p,q}}\assign \( \sum_{n\geq -1}\left(2^{n\alpha }\|\Delta_{n}f\|_{L^p}^q \right) \) ^{\frac{1}{q}}.
\end{equation*}

We remark that  $B^{\alpha}_{2,2}$ coincides with the Sobolev space $H^{\alpha}$, $\alpha\in \mathbb{R}$. We say $f\in C^\alpha$, $\alpha\in \mathbb{N}$, if $f$ is $\alpha$-times differentiable. For $\alpha \in \mathbb{R}\setminus \mathbb{N} $, we have $C^{\alpha}=B^{\alpha}_{\infty,\infty}$. We will often write $\|\cdot\|_{C^{\alpha}}$ instead of $\|\cdot\|_{B^{\alpha}_{\infty,\infty}}$. In the case $\alpha\in \mathbb{R}^+\setminus \mathbb{N}$, $C^{\alpha}$ coincides with the usual H\"older space. We use $C^{\infty}$ to denote the space of infinitely differentiable functions on $\mathbb{T}^2$.

We quote the following result about  Besov spaces.
\begin{lemma}
	\label{lemma triebel}	
	(i) For $\alpha, \beta \in \mathbb{R}$ and $p, p_1, p_2, q \in [1, \infty]$ where $\frac{1}{p} = \frac{1}{p_1} + \frac{1}{p_2}$, the bilinear map $(u, v) \mapsto u\cdot v$ can be extended to a continuous map from $B_{p_1, q}^{\alpha} \times B_{p_2, q}^{\beta}$ to $B_{p, q}^{\alpha \wedge \beta}$, provided that $\alpha + \beta > 0$. (cf. {\cite[Corollary 2]{mourrat2017global}}).
	
	(ii) (Duality) Consider $\alpha \in (0,1),$ $p,q\in[1,\infty],$ $p'$ and $q'$ as their respective conjugate exponents. The bilinear form $(u, v) \mapsto \<u,v\>=\int uv \dif x$ can be extended to a continuous map on $B^\alpha_{p,q}\times B^{-\alpha}_{p',q'}$ and satisfies the inequality $|\<u,v\>| \lesssim \|u\|_{B^\alpha_{p,q}}\|v\|_{B^{-\alpha}_{p',q'}}$ (cf. \cite[Proposition 7]{mourrat2017global}).
	
	(iii) Let $\alpha \leq \beta \in \mathbb{R},$ $q\in [1,\infty],$ $p_1 \geq p_2 \in [1,\infty]$ such that $\beta \geq \alpha + d\left(\frac{1}{p_2}-\frac{1}{p_1}\right)$. Then the embedding $B^{\beta}_{p_2,q}\hookrightarrow B^{\alpha}_{p_1,q}$ is continuous. (cf. \cite[Proposition 2]{mourrat2017global}).

\end{lemma}

\subsection{ Conditional probability}\label{conditional}

We prove the following result about  Fubini's theorem between deterministic integral and conditional expectation. 
 \begin{lemma}\label{lem:con3}
 	Under Assumption \ref{assumptiona},
 	for a  given   $(\mathcal{F}_t)_{t\in[0,T]}$ progressive measurable stochastic process $\{h_t,t\in[0,T]\}$ with $\mathbb{E}\int_{0}^{T}|h_{s}|\mathd s<\infty,$
 	we have	for each $t\in [0,T],$
 	\begin{align*}
 		\mathcal{L}(h_t|\mathcal{F}^W_T)&= \mathcal{L} (h_t|\mathcal{F}^W_t), \qquad 
 		\mathcal{L}(h_t|\mathcal{F}^{\xi^N,W}_T)= \mathcal{L} (h_t|\mathcal{F}^{\xi^N,W}_t), \quad a.s.,
 	\end{align*}
 	 and
 	\begin{align*}
 		&\mathbb{E}\bigg[\int_0^th_s \mathd s| \mathcal{F}^W_T\bigg]=\int_0^t\mathbb{E}[ h_s|\mathcal{F}^W_T]\mathd s=\int_0^t\mathbb{E}[ h_s|\mathcal{F}^W_s]\mathd s,
 		\\&\mathbb{E}\bigg[\int_0^th_s \mathd s| \mathcal{F}^{\xi^N,W}_T\bigg]=\int_0^t\mathbb{E}[ h_s|\mathcal{F}^{\xi^N,W}_T]\mathd s=\int_0^t\mathbb{E}[ h_s|\mathcal{F}^{\xi^N,W}_s]\mathd s, \quad a.s..
 	\end{align*}
 \end{lemma}
\begin{proof}
	We will  present the proofs for the outcomes pertaining to $\mathcal{F}^W_{\cdot}$. The results for $\mathcal{F}^{\xi^N,W}_{\cdot}$ can be derived using the same methodology.
	Using  the fact that $$W_k(t+\cdot)-W_k(t) \perp \mathcal{F}_t , \quad \forall k\in\{1,\cdots,d\},$$ we have $\mathcal{F}_t\perp \mathcal{F}^W_T|\mathcal{F}^W_t$ for $t\in [0,T],$  i.e. for every $\mathcal{F}_t$-measurable integrable random variable $\phi,$ it holds that $\mathbb{E}[\phi |\mathcal{F}^W_T]=\mathbb{E}[\phi | \mathcal{F}^W_t].$ We then have  $$\forall t\in [0,T],\quad  \mathcal{L}(h_t|\mathcal{F}^W_T)= \mathcal{L} (h_t|\mathcal{F}^W_t)\quad \mathbb{P}-a.s.$$  By the linearity and using a routine approximation,  it suffices to prove the desire result for elementary process $h_s=\varphi_rI_{(r,u]}(s),$ with $0\leq r\leq u\leq t$ and $\varphi_r$ is an $\mathcal{F}_r-$measurable integrable random variable. Using the property of conditional independence gives
	\begin{align*}
\mathbb{E}\bigg[\int_0^t\varphi_rI_{(r,u]}(s) \mathd s| \mathcal{F}^W_T\bigg]=\mathbb{E}[\varphi_r(u-r)|\mathcal{F}^W_T]=\int_{0}^{t}	\mathbb{E}[\varphi_rI_{(r,u]}(s)|\mathcal{F}^W_T] \mathd s=	\int_{0}^{t}	\mathbb{E}[\varphi_rI_{(r,u]}(s)|\mathcal{F}^W_s] \mathd s.
	\end{align*}
\end{proof}
We also recall the following lemma (\cite[Lemma B.1]{lacker2022superposition}) for stochastic integral. 
\begin{lemma}\label{lem:con1}
	Under Assumption \ref{assumptiona},
	for a  given   $(\mathcal{F}_t)_{t\in[0,T]}$ progressive measurable stochastic process $\{h_t,t\in[0,T]\}$ with $\mathbb{E}\int_{0}^{T}h_{s}^{2}\mathd s<\infty,$ we have
	
	\begin{align*}
		\mathbb{E}\bigg[\int_0^th_s \mathd B_s| \mathcal{F}^W_t\bigg]=0,\qquad \mathbb{E}\bigg[\int_0^th_s \mathd B_s| \mathcal{F}^{\xi^N,W}_t\bigg]=0,
	\end{align*}
and 
\begin{align*}
&\mathbb{E}\bigg[\int_0^th_s \mathd W_s| \mathcal{F}^W_t\bigg]=\int_0^t \mathbb{E}[h_s|\mathcal{F}^W_s]\mathd W_s,
\\&\mathbb{E}\bigg[\int_0^th_s \mathd W_s| \mathcal{F}^{\xi^N,W}_t\bigg]=\int_0^t \mathbb{E}[h_s|\mathcal{F}^{\xi^N,W}_s]\mathd W_s
,\quad t\in[0,T].
\end{align*}

\end{lemma}

\subsection{ Relative entropy, entropy and Fisher information functional}\label{entropy}

   We start this section by recalling the definitions of (relative)  entropy and Fisher information associated to any two probability measures $\rho^N$ and $\eta^N$ on $\mathbb{T}^{2N},N\geq 1$.  
   For convention, we also write $\rho^N$ and $\eta^N$ for their density functions whenever they exist. 
   Firstly, we recall the \textit{entropy functional} for $\rho^N$ defined by
	
	\begin{align*} 
			H(\rho^N)\assign
			\begin{cases}
					{\displaystyle \int_{\mathbb{T}^{2N}} \rho^N\log\rho^N\mathd x^N},    &\rho^N\in  L^1(\mathbb{T}^{2N});\vspace{5pt}
					 \\
					+\infty, & otherwise.
				\end{cases} 
	\end{align*} 
	The {\it relative entropy} $H(\rho^N|\eta^N )$\footnote{The definition of (normalized) relative entropy associated with any two probability measures, denoted as $\rho^N$ and $\eta^N$ (we maintain this notation for simplicity), on $\mathbb{R}^N\times\mathbb{T}^{2N}$, where $N\geq 1$, is entirely identical to the definition of relative entropy for two probability measures on $\mathbb{T}^{2N}$—with the sole modification being the replacement of the state space $\mathbb{T}^{2N}$ by $\mathbb{R}^N\times\mathbb{T}^{2N}.$} is defined as 
	\begin{align*} 
			H(\rho^N|\eta^N )\assign 
			\begin{cases}
			\displaystyle{	\int_{\mathbb{T}^{2N}} \log\frac{d\rho^N}{d\eta^N}\mathd \rho^N, }   &\rho^N\ll\eta^N; \vspace{5pt}\\
			+\infty, &otherwise.
				\end{cases} 
	\end{align*}  
 Here $\frac{d\rho^N}{d\eta^N}$ represents the Radon–Nikodym derivative of $\rho^N$ with respect to $\eta^N$.  
 In particular, when both $\rho^N$ and $\eta^N$ have densities, a simple calculation shows that
 \begin{align*}
 	H(\rho^N|\eta^N ) = \int_{\mathbb{T}^{2N}}\rho^N \log \rho^N\mathd x^N-\int_{\mathbb{T}^{2N}}\rho^N\log\eta^N \mathd x^N. 
 	\end{align*}
  Furthermore, the \textit{normalized relative entropy} for any two probability measures $\rho^N$ and $\eta^N$ on $\mathbb{T}^{2N}$ is defined as follows:
 \begin{align}\label{defrela}
 	H_N(\rho^N|\eta^N )\assign \frac{1}{N} H(\rho^N|\eta^N).
 	\end{align}
 Finally, we recall the \textit{Fisher information} of probability measures  $\rho^N$ on $\mathbb{T}^{2N}$  is defined by 
 \begin{align*}  
		I(\rho^N)\assign
		\begin{cases}
			\displaystyle{\int_{\mathbb{T}^{2N}}\frac{|\nabla\rho^N|^2}{\rho^N}\mathd x^N, }   &\rho^N\in W^{1,1}(\mathbb{T}^{2N}); \vspace{5pt} \\
		 +\infty, &otherwise.
		\end{cases} 
\end{align*} 
	Here $W^{1,1}(\mathbb{T}^{2N})$ denotes the usual Sobolev space 
	$W^{1,1}(\mathbb{T}^{2N})\assign \{f\in L^1(\mathbb{T}^{2N}): \nabla f\in L^1(\mathbb{T}^{2N})\}.$

 When Fisher information is finite, an important result from \cite[Lemma 3.3]{fournier2014propagation} states that the distance between different particles cannot be too close in the sense of the following Lemma \ref{fournier3.3}. The same arguments also works for $\mathbb{T}^2$ as in \cite[Lemma 3.3]{fournier2014propagation}.\footnote{The difference between Lemma \ref{fournier3.3} and Lemma 3.3 in \cite{fournier2014propagation} lies  in the space $\mathbb{T}^2$ of the probability measure $\rho$, which is  $\mathbb{R}^2$ in \cite{fournier2014propagation}.  }  

 \begin{lemma}\label{fournier3.3}
  For any $r\in(0,2)$ and $\beta>\frac{r}{2},$ there exists a constant $C_{r,\beta}>0$ depends on $r,\beta$ such that for any probability measure $\rho$ on $\mathbb{T}^2\times \mathbb{T}^2$
  with finite Fisher information,
	\begin{align*}\int_{\mathbb{T}^{2}}\frac{1}{|x_1-x_2|^{r}}\rho (\mathd x_1\mathd x_2)\leq C_{r,\beta}(I^{\beta}(\rho)+1).
		\end{align*}
 \end{lemma}
 To achieve the main theorem in this paper, we also need the following three auxiliary lemmas, which can be proved similarly as in the proofs for Lemma~1 and Theorems~3, 4 in   \cite{jabin2018quantitative}\footnote{The  distinction lies in that we take integration in the spaces $(\mathbb{R}\times\mathbb{T}^2)^{ \nocomma N}$, $(\mathbb{R}\times\mathbb{T}^2)^{ \nocomma N},$ and $\mathbb{R}\times\mathbb{T}^{2},$ while the state spaces are  $\mathbb{T}^{2N}$, $\mathbb{T}^{2N}$, and $\mathbb{T}^{2}$  in \cite{jabin2018quantitative}.}. 
\begin{lemma}\label{lemma jw0} 

 For  any  two probability densities $\rho_N,\bar{\rho}_N$ on $(\mathbb{R}\times\mathbb{T}^2)^{N},N\geq 1$ and any function $\phi \in L^{\infty} \big((\mathbb{R}\times\mathbb{T}^2)^{N} \big) $, one has that  for any constant $\eta>0,$
 $$\int_{(\mathbb{R}\times\mathbb{T}^2)^{N}}\phi \rho_N \mathd x^{N}\leq \frac{1}{\eta} \bigg(H_N(\rho_N|\bar{\rho}_N)+\frac{1}{N}\log \int_{(\mathbb{R}\times\mathbb{T}^2)^{N}}\bar{\rho}_N\exp\{N\eta\phi\}\mathd x^{N}\bigg).$$
\end{lemma}
\begin{lemma}
	\label{lemma jw1}For any
	probability density $\bar{\rho}$ on $\mathbb{R}\times\mathbb{T}^2$,
	and any $\phi (x, y) \in
	L^{\infty} ((\mathbb{R}\times\mathbb{T}^2)^{2 })$ with

	\begin{equation}
	 \| \phi \|_{L^{\infty}((\mathbb{R}\times\mathbb{T}^2)^{2 })} < \frac{1}{2e}, \nonumber
	\end{equation}
	Assume that $\phi$ satisfies the
	following cancellation

	 $$\int_{\mathbb{R}\times\mathbb{T}^2} \phi (y,z)  \bar{\rho} (z) \mathd \nocomma z = 0 \quad y \in \mathbb{R}\times\mathbb{T}^2,$$
Then it holds that
		\begin{equation}
		\sup_{N \geq 2 }\, 	\int_{(\mathbb{R}\times\mathbb{T}^2)^{ \nocomma N}} \bar \rho^{\otimes N} \exp \big(N\<\phi_{x_1}^{\otimes2},\mu_{x^N}\otimes\mu_{x^N}\>\big)\mathd x^N\leqslant C,\nonumber
	\end{equation}
	where 
	$x^N\assign(x_1,..,x_N)\in (\mathbb{R}\times\mathbb{T}^2)^{ \nocomma N}, \phi_{x_1}^{\otimes2}(z_1,z_2)\assign \phi(x_1,z_1)\phi(x_1,z_2),\mu_{x^N}\assign \frac{1}{N}\sum_{i=1}^{N}\delta_{x_i}$ and $C$ is a constant independent of $\phi.$
\end{lemma}

\begin{lemma}
	\label{lemma jw2}For any
	probability measure $\bar{\rho}$ on $\mathbb{R}\times\mathbb{T}^2$,
	and any $\phi (x, y) \in
	L^{\infty} ((\mathbb{R}\times\mathbb{T}^2)^{2 })$ with
	\begin{equation}
	 \tilde{C}\| \phi \|_{L^{\infty} \big((\mathbb{R}\times\mathbb{T}^2)^{2} \big)}^2< \frac{1}{2}, \nonumber
	\end{equation}
	where $ \tilde{C}$ is a universal constant  independent of $\bar{\rho}$. Assume that $\phi$ satisfies the
	following cancellations

  $$\int_{\mathbb{R}\times\mathbb{T}^2} \phi (z,x)  \bar{\rho} (z) \mathd \nocomma z = 0,$$ and $$\int_{\mathbb{R}\times\mathbb{T}^2} \phi (x,z)  \bar{\rho} (z) \mathd \nocomma z = 0, \quad x\in \mathbb{R}\times\mathbb{T}^2,$$
	Then  it holds that
		\begin{equation}
		\sup_{N \geq 2 }\, 	\int_{(\mathbb{R}\times\mathbb{T}^2)^{\nocomma N}} \bar \rho^{\otimes N} \exp \big(N \<\phi,\mu_{x^N}\otimes \mu_{x^N}\>\big) \mathd x^N \leq C,\nonumber
	\end{equation}
	where 
	$x^N\assign(x_1,\cdots, x_N)\in (\mathbb{R}\times\mathbb{T}^2)^{N},  \mu_{x^N}=\frac{1}{N}\sum_{i=1}^{N}\delta_{x_i}$ and $C$ is a constant independent of $\phi.$
\end{lemma}

\subsection{Definitions of solutions  } \label{solutiondef}

We now present several distinct definitions for solutions to the stochastic 2-dimensional Navier-Stokes equation \eqref{eqt:ns},  the stochasitc vortex model \eqref{eqt:vortex} and the  conditional Mckean-Vlasov equation \eqref{eqt:ncopy}. 

\begin{definition}\label{defnsw}
	A probabilistically weak solution $(\Omega,\mathcal{G},(\mathcal{G}_t)_{t\in [0,T]},\mathbb{P},W,(  v_t)_{t\in [0,T]})$   to \eqref{eqt:ns} with initial value $  v_0\in H^{3}(\mathbb{T}^2)$ is defined as  a stochastic basis $(\Omega,\mathcal{G},(\mathcal{G}_t)_{t\in [0,T]},\mathbb{P})$ supporting   independent $(\mathcal{G}_t)_{t\in [0,T]}$-Brownian motions $\{W^{k},k=1,\cdots,d\}$ (denoted by $W$) and a continuous $L^2(\mathbb{T}^2)$ valued $(\mathcal{G}_t)_{t\in [0,T]}$-adapted stochastic process $(  v_t)_{t\in [0,T]}$  such that 
	\begin{enumerate}
		\item For all $t\in [0,T]$, \begin{align}\label{regu-1}
			\|  v_t\|_{L^{\infty}}\leq \|  v_0\|_{L^{\infty}},    \quad\underset{x\in \mathbb{T}^2 }{ess\inf}   v_t\geq  \underset{x\in \mathbb{T}^2 }{ess\inf}   v_0,  \quad\|  v_t\|_{L^2}\leq \|  v_0\|_{L^2},\quad \mathbb{P}-a.s..
		\end{align}  
		\item  It holds that 
		\begin{align}\label{regu-2}
				&\mathbb{E} \int_0^T \|  v_t\|_{H^4}^2\mathd t <\infty , \quad \mathbb{E}[\sup_{t\in [0,T ]}\|  v_t\|_{H^2}^2]< \infty.
		\end{align}
		\item   For all $\varphi\in C^{\infty}(\mathbb{T}^2)$,  it holds almost surely that for all $t\in [0,T],$ 
		\begin{align*}
			\left\langle  \varphi,  v_t\right\rangle  =  \left\langle \varphi,   v_0\right\rangle +\int_0^t \left\langle \Delta\varphi,   v_s\right\rangle \mathd s +\int_0^t\left\langle  \nabla \varphi,  K*  v_s   v_s\right\rangle \mathd s+ \sum_{k=1}^{d}\int_0^t  \left\langle \nabla \varphi,    v_s\sigma_k\right\rangle\circ \mathd W_s^{k} .
		\end{align*}  
	\end{enumerate}

\end{definition}
\begin{definition}\label{defnss}
	Given   independent 1-dimensional Brownian motions $\{W^{k},k=1,\cdots,d\}$  (denoted by $W$) on probability space $(\Omega,\mathcal{G},\mathbb{P}),$  we say that 
	$(  v_t)_{t\in [0,T]}$  is a probabilistically strong solution  to \eqref{eqt:ns} with initial value $  v_0\in H^{3}(\mathbb{T}^2)$   if $(\Omega,\mathcal{G},(\mathcal{G}^{W}_t)_{t\in [0,T]},\mathbb{P},(  v_t)_{t\in [0,T]})$  is a  probabilistically weak solution to \eqref{eqt:ns} with initial value $  v_0\in H^{3}(\mathbb{T}^2),$ where $(\mathcal{G}^{W}_t)_{t\in [0,T]}$ is the normal filtration generated by Brownian motions $\{W^{k},k=1,\cdots,d\}.$  
\end{definition}
\begin{definition}\label{unins}
	We say that pathwise uniqueness holds for \eqref{eqt:ns} if two probabilistically weak solutions $(  v_t)_{t\in [0,T]}$ and $(\tilde{  v}_t)_{t\in [0,T]}$  on the same probability space $(\Omega,\mathcal{G},(\mathcal{G}_t)_{t\in [0,T]},\mathbb{P})$ with the same noise $\{W^{k},k=1,\cdots,d\}$ and same initial data $v_0$ implies $$\mathbb{P}\(\|  v_t-\tilde{  v}_t\|_{L^2}=0, \forall t\in [0,T] \)=1.$$
\end{definition}
\begin{definition}\label{defpart}
Given  independent 1-dimensional Brownian motions $\{W^{k},k=1,\cdots,d\}$ and 2-dimensional Brownian motions $\{B_i,i=1,\cdots,N\}$ on  probability space $(\Omega,\mathcal{F},\mathbb{P}),$  a  probabilistically strong solution  to \eqref{eqt:vortex} 
is defined as a continuous $(\mathcal{F}_t)_{t\in [0,T]}$-adapted stochastic process $(X^N(t))_{t\in [0,T]},$  where $(\mathcal{F}_t)_{t\in [0,T]}$ is the normal filtration generated by $\mathbb{R} \times \mathbb{T}^2$ valued random variables $\{(\xi_i, X_i(0)), i=1,\cdots,N\},$ Brownian motions $\{W^{k},k=1,\cdots,d\}$ and $\{B_i,i=1,\cdots,N\},$ 
such that   it holds almost surely that for every $0<t\leq T,$ \begin{equation*}
		X^N_i(t)=X^N_i(0)+\frac{1}{N}\sum_{j\neq i}\int_0^t\xi_jK(X^N_i(s)-X^N_j(s))\mathd s+\sqrt{2}B_i(t)+\sum_{k=1}^{d}\int_0^t\sigma_k(X^N_i(s))  \circ \mathd W^k_s \quad \forall i=1,\cdots, N.
	\end{equation*}
The definition for probabilistically strong solution $(Y_i(t))_{t\in[0,T]}$  to \eqref{eqt:ncopy} is given similarly.
\end{definition}

\begin{definition}\label{unipart}
	Given  independent 1-dimensional Brownian motions $\{W^{k},k=1,\cdots,d\}$ and 2-dimensional Brownian motions $\{B_i,i=1,\cdots,N\}$ on  stochastic basis $(\Omega,\mathcal{F},(\mathcal{F}_t)_{t\in [0,T]},\mathbb{P})$, where $(\mathcal{F}_t)_{t\in [0,T]}$ is the normal filtration generated by Brownian motions $\{W^{k},k=1,\cdots,d\}$ and $\{B_i,i=1,\cdots,N\},$ we say that pathwise uniqueness holds for \eqref{eqt:vortex} 
	if two probabilistically strong solutions  $(X^N(t))_{t\in [0,T]},$ $(\tilde{X}^N(t))_{t\in [0,T]}$  to \eqref{eqt:vortex} with $\mathbb{P}(\mid X^N_0-\tilde{X}^N_0\mid=0)=1$  
	 implies
	$$\mathbb{P}\(\mid X^N(t)-\tilde{X}^N(t)\mid=0, \forall t\in [0,T] \)=1,$$
	Pathwise uniqueness for \eqref{eqt:ncopy} is defined similarly.

\end{definition}
\section{Conditional Mckean-Vlasov equation and Stochastic Fokker-Planck equation}\label{sec:conditionalmv}
This section aims to derive  well-posedness of the  conditional McKean-Vlasov equation \eqref{eqt:ncopy} and the mean field limit equation \eqref{eqt:ns} and to demonstrate their interrelationship. More precisely, it is shown that the probabilistically strong solution $(  v_t)_{t\in [0,T]}$ to the stochastic 2-dimensional Navier-Stokes equation  \eqref{eqt:ns} is equivalent to the weighted distribution of some identical distributed particles $Y_i, i=1,\cdots,N,$ which are the solutions to \eqref{eqt:ncopy}, i.e. 
\begin{align}\label{weighted}
\mathbb{P}(\|  v_t-\tilde{  v}_t\|_{H^{-(2+  \eps)}}=0, \forall t\in [0,T])=1,\quad \forall   \eps\in (\frac{1}{2},1).
\end{align}
 Here $$\tilde{  v}_t(\mathd x)\assign\int_{\mathbb{R}} z \mathcal{L}((Y_i(t),\xi_i)| \mathcal{F}^{W}_T)(\mathd x\mathd z)=\mathbb{E}[\xi_1]\mathcal{L}((Y_i(t))| \mathcal{F}^{W}_T)(\mathd x),$$ and $\xi_i,i=1,\cdots,N,$ denote the intensities in the particle system \eqref{eqt:vortex} with the same law $\mathcal{L}(\xi).$ The final equality is a result of the independence $\xi_i \perp\mathcal{F}^W_T$ introduced in  Assumption \ref{assumptiona}.

We begin by stating well-posedness of  \eqref{eqt:ns} and analyzing the regularity of the solutions. Subsequently, we present a result on the uniqueness of the linearized stochastic 2-dimensional Navier-Stokes equation \eqref{ptsobol3} in the final part of Section \ref{sec2dns}. This uniqueness result plays a pivotal role in establishing the connection between the conditional McKean-Vlasov equation \eqref{eqt:ncopy} and the stochastic 2-dimensional Navier-Stokes equation \eqref{eqt:ns}. In Section \ref{sec3.1}, we demonstrate  well-posedness of the conditional McKean-Vlasov equation \eqref{eqt:ncopy}. Through It\^o's formula and the results about conditional expectation in Section \ref{conditional}, we find that the conditional distribution  $$\mathcal{L}((Y_i(t))| \mathcal{F}^{W}_T)(\mathd x),i=1,\cdots,N,$$ where $Y_i, i=1,\cdots,N,$ are the solutions to \eqref{eqt:ncopy}, satisfying the linearized stochastic 2-dimensional Navier-Stokes  equation. Building upon the uniqueness of the linearized stochastic 2-dimensional Navier-Stokes equation (see \eqref{ptsobol3} below), we then finalize the connection expressed in \eqref{weighted}.

\subsection{Well-posedness of SPDEs}\label{sec2dns}
To evolve the relative entropy between particle system \eqref{eqt:vortex} and the mean field limit equation \eqref{eqt:ns}, it is essential to ensure  suitable regularity of the solution to \eqref{eqt:ns}, as stated in Theorem \ref{thm:spde} and Lemma \ref{lem:reg}. To focus on the mean-field convergence, we will omit the proof of Theorem \ref{thm:spde} through  the standard martingale argument and present only the derivation of uniform estimates in Section \ref{sec:ns}. The proof for Lemma \ref{lem:reg} is also presented in Section \ref{sec:ns}. 

\begin{theorem}\label{thm:spde}
	Given  independent Brownian motions  $\{W^{k},k=1,\cdots ,d\}$ on  probability space $(\Omega,\mathcal{F},\mathbb{P}),$
	for each $  v_0\in H^{3}(\mathbb{T}^2)$,	there exists a unique probabilistically strong solution $(  v_t)_{t\in [0,T]}$ to \eqref{eqt:ns} in the sense of Definitions \ref{defnss} and \ref{unins}. 
\end{theorem}

\begin{lemma}\label{lem:reg}
	 For a probabilistically weak solution $(\Omega,\mathcal{G},(\mathcal{G}_t)_{t\in [0,T]},\mathbb{P},W,(  v_t)_{t\in [0,T]})$ to  \eqref{eqt:ns} in the sense of Definition \ref{defnsw}, it holds almost surely that
	 $  v\in C([0,T], H^{2}(\mathbb{T}^2)).$
	 In particular,  it holds almost surely that  for all $t\in [0,T],$
	  \begin{align}\label{eq:re1}
	 	\underset{x\in \mathbb{T}^2 }{\sup}  v_t\leq \underset{x\in \mathbb{T}^2 }{\sup}  v_0,  \quad  \underset{x\in \mathbb{T}^2 }{\inf}   v_t\geq  \underset{x\in \mathbb{T}^2 }{\inf}   v_0,
	 \end{align}
	and  it holds almost surely that  for all $t\in [0,T],x \in \mathbb{T}^2,$ 
	 \begin{align}\label{eq:re2}
	 	  v_t(x) -   v_0(x) =\int_0^t  \Delta  v_s(x)\mathd s -\int_0^t  K*  v_s(x) \cdot\nabla   v_s(x)\mathd s- \sum_{k=1}^{d}\int_0^t  \sigma_k(x)\cdot \nabla   v_s(x)\circ \mathd W_s^{k}.  
	 \end{align}
\end{lemma}
Our goal in this section is to establish a connection between the stochastic 2-dimensional Navier-Stokes  equation \eqref{eqt:ns} and the conditional Mckean-Vlasov equation \eqref{eqt:ncopy}. To accomplish this, we first define  the random measures on $\mathbb{T}^2,$   \begin{align}\label{barrhot}
	\bar{\rho}_t(\mathd x)\assign \frac{1}{\mathbb{E}[\xi_1]} v_t(\mathd x),
\end{align} for each $t\in[0,T]$
on the stochastic basis $(\Omega,\mathcal{F},(\mathcal{F}_t)_{t\in [0,T]},\mathbb{P})$ introduced in Assumption \ref{assumptiona}, 
where $\{\xi_i, i=1,\cdots,N\}$  denote the intensities in the weakly interacting system \eqref{eqt:vortex} with the same distribution $\mathcal{L}(\xi)$ introduced in   Assumption \ref{initial}  and $( v_t)_{t\in [0,T]}$ is  the  unique probabilistically strong solution  to  \eqref{eqt:ns} obtained in Theorem \ref{thm:spde} with the initial data $ v_0\assign\mathbb{E}[\xi_1]\bar{\rho_0}\in H^{3}(\mathbb{T}^2).$ Here $\bar{\rho_0}$ is the density of $\mathcal{L}(X(0))$ given in Assumption \ref{initial}.

 We then find that $(\bar{\rho}_t)_{t\in[0,T]}$ solve the linearized stochastic 2-dimensional Navier-Stokes  equation, having the same regularity as $( v_t)_{t\in [0,T]}.$ More precisely, $(\bar{\rho}_t)_{t\in[0,T]}$ is a continuous $L^2(\mathbb{T}^2)$ valued $(\mathcal{F}_t)_{t\in [0,T]}$ adapted stochastic process such that 
\begin{enumerate}
	\item It holds almost surely that for all $t\in [0,T]$, 
	\begin{align}\label{ptsobol1}
		\underset{x\in \mathbb{T}^2 }{\sup}\bar{\rho}_t\leq \underset{x\in \mathbb{T}^2 }{\sup}\bar{\rho}_0,\quad \underset{x\in \mathbb{T}^2 }{\inf} \bar{\rho}_t\geq  \underset{x\in \mathbb{T}^2 }{\inf} \bar{\rho}_0,\quad \|\bar{\rho}_t\|_{L^2}\leq \|\bar{\rho}_0\|_{L^2}.
	\end{align} 
	\item  It holds that 
	\begin{align}\label{ptsobol2}
		\mathbb{E} \int_0^T \|\bar{\rho}_t\|_{H^4}^2\mathd t <\infty,\quad \mathbb{E}[\sup_{t\in [0,T ]}\|\bar{\rho}_t\|_{H^2}^2]< \infty.
	\end{align}
	\item   For all $\varphi\in C^{\infty}(\mathbb{T}^2)$,  it holds almost surely that for every $t\in [0,T],$
	\begin{align}\label{ptsobol3}
		\left\langle  \varphi,\bar{\rho}_t\right\rangle  =  \left\langle \varphi, \bar{\rho}_0\right\rangle +\int_0^t \left\langle \Delta\varphi, \bar{\rho}_s\right\rangle \mathd s +\int_0^t\left\langle  \nabla \varphi,  K*  v_s \bar{\rho}_s\right\rangle \mathd s+ \sum_{k=1}^{d}\int_0^t  \left\langle \nabla \varphi,  \bar{\rho}_s\sigma_k\right\rangle\circ \mathd W_s^{k} .
	\end{align}  
\end{enumerate} 
In the subsequent Section \ref{sec3.1},  we will observe that the conditional distribution  $$\mathcal{L}((Y_i(t))| \mathcal{F}^{W}_T)(\mathd x),\quad i=1,\cdots,N,$$ where $Y_i, i=1,\cdots,N$ are the solutions to \eqref{eqt:ncopy}, also satisfying the linearized stochastic 2-dimensional Navier-Stokes  equation \eqref{ptsobol3}.
To establish the  connection  \eqref{weighted}, we will then rely on the following uniqueness of the linearized stochastic 2-dimensional Navier-Stokes equation \eqref{ptsobol3}.
 
\begin{lemma}\label{lem:linearunique}	Given $(\bar{\rho}_t)_{t\in [0,T]}$ as above, for each $\tilde{f}_0(x)\assign\bar{\rho}_0(x),$   if there exists a solution $\tilde{f}$ to \eqref{ptsobol3}  satisfying the following conditions for some $\eps\in (0,1)$
	\begin{enumerate}
		\item $\tilde{f}$ is a continuous  $H^{-(1+\eps)}$-valued $(\mathcal{F}_t)_{t\in [0,T]}$-adapted process. 
		\item There exists a positive deterministic constant $C$ such that $
		\sup_{t\in [0,T]}  \|\tilde{f}_t\|_{H^{-(1+\eps)}}^2 <C$ almost surely.
		\item For all $\varphi\in C^{\infty}(\mathbb{T}^2),$ it holds almost surely  that for every $t\in [0,T],$
		\begin{align}\label{1}
			\left\langle  \varphi, \tilde{f}_t\right\rangle  =  \left\langle \varphi, \tilde{f}_0\right\rangle +\int_0^t \left\langle \Delta\varphi, \tilde{f}_s\right\rangle \mathd s +\int_0^t\left\langle  \nabla \varphi,  K*  v_s \tilde{f}_s\right\rangle \mathd s+ \sum_{k=1}^{d}\int_0^t  \left\langle \nabla \varphi,  \tilde{f}_s\sigma_k\right\rangle\circ \mathd W_s^{k} .
		\end{align}  
	\end{enumerate}
	Then $(\tilde{f}_t)_{t\in [0,T]}$ coincides with $(\bar{\rho}_t)_{t\in [0,T]}$  almost surely, i.e. $\mathbb{P}(\|\bar{\rho}_t-\tilde{f}_t\|_{H^{-(2+  \eps)}}=0, \forall t\in [0,T])=1.$
\end{lemma}
\begin{proof}Let  $\bar{f}\assign \tilde{f}-\bar{\rho}$,  therefore  $\bar{f}$ solves \eqref{1} in  the distributional sense.
	
	Now we evolve $\|\tilde{f}\|_{H^{-(2+\eps)}},$ for $\eps\in (0,1)$.  
	More precisely, we evolve $|\<\bar{f}_t,e_{m}\>|^2$ for each $e_m$ of  the Fourier basis $\{e_m\assign e^{imx}, m\in \mathbb{Z}^2\}$,
	then sum up  $ (1+|m|^2)^{-2-\eps}  |\<\bar{f}_t,e_{m}\>|^2$ over   $m\in \mathbb{Z}^2$.
	By It\^o's formula, we have
	\begin{align}
		|\<\bar{f}_t,e_{m}\>|^2=\sum_{j=1}^{5}\int_{0}^{t}J_{j}^m(s)\mathd s +\sum_{j=6}^{7}\int_{0}^{t}J_{j}^m(s)\mathd W_s^{k}, \label{3.1-ito}
	\end{align}
where 
\begin{align*}
	J_1^m(s)=& \<\bar{f}_s,e_{-m}\>\< \bar{f}_s,\Delta e_{m}\>  -  \<\bar{f}_s,e_{-m}\>\< K*  v_s \bar{f}_s,\nabla e_{m}\> ,\nonumber
	\\J_2^m(s)=&\<\bar{f}_s,e_{m}\>\< \bar{f}_s,\Delta e_{-m}\> ,
	- \<\bar{f}_s,e_{m}\>\< K*  v_s \bar{f}_s,\nabla e_{-m}\> , \nonumber
	\\J_3^m(s)=&\sum_{k=1}^{d} | {\<\sigma_k\cdot \nabla\bar{f}_s,e_{m}\>} |^2, \nonumber
	\\J_4^m(s)=&\frac{1}{2}\sum_{k=1}^{d}\<\bar{f}_s,e_{-m}\> \left\langle \sigma_k \cdot \nabla \(\sigma_k \cdot \nabla \bar{f}_s\), e_m\right\rangle , \nonumber
	\\J_5^m(s)=&\frac{1}{2}\sum_{k=1}^{d}\<\bar{f}_s,e_{m}\> \left\langle \sigma_k \cdot \nabla \(\sigma_k \cdot \nabla \bar{f}_s\), e_{-m}\right\rangle , \nonumber
	\\J_6^m(s)=&-\sum_{k=1}^{d}\<\bar{f}_s,e_{-m}\> \left\langle \sigma_k \cdot \nabla \bar{f}_s, e_m \right\rangle , \nonumber
	\\J_7^m(s)=&-\sum_{k=1}^{d}\<\bar{f}_s,e_{m}\> \left\langle \sigma_k \cdot \nabla \bar{f}_s, e_{-m} \right\rangle. \nonumber
\end{align*}

Here we expressed the Stratonovich integral into It\^o's form.  $J_4^m$ and $J_5^m$ represent the terms derived from the transformation of the Stratonovich integral.

We then find 

	 \begin{align*}
	 	&\sum_{m\in \mathbb{Z}^2}\frac{1}{(1+|m|^2)^{2+\eps}}J_1^m(s)
			\\&=\sum_{m\in \mathbb{Z}^2}\frac{1}{(1+|m|^2)^{2+\eps}} \<\bar{f}_s,e_{-m}\>\<\bar{f}_s,\Delta e_{m}\>- \sum_{m\in \mathbb{Z}^2}\frac{1}{(1+|m|^2)^{2+\eps}} \<\bar{f}_s,e_{-m}\>\< K*  v_s \bar{f}_s,\nabla e_{m}\>
			\\&=-\sum_{m\in \mathbb{Z}^2}\frac{|m|^2}{(1+|m|^2)^{2+\eps}} |\<\bar{f}_s,e_{m}\>|^2- \sum_{m\in \mathbb{Z}^2}\frac{im}{(1+|m|^2)^{2+\eps}} \<\bar{f}_t,e_{-m}\>\< K*  v_s \bar{f}_s,e_{m}\>.
		\end{align*}
	Furthermore, we deduce from the Lemma \ref{lemma triebel}-(3) that
	\begin{align*}
		\|K*  v\|_{B^{2+2\eps}_{\infty,2}}= \|\nabla^{\perp} (-\Delta)^{-1}  v\|_{B^{2+2\eps}_{\infty,2}}
		\lesssim\|  v\|_{H^4}.
	\end{align*}
Then we have 
 \begin{align}\label{unilineaK}
 		\|K*  v\bar{f}\|_{H^{-(2+\eps)}}\lesssim& \|K*  v\|_{B^{2+2\eps}_{\infty,2}}\|\bar{f}\|_{H^{-(2+\eps)}}\lesssim \|  v\|_{H^4}\|\bar{f}\|_{H^{-(2+\eps)}}.
 \end{align}
	by the Lemma \ref{lemma triebel}-(1).
	Therefore, we  obtain
	\begin{align}
	&\sum_{m\in \mathbb{Z}^2}\frac{1}{(1+|m|^2)^{2+\eps}}J_1^m(s)  \nonumber
	\\\leq & -\sum_{m\in \mathbb{Z}^2}\frac{|m|^2}{(1+|m|^2)^{2+\eps}} |\<\bar{f}_s,e_{m}\>|^2+\sum_{m\in \mathbb{Z}^2}\frac{1 }{2(1+|m|^2)^{2+\eps}}\left( |m|^2 |\<\bar{f}_t,e_{-m}\>|^2+  |\< K*  v_s \bar{f}_s,e_{m}\>|^2\right)  \nonumber
		\\\leq & -\frac{1}{2}\|\bar{f}_s\|_{H^{-(1+\eps)}}^2 +\frac{1}{2}\|\bar{f}_s\|_{H^{-(2+\eps)}}^2 
	+	\frac{1}{2}\|K*  v_s\bar{f}_s\|_{H^{-(2+\eps)}}^2 \nonumber
	\\\leq & -\frac{1}{2}\|\bar{f}_s\|_{H^{-(1+\eps)}}^2  
	+C\(1+\|  v_s\|_{H^4}\)\|\bar{f}\|_{H^{-(2+\eps)}}, \nonumber
	\end{align}
where we use Cauchy-Schwarz inequality to get the first inequality and  \eqref{unilineaK} to get the last inequality. Here $C$ is a positive universal constant.

Similarly,  we have 
\begin{align*}
	\sum_{m\in \mathbb{Z}^2}\frac{1}{(1+|m|^2)^{2+\eps}}J_2^m(s)	\leq  -\frac{1}{2}\|\bar{f}_s\|_{H^{-(1+\eps)}}^2 + C  \(1+	\|  v_s\|_{H^4}\)\|\bar{f}\|_{H^{-(2+\eps)}}.
\end{align*}
We then have 
\begin{align}
	\sum_{m\in \mathbb{Z}^2}\frac{1}{(1+|m|^2)^{2+\eps}}(J_1^m(s)+J_2^m(s))\leq-\|\bar{f}_s\|_{H^{-(1+\eps)}}^2 + C  \(1+	\|  v_s\|_{H^4}\)\|\bar{f}\|_{H^{-(2+\eps)}},\label{3.1-est1}
\end{align}
where $C$ is a positive universal constant.

	On the other hand, we find that 
	\begin{align*}
	\sum_{m\in \mathbb{Z}^2}\frac{1}{(1+|m|^2)^{2+\eps}}J_3^m(s)=& \sum_{k=1}^{d}\sum_{m\in \mathbb{Z}^2} \frac{1}{(1+|m|^2)^{2+\eps}}    | {\<\sigma_k\cdot \nabla\bar{f}_s,e_{m}\>} |^2
		\\=&\sum_{k=1}^{d} \|\sigma_k \cdot \nabla \bar{f}_s\|_{H^{-(2+\eps)}}^2.
	\end{align*}
	Furthermore,  
	\begin{align*}
	\sum_{m\in \mathbb{Z}^2}\frac{1}{(1+|m|^2)^{2+\eps}}J_4^m(s)=&\frac{1}{2}\sum_{k=1}^{d}\sum_{m\in \mathbb{Z}^2} \frac{1}{(1+|m|^2)^{2+\eps}}    \<\bar{f}_s,e_{-m}\> \left\langle \sigma_k \cdot \nabla \(\sigma_k \cdot \nabla \bar{f}_s\), e_m\right\rangle
	\\=&\frac{1}{2}\sum_{k=1}^{d} \left\langle \sigma_k \cdot \nabla \(\sigma_k \cdot \nabla \bar{f}_s\),\bar{f}_s \right\rangle_{H^{-(2+\eps)}}.
	\end{align*}
Similarly, we have
	\begin{align*}
	\sum_{m\in \mathbb{Z}^2}\frac{1}{(1+|m|^2)^{2+\eps}}J_5^m(s)=&\frac{1}{2}\sum_{k=1}^{d} \left\langle \sigma_k \cdot \nabla \(\sigma_k \cdot \nabla \bar{f}_s\),\bar{f}_s \right\rangle_{H^{-(2+\eps)}}.
\end{align*}
Using  the estimate (2.4) in	\cite[Lemma 2.3 ]{krylov2015filtering}   (with $n$ and $\sigma(t,x,  v)$ in	\cite[Lemma 2.3 ]{krylov2015filtering}   defined as $-(2+\eps)$ and $\sigma_k(x)$ respectively),  we have 
	\begin{align}
	&\sum_{m\in \mathbb{Z}^2}\frac{1}{(1+|m|^2)^{2+\eps}}\(J_3^m(s)+J_4^m(s)+J_5^m(s)\)\\=&\sum_{k=1}^{d} \left( 	\|\sigma_k \cdot \nabla \bar{f}_s\|_{H^{-(2+\eps)}}^2+ \left\langle \sigma_k \cdot \nabla \(\sigma_k \cdot \nabla \bar{f}_s\),\bar{f}_s \right\rangle_{H^{-(2+\eps)}}  \right) \nonumber
		 \\ \leq &C
		 \|\bar{f}_s\|_{H^{-(2+\eps)}}^2, \label{est345}
	\end{align}
where $C$   depends on $\{\sigma_k, k=1,\cdots ,d\}$. 

By Cauchy- Schwarz inequality, we have
\begin{align*}
	&\int_{0}^{T}\bigg(\sum_{m\in \mathbb{Z}^2}\frac{1}{(1+|m|^2)^{2+\eps}}J_6^m(s)\bigg)^2\mathd s\\=&\int_{0}^{T}\bigg(\sum_{k=1}^{d}\sum_{m\in \mathbb{Z}^2}\frac{1}{(1+|m|^2)^{2+\eps}}\<\bar{f}_s,e_{-m}\> \left\langle \sigma_k \cdot \nabla \bar{f}_s, e_m \right\rangle \bigg)^2\mathd s\\\lesssim&\int_{0}^{T}\biggl(\sum_{k=1}^{d}\sum_{m\in \mathbb{Z}^2}\frac{1}{(1+|m|^2)^{2+\eps}}\<\bar{f}_s,e_{-m}\> ^2\biggr)\biggl(\sum_{k=1}^{d}\sum_{m\in \mathbb{Z}^2}\frac{1}{(1+|m|^2)^{2+\eps}}\left\langle \sigma_k \cdot \nabla \bar{f}_s, e_m \right\rangle^2\biggr)\mathd s
	\\\lesssim&\int_{0}^{T}	\|\bar{f}_s\|_{H^{-(2+\eps)}}^2\|\sigma_k\cdot\nabla\bar{f}_s\|_{H^{-(2+\eps)}}^2\mathd s
	\\ \lesssim& \sup_{t\in [0,T]}  \|\bar{f}_t\|_{H^{-(1+\eps)}}^4.
\end{align*}
Similarly, we have 
\begin{align*}
	&\int_{0}^{T}\bigg(\sum_{m\in \mathbb{Z}^2}\frac{1}{(1+|m|^2)^{2+\eps}}J_6^m(s)\bigg)^2 \mathd s\lesssim \sup_{t\in [0,T]}  \|\bar{f}_t\|_{H^{-(1+\eps)}}^4.
\end{align*}
We then set $M_t \assign \int_{0}^{t}\sum_{m\in \mathbb{Z}^2}\frac{1}{(1+|m|^2)^{2+\eps}}\(J_6^m(s)+J_7^m(s)\)\mathd W^k_s,$ which is a continuous martingale.

Therefore, by merging  \eqref{3.1-ito}, \eqref{3.1-est1}, and \eqref{est345},  we arrive at 
\begin{align*}
	\|\bar{f}_t\|_{H^{-(2+\eps)}}^2\lesssim \int_0^t \(1+\|  v_s\|_{H^{4}}^2\)\|\bar{f}_s\|_{H^{-(2+\eps)}}^2\mathd s + M_t.
\end{align*}
 
Then   applying  the stochastic Gronwall's inequality (see \cite[Corollary 5.4]{geiss2021sharp} 
leads to the conclusion that $\sup_{t\in [0,T]}\|\bar{f}\|_{H^{-(2+\eps)}}=0$ almost surely. The result then follows. 
\end{proof}

\subsection{Conditional Mckean-Vlasov equation}\label{sec3.1}

In this section, we will consider the following  stochastic differential equations,  with the same initial value $X_i(0)$ as the weakly interacting system \eqref{eqt:vortex}, 
\begin{equation}\label{ncopy}
	Y_i(t)=X_i(0)+\int_0^tK*  v_s(Y_i(s))\mathd s +\sqrt{2}B_i(t)+\sum_{k=1}^{d}\int_0^t\sigma_k(Y_i(s))\circ\mathd W^k_s,
\end{equation} 
where $ i\in\{1,\cdots, N\}$ and $(  v_t)_{t\in [0,T]}$ is the probabilistically strong solution to the stochastic 2-dimensional Navier-Stokes  equation \eqref{eqt:ns} with the initial data \eqref{initial data}. 
Notably,  \eqref{ncopy} is a stochastic differential equation featuring the drift term with random Lipschitz coefficient depending on time $t.$
We will use  stopping times and the regularity of $  v$, i.e. $\mathbb{E} \int_0^T \|  v_t\|_{H^4}^2\mathd t <\infty $, to obtain well-posedness of \eqref{ncopy} through a standard contraction argument.
\begin{lemma}\label{lem:uniquesde}
For each $ i\in\{1,\cdots, N\}$,  given the independent 1-dimensional Brownian motions $\{W^{k},k=1,\cdots, d\}$  and 2-dimensional Brownian motions $\{B_i, i=1,\cdots,N\}$ on  probability space $(\Omega,\mathcal{F},(\mathcal{F}_t)_{t\in [0,T]},\mathbb{P})$, there exists a unique probabilistically strong solution $Y_i$ to  \eqref{ncopy} in the sense of Definition \ref{defpart}. 
\end{lemma}
\begin{proof}	The result would follow by a standard  contraction argument, c.f. \cite{cao2003type}. For fixed $i\in\{1,\cdots,N\},$ let $S_i$ be the closed set (with respect to the topology of $L^2(\Omega;C([0,T];\mathbb{T}^2))$) $$\left\{ Y\in L^2(\Omega;C([0,T];\mathbb{T}^2));  Y \text{ is adapted to the filtration }(\mathcal{G}^i_t)_{t\in [0,T]}  \right\}.$$
	Here $(\mathcal{G}^i_t)_{t\in [0,T]}$ is the normal filtration generated by $X_i(0), B_i, \text{and }\{W^{k},k=1,\cdots,d\}.$
Define  the following $(\mathcal{G}^i_t)_{t\in[0,T]}$-stopping times $\tau_R\assign\inf\{t\in (0,T], \int_0^t \|  v_s\|_{H^4}^2 \mathd s >R\}\wedge T,$ $R>0$.  
Furthermore, we define the  map $\Phi_R^i:S_i\rightarrow S_i$ by 
	\begin{align*}
		\Phi_R^i(Y)(t)\assign X_{i}(0)+\int_0^{t\wedge\tau_R}K*{  v}_s(Y(s))\mathd s+\sqrt{2}B_i(t)+\sum_{k=1}^{d}\int_0^t\sigma_k(Y(s))\circ\mathd W^k_s.
	\end{align*}

Next we rewrite $\Phi_R(Y_i)$ in It\^o's formulation, 
\begin{align*}
	\Phi_R^i(Y)(t)=&X_{i}(0)+ \int_0^{t\wedge \tau_R}K*{  v}_s(Y(s))\mathd s+\sqrt{2}B_i(t)+\sum_{k=1}^{d}\int_0^t\sigma_k(Y(s))\cdot\mathd W^k_s\nonumber
	\\&+\frac{1}{2}\sum_{k=1}^{d}\int_0^t \sigma_k(Y(s))\cdot \nabla \sigma_k(Y(s))\mathd s.
\end{align*}
Given $Y^1$ and $Y^2$ in $S_i$, we have 
\begin{align*}
	&\mathbb{E} \sup_{s\in [0,t]} \left| \Phi_R^i(Y^{1})(s)-\Phi_R^i(Y^2)(s)\right| ^2
	\\\lesssim & t\mathbb{E}\int_0^{t\wedge\tau_R} \left| K*{  v}_s(Y^{1}(s))-K*{  v}_s(Y^{2}(s))\right|^2 \mathd s
	\\&+ t\mathbb{E}\int_0^t\left| \sum_{k=1}^{d} \sigma_k(Y^{1}(s))\cdot \nabla \sigma_k(Y^{1}(s))-\sigma_k(Y^{2}(s))\cdot \nabla \sigma_k(Y^{2}(s)) \right|^2 \mathd s
	\\&+  \mathbb{E} \sup_{s\in [0,t]}\left| \sum_{k=1}^{d}\int_0^s\sigma_k(Y^{1}(r))-\sigma_k(Y^{2}(r))\cdot\mathd W^k_r\right| ^2
	\\\lesssim &t\mathbb{E}\int_0^{t\wedge\tau_R}\|\nabla (K*{  v}_s)\|_{L^{\infty}}^2|Y^{1}(s)-Y^{2}(s)|^2\mathd s+t\mathbb{E}\int_0^t  \left( \sum_{k=1}^{d}\|\sigma_k\|_{C^2}^2 \right) ^2 |Y^{1}(s)-Y^{2}(s)|^2 \mathd s
		\\&+ \left( \sum_{k=1}^{d}\|\sigma_k\|_{C^2}^2\right) \mathbb{E} \int_0^t|Y^{1}(s)-Y^{2}(s)|^2 \mathd s
			\\\lesssim &\mathbb{E} \int_0^t \(1+ t\| {  v}_s\|_{H^4}^21_{\{s<\tau_R\}}\)|Y^{1}(s)-Y^{2}(s)|^2 \mathd s,
\end{align*}
where we used  Assumption \ref{assumption2}  $\sigma_k\in C^{\infty},k=1,\cdots,d,$ and Burkholder-Davis-Gundy inequality to get the second inequality, and use
\begin{equation}\label{eqt:conest}
	\|\nabla (K*{  v})\|_{L^{\infty}}\lesssim \|K \|_{L^1}\|\nabla   v\|_{L^{\infty}}\lesssim \|  v\|_{H^4}
\end{equation} to get the final inequality. 
We  employed  Sobolev embedding theorem to establish the second inequality in \eqref{eqt:conest}. Furthermore, we have 
\begin{align*}
	&\mathbb{E} \sup_{s\in [0,t]} \left| \Phi_R^i(Y^{1})(s)-\Phi_R^i(Y^2)(s)\right| ^2
	\\\lesssim&\mathbb{E}\left[ \left( \sup_{s\in [0,t]}|Y^{1}(s)-Y^{2}(s)|^2\right)   \left(t+t \int_0^{t\wedge\tau_R}\| {  v}_s\|_{H^4}^2\mathd s\right) \right] 
	\\\lesssim& t(1+R)\mathbb{E} \left(\sup_{s\in [0,t]}|Y^{1}(s)-Y^{2}(s)|^2\right),
\end{align*}
where the last inequality follows by the definition of $\tau_R$.
Therefore, there exists a sufficiently small $t_0\leq T$ such that $\Phi_R^i$ is a contraction map on $L^2(\Omega;C([0,t_0];\mathbb{T}^2)).$ By the fixed point theorem and applying this argument for finite times, for every $i\in\{1,\cdots, N\},$ we find a unique $Y^R_i\in S_i$ such that $\Phi_R(Y^R_i)= Y^R_i.$  That is, 
\begin{align*}
	Y^R_i(t)= X_{i}(0)+\int_0^{t\wedge\tau_R}K*{  v}_s(Y^R_i(s))\mathd s+\sqrt{2}B_i(t)+\sum_{k=1}^{d}\int_0^t\sigma_k(Y^R_i(s))\circ\mathd W^k_s.
\end{align*}
Notice that $Y^R_i$ solves \eqref{ncopy} when $t< \tau_R$. Since $\tau_R\uparrow T$ almost surely as $R$ goes to infinity by $\mathbb{E} \int_0^T \|  v_t\|_{H^4}^2\mathd t <\infty $, we find   $Y_i\assign \lim_{R\rightarrow\infty}Y^R_i, $ solves the  stochastic differential equation \eqref{ncopy}. The uniqueness follows by the fact that every solution to \eqref{ncopy} coincides with  $Y^R_i$ before $\tau_R$. 

\end{proof}

Subsequently, we establish a connection between the SDE \eqref{ncopy} and the stochastic 2-dimensional Navier-Stokes equation \eqref{eqt:ns} via the stochastic Fokker-Planck equation, which is represented by the linearized stochastic 2-dimensional Navier-Stokes equation \eqref{ptsobol3}. More precisely, we have
\begin{align*}
	\mathbb{P}(\|  v_t-\tilde{  v}_t\|_{H^{-(2+  \eps)}}=0, \forall t\in [0,T])=1, \quad \forall   \eps\in(\frac{1}{2},1),
\end{align*}
where the random distribution on torus $\mathbb{T}^2,$ $\tilde{  v}_t(\mathd x_i)(\omega)$ is defined by
\begin{align}\label{eqt:tildev}
	\tilde{  v}_t(\mathd x_i)(\omega)\assign&\mathbb{E}[\xi_1]\tilde{\rho}_t(\mathd x_i)(\omega),\quad \omega\in \Omega.
\end{align} Here the random distribution $ \tilde{\rho}_t(\mathd x_i)$ on torus $\mathbb{T}^2,$ is defined as follows:
\begin{align}\label{def-rhomec}
	 \tilde{\rho}_t(\mathd x_i)(\omega)\assign\mathcal{L}((Y_i(t))| \mathcal{F}^{W}_T)(\mathd x_i)(\omega),\quad x_i\in\mathbb{T}^2,\omega\in \Omega,
\end{align}
 which describes  the trajectory  $Y_i(t)$ of  the vortice described by  \eqref{ncopy}  given  enviornmental noise  $\mathcal{F}^{W}_T.$ 

We thus have that  the stochastic differential equation \eqref{ncopy} can be interpreted as  a conditional Mckean Vlasov equation. 

\begin{lemma}\label{lem:unirho}
For each $i\in \{1,\cdots ,N\},$ the  processes of conditional distributions  $\tilde{\rho}_t$ defined in \eqref{def-rhomec} coincides with $\bar{\rho}_t=\frac{1}{\mathbb{E}[\xi_1]}  v_t$, i.e. $$\mathbb{P}(\|\bar{\rho}_t-\tilde{\rho}_t\|_{H^{-(2+  \eps)}}=0, \forall t\in [0,T])=1, \quad \forall   \eps\in(\frac{1}{2},1),$$
where $(  v_t)_{t\in [0,T]}$ is the probabilistically strong solution to  the stochastic 2-dimensional Navier-Stokes  equation \eqref{eqt:ns}.
In particular, $\tilde{  v}_t(\mathd x_i)$ defined in \eqref{eqt:tildev} coincides with $v_t(\mathd x_i),$
i.e. $$\mathbb{P}(\|  v_t-\tilde{  v}_t\|_{H^{-(2+  \eps)}}=0, \forall t\in [0,T])=1, \quad \forall  \eps\in(\frac{1}{2},1).$$
\end{lemma}
\begin{proof}
	By It\^o's formula and  taking conditional expectations with respect to $\mathcal{F}^W_t$ gives that 
	\begin{align*}
	\mathbb{E} \left(	\varphi(	Y_i(t))|\mathcal{F}^W_t\right)=	&\mathbb{E} \left(	\varphi(	X_i(0))|\mathcal{F}^W_t\right)+\mathbb{E} \left( \int_0^t\nabla\varphi(Y_i(s))\cdot K*  v_s(Y_i(s))\mathd s | \mathcal{F}_t^W\right)
	\\&+ \mathbb{E} \left( \int_0^t \Delta \varphi(Y_i(s))\mathd s| \mathcal{F}_t^W\right)
+\sqrt{2}\mathbb{E} \left( \int_0^t\nabla\varphi(Y_i(s))\cdot\mathd B_i(s) | \mathcal{F}_t^W\right)
	\\&+\sum_{k=1}^{d}\mathbb{E} \left( \int_0^t\nabla\varphi(Y_i(s))\cdot \sigma_k(Y_i(s))\circ \mathd W^k_s| \mathcal{F}_t^W\right).
	\end{align*}
	Using Lemma \ref{lem:con3} and  Lemma  \ref{lem:con1},   we find that 
	\begin{align*}
		\mathbb{E} \left(	\varphi(	Y_i(t))|\mathcal{F}^W_t\right)=	&\mathbb{E} \left(	\varphi(	X_i(0))\right)+\int_0^t \mathbb{E} \left( \nabla\varphi(Y_i(s))\cdot K*  v_s(Y_i(s))| \mathcal{F}_s^W\right)\mathd s 
		\\&+ \int_0^t  \mathbb{E} \left(\Delta \varphi(Y_i(s))| \mathcal{F}_s^W\right)\mathd s
+\sum_{k=1}^{d} \int_0^t\mathbb{E} \left(\nabla\varphi(Y_i(s))\cdot \sigma_k(Y_i(s))| \mathcal{F}_s^W\right)\circ \mathd W^k_s.
	\end{align*}

	Therefore   $\left(\mathcal{L}(Y_i(t)|\mathcal{F}^{W}_T)\right)$, which is a continuous version of $\left(\mathcal{L}(Y_i(t)|\mathcal{F}^{W}_t)\right)$, i.e. $\left(\mathcal{L}(Y_i(t)|\mathcal{F}^{W}_T)\right)\in C([0,T],\mathcal{P}(\mathbb{T}^2)) ,$\footnote{$\mathcal{P}(\mathbb{T}^2)$ denote the space of probability measures on $\mathbb{T}^2.$}   solves the equation $ \eqref{ptsobol3}$ in the distributional sense with the initial value $\bar{\rho}_0(x)$.  
	Since $\mathcal{P}(\mathbb{T}^2)\subset H^{-(1+\eps)}$ for any $  \eps\in(\frac{1}{2},1)$ in
	 \cite[Lemma 2.1]{hauray2014kac} and the fact that the norm $\|\cdot\|_{H^{-(1+  \eps)}}$ for a probability measure can be controlled by a positive constant $C\assign(\sum_{m\in \mathbb{Z}^2}(1+|m|^2)^{-(1+  \eps)})^{\frac{1}{2}},$ we deduce the result from Lemma \ref{lem:linearunique}. 
	
\end{proof}
Before proceeding, we introduce several random measures which characterize the motion of vortices given the information about  environment.
Recall that $X^N$ is the probabilistically strong solution to the interacting particle system \eqref{eqt:vortex} and  $\{\xi_i, i=1,\cdots,N\}$  denote the intensities in the  interacting particle system \eqref{eqt:vortex} with the same distribution $\mathcal{L}(\xi)$ introduced in  Assumption \ref{initial}. We first introduce the  random distribution $F^N_t(\mathd x^N,\mathd z^N)$  on $\mathbb{T}^{2N}\times\mathbb{R}^N,$ characterizing the trajectory and circulation $(X^N(t),\xi^N)$ of the vortices described by the interacting particle system \eqref{eqt:vortex},  given the  enviornmental noise  $\mathcal{F}^{W}_T,$ is defined as follows:
\begin{equation}\label{def-Fparticle}
	F^N_t(\mathd x^N,\mathd z^N)( \omega)\assign \mathcal{L}((X^N(t),\xi^N)|\mathcal{F}_{T}^{W})(\mathd x^N,\mathd z^N)(   \omega), \quad x^N\in\mathbb{T}^{2N},z^N\in\mathbb{R}^N,   \omega\in\Omega.
\end{equation}
We then introduce the random measure $\tilde{F}_t(\mathd x_i,\mathd z_i)$ on $\mathbb{T}^2\times\mathbb{R},$ 
the conditional distribution of $(Y_i(t),\xi_i)$  described by the conditional Mckean-Vlasov equation \eqref{ncopy},  given the  enviornmental noise  $\mathcal{F}^{W}_T,$ is defined by
\begin{align}\label{def-Fmckean}
	\tilde{F}_t(\mathd x, \mathd z)( \omega)&\assign \mathcal{L}((Y_i(t),\xi_i)| \mathcal{F}^{W}_T)(\mathd x_i,\mathd z_i)(   \omega), \quad x_i\in\mathbb{T}^2,z_i\in\mathbb{R},   \omega\in\Omega.
\end{align} 

 Based on the mean field  limit equation \eqref{eqt:ns}, we construct the  random distribution $\bar{F}^N_t(\mathd x^N,\mathd z^N)$  on $\mathbb{T}^{2N}\times\mathbb{R}^N,$ as follows:
\begin{align}\label{def-Fmean}
	\bar{F}^N_t(\mathd x^N,\mathd z^N)(\omega)\assign \bar{\rho}^{N}_{t}(\mathd x^N)(\omega)\mathcal{L}(\xi^N)(\mathd z^N), \quad x^N\in\mathbb{T}^{2N},z^N\in\mathbb{R}^N,   \omega\in\Omega,
\end{align}
where the random measure 	$\bar{\rho}^{N}_{t}(\mathd x^N)$ on torus $\mathbb{T}^{2N}$ is given by
 \begin{equation}\label{eqt:disinmean}
	\bar{\rho}^{N}_{t}(\mathd x^N)\assign\bar{\rho}^{\otimes N}_{t}(\mathd x^N)=\frac{1}{\mathbb{E}[\xi_1]^N}  v_t^{\otimes N}(\mathd x^N).
\end{equation} We now show that the random measure $\bar{F}^N_t(\mathd x^N,\mathd z^N)$  constructed from \eqref{eqt:ns} actually characterizes the conditional law of the solutions to the conditional Mckean-Vlasov equation \eqref{ncopy},  given the  enviornmental noise  $\mathcal{F}^{W}_T,$ i.e.  Corollary \ref{cor}.
	In Lemma \ref{lem:unirho},	we have shown that $$\mathbb{P}(\|\bar{\rho}_t-\tilde{\rho}_t\|_{H^{-(2+  \eps)}}=0, \forall t\in [0,T])=1, \quad \forall   \eps\in(\frac{1}{2},1),$$
	where $\bar{\rho}_t$ and $\tilde{\rho}_t$ are defined in \eqref{barrhot} and \eqref{def-rhomec} respectively. 
Applying  disintegration from Lemma \ref{lem:disinte}, we directly have the following result, which means our final result Theorem \ref{thm:entropy}  states that  as the number of particles approaches infinity, the particles within the weakly interacting particle system \eqref{eqt:vortex} tend to behave  like  some identical distributed particle $Y_i$ from \eqref{ncopy}. 
\begin{corollary}\label{cor}
	Under Assumption \ref{assumptiona}, we have 
	\begin{align*}
		\tilde{F}_t(\mathd x_i, \mathd z_i)&= \tilde{\rho}_t(\mathd x_i)\mathcal{L}(\xi)(\mathd z_i),\quad\forall t\in[0,T], \quad {\mathbb{P}}-a.s.,
	\end{align*}
	where the random measure $\tilde{F}_t(\mathd x_i, \mathd z_i)$ on $\mathbb{T}^{2}\times\mathbb{R},$ is defined in \eqref{def-Fmckean}, and the random measure $\tilde{\rho}_t(\mathd x_i)$ on $\mathbb{T}^{2},$ is defined in \eqref{def-rhomec}.\footnote{We notice the independence $Y_i\perp \xi_i$ as established in the proof of Lemma \ref{lem:uniquesde}.
		Consequently, we have $\tilde{\rho}_t(\mathd x)=\mathcal{L}((Y_i(t))| \mathcal{F}^{W}_T)(\mathd x)=\mathcal{L}((Y_i(t))| \mathcal{F}^{\xi_i,W}_T)(\mathd x)$ } 
	In particular, 
	\begin{align*}
	\bar{F}^N_t(\mathd x^N, \mathd z^N)=\tilde{F}^{\otimes N}_t(\mathd x^N, \mathd z^N),\quad\forall t\in[0,T], \quad {\mathbb{P}}-a.s.,
	\end{align*}
and
	\begin{align*}
		H_N(F^N_t|\bar{F}^N_t)=H_N(F^N_t|\tilde{F}^{\otimes N}_t), \quad\forall t\in[0,T],\quad \mathbb{P}-a.s.,
	\end{align*}
where	the  random distribution $\bar{F}^N_t(\mathd x^N,\mathd z^N)$  on $\mathbb{T}^{2N}\times\mathbb{R}^N,$  is defined in \eqref{def-Fmean}.
\end{corollary}

\section{Quantitative convergence}\label{sec:quant}
In this section, we will  establish a quantitative mean-field convergence concerning the relative entropy $H_N(F_t^N|\bar{F}_t^N)$ defined in Section \ref{entropy} between the particle system  \eqref{eqt:vortex} and the stochastic 2-dimensional Navier-Stokes  equation \eqref{eqt:ns}, as stated in Theorem \ref{thm:entropy}.
Here  the random measure $ F_t^N(\mathd x^N,\mathd z^N)$  on $\mathbb{T}^{2N}\times \mathbb{R}^{N}$ is defined in \eqref{def-Fparticle} and the random measure $\bar{F}_t^N(\mathd x^N,\mathd z^N)$  on $\mathbb{T}^{2N}\times \mathbb{R}^{N}$ is defined in \eqref{def-Fmean}.

Firstly, we establish the well-posedness of the particle system \eqref{eqt:vortex}, also known as the stochastic point vortex model, as a stochastic variant of the classical point vortex model. Our model is inspired by the model presented in \cite{flandoli2011full}. The primary distinction is the introduction of individual   noise, denoted as $\{B_i, i=1,\cdots,N\}$ which implies the appearance of $\Delta  v$ describing the viscosity of the fluid in the mean field limit equation \eqref{eqt:ns}.

\subsection{Stochastic point vortex model}\label{sec:sto vortex}

The well-posednesss of the point vortex model in various settings has been  established in \cite{osada1985stochastic,marchioro2012mathematical,takanobu1985existence,fontbona2007paths,flandoli2011full}.  The common idea is  to show that $X_i(t)\neq X_j(t)$ for all $t\in [0,T]$ and $i\neq j$ almost surely. Thus, the singularity of the kernel will never  be visited almost surely.

Let $G$ be the Green function induced by the Laplace operator on torus $\mathbb{T}^2$, with logarithmic divergence at $x=0$ and satisfies 
\begin{equation*}
	C_1\log|x|-C_3\leq G(x) \leq  C_2\log|x|+C_3,\quad 
	|\nabla G(x)|\leq C_4\frac{1}{|x|},\quad|\nabla^{ 2}G(x)|\leq C_4 \frac{1}{|x|^2},
\end{equation*}
where $C_i$, $i=1,2,3,4,$ are   positive universal  constants and Biot-Savart law $K=\nabla^{\perp}G.$ \footnote{$\nabla^{\perp}$ denotes $(\partial_2,-\partial_1).$}
As done in \cite[Section 3.2]{flandoli2011full}, we then regularize $G$ by a smooth periodic function $G_{\eps}$ satisfying $G_{\eps}(x)= G(x)$ for any $|x|>\eps$ and 
\begin{align*}
	C_1\log(|x|\vee \eps)-C_3\leq G_{\eps}(x) \leq  C_2\log(|x|\vee \eps)+C_3,\\ 
	|\nabla G_{\eps}(x)|\leq C_4\frac{1}{(|x|\vee \eps)},\quad |\nabla^{ 2}G_{\eps}(x)|\leq C_4 \frac{1}{(|x|\vee \eps)^2},
\end{align*}
where $C_i$, $i=1,2,3,4,$ are   positive universal  constants. 
Thus, for fixed $\eps>0$ and fixed $N>0,$ the following regularized system 
\begin{equation}
	X_{i,\eps}(t)=X_{i}(0)+ \frac{1}{N}\sum_{j\neq i}\int_0^t \xi_j\nabla^{\perp}G_{\eps}(X_{i,\eps}(s)-X_{j,\eps}(s))\mathd s+\sqrt{2}B^i_t+\sum_{k=1}^d\int_0^t\sigma_k(X_{i,\eps}(s))\circ\mathd W^k_s.\label{eqt:approvortex}
\end{equation} 
has a unique probabilistically strong solution $X^{\eps,N}\assign(X_{1,\eps},...,X_{N,\eps})$, which solves the original stochastic point vortex model \eqref{eqt:vortex} before the stopping time  $\tau_{\eps}$:
\begin{equation}
	\tau_{\eps}=\inf\{t\in [0,T], \exists i\neq j, |X_{i,\eps}(t)-X_{j,\eps}(t)|\leq \eps\}\nonumber.
\end{equation}
By showing  $\mathbb{P}(\lim_{\eps\rightarrow 0 }\tau_{\eps}\leq T)=0$, we obtain the following well-posedness of  the  stochastic point vortex model  in the form of \eqref{eqt:vortex}, which is driven by both individual and common   noise.  The  detailed proof is  postponed into Appendix   \ref{appendix:vortex}.
\begin{proposition}\label{thm:vortex}
Given  independent 1-dimensional Brownian motions $\{W^{k},k=1,\cdots,d\}$ and 2-dimensional Brownian motions $\{B_i, i=1,\cdots,N\}$ on  probability space $(\Omega,\mathcal{F},(\mathcal{F}_t)_{t\in [0,T]},\mathbb{P}),$	
then  there exists  a unique probabilistically strong solution  $X^N\assign(X_1,\cdots,X_N)$ to $\eqref{eqt:vortex}$ in the sense of Definition \ref{defpart}. Furthermore, it holds almost surely that $\forall t\in [0,T],$ $X^{\eps,N}(t)\rightarrow X^N(t),\text{ as }  \eps \rightarrow0.$ 
\end{proposition}

 Our objective in this section is to evolve the normalized relative entropy $H_N(F_t^N|\bar{F}_t^N)$ between the conditional joint measure $ F_t^N(\mathd x^N,\mathd z^N)$ about $(X^N,\xi^N)$   on $\mathbb{T}^{2N}\times \mathbb{R}^{N}$  defined in \eqref{def-Fparticle} and the random measure $\bar{F}_t^N(\mathd x^N,\mathd z^N)$  on $\mathbb{T}^{2N}\times \mathbb{R}^{N},$ defined in \eqref{def-Fmean}. In contrast to \cite{jabin2018quantitative}, we now have to consider two additional sources of randomness:  random intensities $\{\xi_1,\cdots,\xi_N\}$ and environmental noise $\{W_1,\cdots,W_d\}.$ On one hand, we  focus on  the conditional distribution with respect to  $\mathcal{F}_T^{W},$ which  can be roughly understood as fixing the random perturbation related to environmental noise. On the other hand, we aim to employ the technique of disintegration to address another random perturbation, which can be seen as fixing the random perturbation associated with random intensities. More precisely,   using Lemma \ref{lem:disinte}, we can decompose $F^N_t(\mathd x^N,\mathd z^N)$ into two distinct probability measures, one on torus $\mathbb{T}^{2N}$ and the other on $\mathbb{R}^N.$  
 \begin{align}\label{eqt:disparti}
	F^N_t\(\mathd x^N,\mathd z^N\)(\omega)=\rho^N_t(\mathd x^N, z^N,\omega)\mathcal{L}(\xi^N)(\mathd z^N), \quad\forall t\in[0,T],\omega\in \Omega, \quad {\mathbb{P}}-a.s,
\end{align}
where the random measure $\rho_t^N(\mathd x^N, z^N,\omega)$ on $\mathbb{T}^{2N},$ \footnote{Here $\rho_t^N(\mathd x^N, z^N,\omega)$ is a function of two variables $(z^N, \omega)\in \tilde{\Omega}$ and $B\in \mathbb{B}(\mathbb{T}^{2N}),$ such that  $\rho_t^N(B, z^N,\omega)$ is $\tilde{\mathcal{F}}$-measurable in  $(z^N, \omega)$ for fixed $B$ and a measure in $B$ for fixed $(z^N, \omega).$ We may use $\rho_t^{N}(z^N,\omega)$ or $\rho_t^{N}$ to denote $\rho^N_t(\mathd x^N, z^N,\omega)$ in the following content for simlicity.}  denotes $\mathcal{L}((X_t^N)|\mathcal{F}_T^{W,\xi^N})(\mathd x^N,z^N,W(  \omega)),$  and $(z^N, \omega)$ belongs to  the extended probability space, 
\begin{align}\label{eqt:extenpro}
\(\tilde{\Omega},\tilde{\mathcal{F}},\tilde{\mathbb{P}}\)\assign \(\mathbb{R}^N\times\Omega,\overline{\mathbb{B}(\mathbb{R}^N)\otimes \mathcal{F}},\mathcal{L}(\xi^{ N})\times \mathbb{P}\).
\end{align}
The $\sigma$-algebra $ \overline{\mathbb{B}(\mathbb{R}^N)\otimes \mathcal{F}}$ denotes the completion of ${\mathbb{B}(\mathbb{R}^N)\otimes \mathcal{F}}$ with respect to $\mathcal{L}(\xi^{ N})\times \mathbb{P}.$

Based on the decompositions of the random measures  $F^N(\mathd x^N,\mathd z^N)$ and $\bar{F}^N(\mathd x^N,\mathd z^N),$ i.e. \eqref{eqt:disparti} and \eqref{def-Fmean},
applying Lemma 1.4.3 in \cite{dupuis1993weak}, we have
the following relationship between $	H_N(F_t^N|\bar{F}_t^N)$ and $		H_N(\rho_t^N|\bar{\rho}_t^N),$ 
\begin{align}	H_N(F_t^N|\bar{F}_t^N)=\int_{\mathbb{R}^N}H_N(\rho_t^N|\bar{\rho}_t^N)\mathcal{L}(\xi^N)(\mathd z^N),\quad \forall t\in[0,T],\quad \mathbb{P}-a.s..\label{con:Frho}
\end{align}

 We start with the relative entropy  $ H_N(\rho_t^N|\bar{\rho}_t^{ N}) $ conditioning on the intensities, and write it as 
\begin{align*}
	H_N(\rho_t^N|\bar{\rho}_t^{ N}) = \frac{1}{N}\int_{\mathbb{T}^{2N}} \rho^N_t \log \rho^N_t \mathd x^N - \frac{1}{N}\int_{\mathbb{T}^{2N}}\rho^N_t \log \bar{\rho}_t^N \mathd x^N.
\end{align*} 
We then evolve the entropy $H(\rho_t^N)= \int_{\mathbb{T}^{2N}} \rho^N_t \log \rho^N_t \mathd x^N$ for $\rho_{t}^N$, which is also the first term in the relative entropy $		H_N(\rho_t^N|\bar{\rho}_t^N).$ The idea is to show that  an uniform estimate for the entropy and  Fisher information for the regularized system \eqref{eqt:approvortex}  in $\eps>0$ (see \eqref{est:unifishereps} below).  The conclusion then follows by the lower semicontinuity of Fisher information and entropy functionals.
\begin{lemma}\label{lem:entropy solution}
	Given the random measures $( \rho_t^{N}(z^N,\omega))_{t\in[0,T]}$ on $\mathbb{T}^{2N}$ defined in \eqref{eqt:disparti}, where $(z^N,\omega)\in \tilde{\Omega}$ defined in \eqref{eqt:extenpro}, for each nonnegative $(\mathcal{F}_t)_{t\in [0,T]}$-progressive measurable procress $C(t)\in L^1(\Omega\times[0,T]),$ it holds $\tilde{\mathbb{P}}$-almost surely for each  $t\in [0,T]$  that  
	\begin{align}\label{eqt:scaentr}
&e^{- \int_0^{t}C(s)\mathd s} \int_{\mathbb{T}^{2N}} \rho_t^{N}\log {\rho}_t^{N }  \mathd x^N 
-  \int_{\mathbb{T}^{2N}} \rho_0^{N}\log {\rho}_0^{N }  \mathd x^N 
\\\leq &- \int_0^t 	C(s)e^{- \int_0^{s}C(r)\mathd r}  \int_{\mathbb{T}^{2N}} \rho_s^{N}\log {\rho}_s^{N }  \mathd x^N \mathd s \nonumber
-\int_0^t 	e^{- \int_0^{s}C(r)\mathd r} \int_{\mathbb{T}^{2N}}\frac{|\nabla\rho_s^{N}|^2}{\rho^{N}_s}   \mathd x^N  \mathd s .
	\end{align}
where $\tilde{\mathbb{P}}$ is defined in \eqref{eqt:extenpro}.
In particular,  
	\begin{align}
	\int_{\mathbb{T}^{2N}}\rho_t^{N}\log \rho_t^{N} \mathd x^N +\int_0^t\int_{\mathbb{T}^{2N}}\frac{|\nabla\rho_s^{N}|^2}{\rho^{N}_s} \mathd x^N \mathd s\leq \int_{\mathbb{T}^{2N}}\rho_0^N\log \rho_0^N \mathd x^N . \label{est:unirela}
\end{align}
\end{lemma}
\begin{proof}
	The proof consists of two steps. 
	
Step 1: In this step, we deduce a similar control as \eqref{eqt:scaentr} for the regularized particle system \eqref{eqt:approvortex}, i.e. \eqref{est:unifishereps}. 

For fixed $\eps>0,$	we proceed the disintegration about random measure on  $\mathbb{T}^{2N}\times \mathbb{R}^{N},$ $$F^{  \eps,N}_t(\mathd x^N \mathd z^N)\assign\mathcal{L}((X^{\eps,N}_t,\xi^N)|\mathcal{F}_T^{W}),$$ and define random measure $\rho^{\eps,N}_t$ on torus $\mathbb{T}^{2N}$ analogously,  i.e. \begin{align}
		F^{\eps,N}_t\(\mathd x^N,\mathd z^N\)(\omega)=\rho^{\eps,N}_t(\mathd x^N)(z^N,\omega)\mathcal{L}(\xi^N)(\mathd z^N), \quad \forall t\in[0,T],\quad \mathbb{P}-a.s.,\nonumber
	\end{align}
	 where $(X^{\eps,N}(t))_{t\in[0,T]}$ is the probabilistically strong solution to the regularized particle system \eqref{eqt:approvortex} and $\rho_t^{\eps,N}$ denotes  
	 \begin{align}\label{def:regparti}
	 \mathcal{L}((X_t^{\eps,N})|\mathcal{F}_T^{W,\xi^N})(\mathd x^N,z^N,W(  \omega)).
	 \end{align}  Here $(z^N, \omega)$ belongs to  the extended probability space $\(\tilde{\Omega},\tilde{\mathcal{F}},\tilde{\mathbb{P}}\)$ defined in \eqref{eqt:extenpro}.
	 
	 As in  Lemma \ref{lem:unirho}, we use  It\^o's formula and take the conditional expectations for the solution of particle system with respect to $\mathcal{F}_T^{\xi^N,W},$  we have 
	 for all $\varphi\in C^{\infty}(\mathbb{T}^{2N}),$ it holds $\mathbb{P}$-almost surely that for all $t\in[0,T],$
	 \begin{align*}
	 	&\left\langle \mathcal{L}((X_t^{\eps,N})|\mathcal{F}_T^{W,\xi^N})(\mathd x^N,\xi^N,W),\varphi\right\rangle-\left\langle \mathcal{L}((X_0^{\eps,N})|\mathcal{F}_T^{W,\xi^N})(\mathd x^N,\xi^N,W),\varphi\right\rangle\\&=\sum_{i=1}^{N}\int_{0}^{t}\left\langle \mathcal{L}((X_s^{\eps,N})|\mathcal{F}_T^{W,\xi^N})(\mathd x^N,\xi^N,W),\Delta\varphi\right\rangle \mathd s\\&+\frac{1}{N}\sum_{i\neq j}^{N}\int_{0}^{t}\left\langle z_j\nabla^{\perp}G_{\eps}(x_i-x_j)\cdot\nabla_{x_i}\varphi, \mathcal{L}((X_s^{\eps,N})|\mathcal{F}_T^{W,\xi^N})(\mathd x^N,\xi^N,W)\right\rangle \mathd s\\&+\sum_{i=1}^{N}\sum_{k=1}^{d}\int_{0}^{t}\left\langle \sigma_k(x_i)\cdot\nabla_{x_i}\varphi,\mathcal{L}((X_s^{\eps,N})|\mathcal{F}_T^{W,\xi^N})(\mathd x^N,\xi^N,W)\right\rangle \circ\mathd W^k_s.
	 \end{align*}
 This implies that 
 for all $\varphi\in C^{\infty}(\mathbb{T}^{2N}),$ it holds $\tilde{\mathbb{P}}$-almost surely that for all $t\in[0,T],$
 \begin{align*}
 	\left\langle \rho^{\eps,N}_t,\varphi\right\rangle-\left\langle \rho^{N}_0,\varphi\right\rangle&=\sum_{i=1}^{N}\int_{0}^{t}\left\langle \rho^{\eps,N}_s,\Delta\varphi\right\rangle \mathd s+\frac{1}{N}\sum_{i\neq j}^{N}\int_{0}^{t}\left\langle z_j\nabla^{\perp}G_{\eps}(x_i-x_j)\cdot\nabla_{x_i}\varphi, \rho^{\eps,N}_s\right\rangle \mathd s\\&+\sum_{i=1}^{N}\sum_{k=1}^{d}\int_{0}^{t}\left\langle \sigma_k(x_i)\cdot\nabla_{x_i}\varphi,\rho^{\eps,N}_s\right\rangle \circ\mathd W^k_s.
 \end{align*}
where the random measures $(\rho^{\eps,N}_t)_{t\in [0,T]}$ on $\mathbb{T}^{2N}$
	  denote $(\mathcal{L}((X_t^{\eps,N})|\mathcal{F}_T^{W,\xi^N})(\mathd x^N,z^N,W(  \omega)))_{t\in[0,T]}.$
	Briefly, $(\rho^{\eps,N}_t)_{t\in [0,T]}$ solves  the following stochastic Fokker-Planck equation \eqref{eqt:louappro} in the distributional sense $\tilde{\mathbb{P}}$-almost surely with initial value $ \rho^{N}_0=\bar{\rho}_0^{\otimes N}\in H^3(\mathbb{T}^{2N})$ given in Assumption \ref{initial},
	\begin{align}\label{eqt:louappro}
		\mathd f=\sum_{i=1}^N\Delta_{x_i}f\mathd t -\frac{1}{N}\sum_{i\neq j}\div_{x_i}\(z_j\nabla^{\perp}G_{\eps}(x_i-x_j)f\)\mathd t-\sum_{i=1 }^N\sum_{k=1}^d\div_{x_i}\(\sigma_k(x_i)f\)\circ\mathd W^k_s.
	\end{align}

	For this stochastic Fokker-Planck equation with smooth coefficients and  initial value $\rho^{N}_0=\bar{\rho}_0^{\otimes N}\in H^3(\mathbb{T}^{2N}),$  according to Theorem 5.1  and Theorem 7.1 in \cite{krylov1999analytic},
	$(\rho^{\eps,N}_t)_{t\in[0,T]}$ belongs to    $$C([0,T],H^3(\mathbb{T}^{2N})),$$ satisfying $\tilde{\mathbb{E}}\int_{0}^{T}\|\rho^{\eps,N}_t\|_{H^4(\mathbb{T}^{2N})}^2\mathd t<\infty.$  As in  Lemma \ref{lem:reg}, we can also obtain that it holds $\tilde{\mathbb{P}}$-almost surely that for all $t\in[0,T]$  and $x^N\in\mathbb{T}^{2N},$
		\begin{align}\label{eq:entropy}
		 \rho^{\eps,N}_t(x^N)- \rho^{\eps,N}_0(x^N)&=\sum_{i=1}^N\int_{0}^{t}\Delta_{x_i}\rho^{\eps,N}_s(x^N)\mathd s -\frac{1}{N}\sum_{i\neq j}\int_{0}^{t}\div_{x_i}\(z_j\nabla^{\perp}G_{\eps}(x_i-x_j)\rho^{\eps,N}_s(x^N)\)\mathd s\nonumber\\&-\sum_{i=1 }^N\sum_{k=1}^d\int_{0}^{t}\div_{x_i}\(\sigma_k(x_i)\rho^{\eps,N}_s(x^N)\)\circ\mathd W^k_s.
	\end{align} 
	 
	 We then evolve the entropy of $(\rho^{\eps,N}_t)_{t\in [0,T]}$ by using It\^o's formula  
	\begin{align}
		&\int_{\mathbb{T}^{2N}}\rho_t^{\eps,N}\log \rho_t^{\eps,N} \mathd x^N\nonumber
		\\=&\int_{\mathbb{T}^{2N}}\rho_0^{\eps,N}\log \rho_0^{\eps,N} \mathd x^N -\int_0^t\int_{\mathbb{T}^{2N}}\frac{|\nabla\rho_s^{\eps,N}|^2}{\rho^{\eps,N}_s} \mathd x^N \mathd s\nonumber
		\\&-\frac{1}{N}\sum_{i\neq j}\int_0^t\int_{\mathbb{T}^{2N}}\(1+\log\rho^{\eps,N}_s\)\div_{x_i}\(z_j\nabla^{\perp}G_{\eps}(x_i-x_j)\rho_s^{\eps,N}\)\mathd x^N \mathd s\nonumber
		\\&-\sum_{i=1 }^N\sum_{k=1}^d\int_0^t\left(\int_{\mathbb{T}^{2N}}\(1+\log\rho^{\eps,N}_s\)\div_{x_i}\(\sigma_k(x_i)\rho_s^{\eps,N}\)\mathd x^N \right)\circ\mathd W^k_s.\nonumber
	\end{align}

	Since  $\nabla\cdot\nabla^{\perp}=0$ and $\{\sigma_k, k=1,\cdots,d\}$ are divergence free by Assumption \ref{assumption2}, integrating by parts leads to 
	\begin{align}
		\int_{\mathbb{T}^{2N}}\rho_t^{\eps,N}\log \rho_t^{\eps,N} \mathd x +\int_0^t\int_{\mathbb{T}^{2N}}\frac{|\nabla\rho_s^{\eps,N}|^2}{\rho^{\eps,N}_s} \mathd x \mathd s=\int_{\mathbb{T}^{2N}}\rho_0^N\log \rho_0^N \mathd x .\label{estiuni}
	\end{align}
	We then have 
	\begin{align}
		&	e^{- \int_0^{t}C(s)\mathd s}  \int_{\mathbb{T}^{2N}} \rho_t^{\eps,N}\log {\rho}_t^{\eps,N }  \mathd x^N	+\int_0^t 	e^{- \int_0^{s}C(r)\mathd r} \left( C(s)\int_{\mathbb{T}^{2N}} \rho_s^{\eps,N}\log {\rho}_s^{\eps,N }  \mathd x^N\right)\mathd s \nonumber
		\\+& \int_0^t 	e^{- \int_0^{s}C(r)\mathd r} \left(\int_{\mathbb{T}^{2N}}\frac{|\nabla\rho_s^{\eps,N}|^2}{\rho^{\eps,N}_s}   \mathd x^N   \right)\mathd s  
		=\int_{\mathbb{T}^{2N}} \rho_0^{N}\log {\rho}_0^{N }  \mathd x^N. \label{est:unifishereps}
	\end{align}

Step 2: In this step, we  deduce the final result, i.e.  \eqref{eqt:scaentr}, based on  the weak convergence of measure for $\rho^{\eps,N}_t.$ 

 In the following, we want to  let $\eps \rightarrow 0$ in \eqref{est:unifishereps}. To achieve this, we first show a weak convergence of measure result.
Since $X_t^{\eps,N}\rightarrow X_t^N$ for all $t\in [0,T]$ almost surely, 
we obtain that it holds $\tilde{\mathbb{P}}-$ almost surely that $$ \rho^{\eps,N}_t \rightharpoonup  \rho^{N}_t, \quad \eps \rightarrow 0,\quad \forall t\in [0,T],$$
 in the sense of weak convergence in $\mathcal{P}(\mathbb{T}^{2N}).$
By  the lower semicontinuity of Fisher information and entropy functionals,  we arrive at 
\begin{align*}
	&e^{- \int_0^{t}C(s)\mathd s} \int_{\mathbb{T}^{2N}} \rho_t^{N}\log {\rho}_t^{N }  \mathd x^N  + \int_0^t e^{- \int_0^{s}C(r)\mathd r} C(s)  \int_{\mathbb{T}^{2N}} \rho_s^{N}\log {\rho}_s^{N }  \mathd x^N   \mathd s 
	\\&+\int_0^t 	e^{- \int_0^{s}C(r)\mathd r} \int_{\mathbb{T}^{2N}}\frac{|\nabla\rho_s^{N}|^2}{\rho^{N}_s}   \mathd x^N  \mathd s \leq  \int_{\mathbb{T}^{2N}} \rho_0^{N}\log {\rho}_0^{N }  \mathd x^N .
\end{align*}
This completes the proof.
\end{proof}

\subsection{Envolving the relative entropy}\label{sec:relaen}

Recall that \begin{align*}
	H_N(\rho_t^N|\bar{\rho}_t^{ N}) = \frac{1}{N}\int_{\mathbb{T}^{2N}} \rho^N_t \log \rho^N_t \mathd x^N - \frac{1}{N}\int_{\mathbb{T}^{2N}}\rho^N_t \log \bar{\rho}_t^N \mathd x^N,
\end{align*} 
where the random measures  $\rho_t^N(\mathd x^N)$  and $ \bar{\rho}_t^N(\mathd x^N)$ on $\mathbb{T}^{2N}$  are defined in \eqref{eqt:disparti} and \eqref{eqt:disinmean}.

Our next objective is to evolve the second term $\frac{1}{N}\int_{\mathbb{T}^{2N}}\rho^N_t \log \bar{\rho}_t^N \mathd x^N$ in the normalized relative entropy $H_N(\rho_t^N|\bar{\rho}_t^{ N})$. The typical approach involves employing Itô's formula to derive the Liouville equation satisfied by $(\rho^{N}_t)_{t\in[0,T]}$ in the distribution sense 
 and $\log \bar{\rho}^N_t,$ often be modeled as a smooth function (such as \cite{jabin2016mean}, \cite{jabin2018quantitative}), is recarded as a test function to calculate the second term $\frac{1}{N}\int_{\mathbb{T}^{2N}}\rho^N_t \log \bar{\rho}_t^N \mathd x^N$. However, due to  the restriction of the regularity in time of $\bar{\rho}_t^N$ caused by the introduction of common   noise $\{W_k, k=1,\cdots,d\},$ it is difficult to calculate the second term $\frac{1}{N}\int_{\mathbb{T}^{2N}}\rho^N_t \log \bar{\rho}_t^N \mathd x^N$ in the normalized relative entropy $H_N(\rho_t^N|\bar{\rho}_t^{ N})$ directly by regarding $\log \bar{\rho}^N_t$ as a test function. The idea here is similiar to the approach employed for the calculation of the first term $\frac{1}{N}\int_{\mathbb{T}^{2N}} \rho^N_t \log \rho^N_t \mathd x^N$ in Lemma \ref{lem:entropy solution}. Initially, we address the regularized system \eqref{eqt:approvortex}, i.e. we calculate $\frac{1}{N}\int_{\mathbb{T}^{2N}}\rho^{\eps,N}_t \log \bar{\rho}_t^N \mathd x^N$ at first. We then obtain $\frac{1}{N}\int_{\mathbb{T}^{2N}}\rho^{N}_t \log \bar{\rho}_t^N \mathd x^N$ by the weak convergence of $\rho^{\eps,N}_t$. It's important to note that we utilize the estimate provided in Lemma \ref{fournier3.3} (Lemma 3.3 in \cite{fournier2014propagation}) to handle the interacting part $J_4^N$ (see below), which is arguably the most challenging part.

\begin{lemma}\label{lem:entro-2}Given the random measures $( \rho_t^{N}(z^N,\omega))_{t\in[0,T]}$ and $( \bar{\rho}_t^{N}(\omega))_{t\in[0,T]}$  on $\mathbb{T}^{2N}$ defined in \eqref{eqt:disparti} and \eqref{eqt:disinmean}, where $(z^N,\omega)\in \tilde{\Omega}=\mathbb{R}^N\times\Omega$ is defined in \eqref{eqt:extenpro} and $\omega \in\Omega,$  it holds $\tilde{\mathbb{P}}$-almost surely for each $t\in [0,T]$ that 
	\begin{align}	\label{eqt:sceond}
		  \int_{\mathbb{T}^{2N}}\rho^N_t \log \bar{\rho}_t^N \mathd x^N =\int_{\mathbb{T}^{2N}}\rho^N_0 \log \bar{\rho}_0^N \mathd x^N+\int_{0}^{t}J_1^N (s)\mathd s +\int_{0}^{t}J_2^N(s) \mathd s+\int_{0}^{t}J_3^N (s)\mathd s+\int_{0}^{t}J_4^N (s)\mathd s,
	\end{align} 
where $J_i,i=1,2,3,4,$ are defined by
	\begin{align*}
		J_1^N(s)\assign&\sum_{i=1}^N\int_{\mathbb{T}^{2N}} \rho_s^N\frac{\Delta \bar{\rho}_s(x_i)}{\bar{\rho}_s(x_i)} \mathd x^N,
		\\J_{2}^N(s)\assign& \sum_{i=1}^N\int_{\mathbb{T}^{2N}}  \rho^N_s(x^N) \Delta_{x_i}\log \bar{\rho}_s^N (x^N)\mathd x^N,
		\\J_{3}^N(s)\assign&-\sum_{i=1}^N \int_{\mathbb{T}^{2N}} \rho^N_s(x^N) K*  v_s(x_i)\cdot \nabla\log \bar{\rho}_s(x_i) \mathd x^N,
\\J_{4}^N(s)\assign&	\frac{1}{N}\sum_{i=1}^{N}\sum_{j\neq i}\int_{\mathbb{T}^{2N}}\left(\rho^{N}_s (x^N)z_jK(x_i-x_j)\right)\cdot\nabla_{x_i} \log \bar{\rho}_s^N(x^N) \mathd x^N.
\end{align*}
Here $\tilde{\mathbb{P}}$ is defined in \eqref{eqt:extenpro} and $x^N=(x_1,\cdots,x_N)\in \mathbb{T}^{2N}, z^N=(z_1,\cdots,z_N)\in\mathbb{R}^N.$
\end{lemma}
\begin{proof}
	The proof consists of two steps. 
	
	Step 1: In the first step, we deduce a similar equality as in \eqref{eqt:sceond} for a regularized version, i.e. \eqref{eqt:secondre}. 
		
		We start with  the approximation model \eqref{eqt:louappro}, i.e.
		\begin{align*}
			\mathd \rho^{\eps,N}_t=\sum_{i=1}^N\Delta_{x_i}\rho^{\eps,N}_t\mathd t -\frac{1}{N}\sum_{i\neq j}\div_{x_i}\(z_j\nabla^{\perp}G_{\eps}(x_i-x_j)\rho^{\eps,N}_t\)\mathd t-\sum_{i=1 }^N\sum_{k=1}^d\div_{x_i}\(\sigma_k(x_i)\rho^{\eps,N}_t\)\circ\mathd W^k_s,
		\end{align*} 
	where the random measures $\rho_t^{\eps,N}(\mathd x^N)$ on $\mathbb{T}^{2N}$ is defined in \eqref{def:regparti}.
	
	Based on \eqref{eq:re2} and \eqref{eq:entropy}, using 	It\^o's formula to $\rho_t^{\eps,N}(x^N)\log\bar{\rho}_t^N(x^N) $ and integrating  on $\mathbb{T}^{2N}$, 
we arrive at
	\begin{align}
	&	\mathd  \int_{\mathbb{T}^{2N}}\rho^{\eps,N}_t(x^N) \log \bar{\rho}_t^N(x^N) \mathd x^N\nonumber
\\=&J_1^{\eps,N}(t) \mathd t +J_2^{\eps,N}(t)\mathd t +J_3^{\eps,N}(t)\mathd t+J_4^{\eps,N}(t)\mathd t+ \mathd M^{\eps,N}(t),\nonumber
\end{align}
where  $J_1^{\eps,N}, J_2^{\eps,N},J_3^{\eps,N}$ and $J_4^{\eps,N}$ denote 
\begin{align*}
	J_1^{\eps,N}(t)=& \sum_{i=1}^N \int_{\mathbb{T}^{2N}}\rho_t^{\eps,N}(x^N)\frac{\Delta \bar{\rho}_t(x_i)}{\bar{\rho}_t(x_i)}\mathd x^N,
	\\J_2^{\eps,N}(t)=&\int_{\mathbb{T}^{2N}} \Delta \rho^{\eps,N}_t (x^N)\log \bar{\rho}_t^N(x^N) \mathd x^N=\sum_{i=1}^N\int_{\mathbb{T}^{2N}}  \rho^{  \eps,N}_s(x^N) \Delta_{x_i}\log \bar{\rho}_s^N (x^N)\mathd x^N,
	\\J_3^{\eps,N}(t)=&-\sum_{i=1}^N \int_{\mathbb{T}^{2N}} \rho^{\eps,N}_t(x^N)K*  v_t(x_i)\cdot  \nabla\log \bar{\rho}_t(x_i)\mathd x^N,
	\\J_4^{\eps,N}(t)= &\frac{1}{N}\sum_{i=1}^{N}\sum_{j\neq i}\int_{\mathbb{T}^{2N}}\left(\rho^{\eps,N}_t (x^N)z_j\nabla^{\perp}G_{\eps}(x_i-x_j)\right)\cdot \nabla_{x_i} \log \bar{\rho}_t^N(x^N)\mathd x^N,
\end{align*}
and 
$M^{\eps,N}(t)$ denotes 
\begin{align*}
	& M^{\eps,N}(t)
	\\=&- \sum_{i=1}^N\sum_{k=1}^{d}\int_0^t\left(\int_{\mathbb{T}^{2N}} \div_{x_i}\left( \sigma_k(x_i)\rho_s^{\eps,N}(x^N)\right)\log \bar{\rho}_s^N(x^N)+ \rho_s^{\eps,N}(x^N)\frac{\sigma_k(x_i)\cdot \nabla \bar{\rho}_s(x_i)}{\bar{\rho}_s(x_i)}\mathd x^N\right) \circ \mathd W^{k}_s.
\end{align*}
Here we used integration by parts to get the second equality in $J_2^{\eps,N}$  term.

Observe the cancellation  in $ M^{\eps,N}(t)$ by integration by parts, 
\begin{align}
	\int_{\mathbb{T}^{2N}} \div_{x_i} \left( \sigma_k(x_i)\rho_s^{\eps,N}\right)\log \bar{\rho}_s^N(x^N)\mathd x^N=&- \int_{\mathbb{T}^{2N}} \rho_s^{\eps,N}(x^N)\sigma_k(x_i)\cdot \nabla _{x_i}\log \bar{\rho}^N_s(x^N)\mathd x^N\nonumber
	\\=&-\int_{\mathbb{T}^{2N}} \rho_s^{\eps,N}(x^N)
	\frac{\sigma_k(x_i)\cdot \nabla \bar{\rho}_s(x_i)}{\bar{\rho}_s(x_i)}\mathd x^N,\quad \forall i,k.\nonumber
\end{align}
We thus have 
\begin{align*}
	M^{\eps,N}(t)=0.
\end{align*}
We then conclude that it holds $\tilde{\mathbb{P}}$-almost surely that for each $t\in [0,T]$ 
\begin{align}\label{eqt:secondre}
&\int_{\mathbb{T}^{2N}}\rho^{\eps,N}_t(x^N) \log \bar{\rho}_t^N(x^N) \mathd x^N-\int_{\mathbb{T}^{2N}}\rho^{N}_0(x^N) \log \bar{\rho}_0^N(x^N) \mathd x^N\nonumber\\&=\int_{0}^{t}J_1^{\eps,N} (s)\mathd s +\int_{0}^{t}J_2^{\eps,N}(s) \mathd s+\int_{0}^{t}J_3^{\eps,N} (s)\mathd s+\int_{0}^{t}J_4^{\eps,N} (s)\mathd s.
\end{align} 

Step 2: Based the regularity of $(\bar{\rho}_t^{N})_{t\in[0,T]}$ given in  \eqref{ptsobol1} and \eqref{ptsobol2}, and Lemma \ref{fournier3.3}, we let $\eps\rightarrow 0$ in \eqref{eqt:secondre} and obtain the final result.

Recall that $$\bar{\rho}=\frac{1}{\mathbb{E}[\xi_1]}  v\in C([0,T];C(\mathbb{T}^2)), \quad \tilde{\mathbb{P}}-a.s.,$$ and  \begin{align*}
	\underset{x\in \mathbb{T}^2}{\inf}\bar{\rho}_t\geq\underset{x\in \mathbb{T}^2}{\inf}\bar{\rho}_0>0, \quad \forall t\in[0,T],\quad\tilde{\mathbb{P}}-a.s.,
\end{align*} from Lemma \ref{lem:reg}\footnote{We can  observe that  the outcome of  Lemma \ref{lem:reg} is considered under $\mathbb{P}.$ However, it's important to note that $\bar{\rho}(t,x^N)(\omega)$ is independent of $z^N$. Consequently, all the results in Theorem \ref{thm:spde} and Lemma \ref{lem:reg} can be considered under the joint probability measure $\tilde{\mathbb{P}} (\mathd z^N, \mathd   \omega)= \mathcal{L}(\xi^N)(\mathd z^N) \otimes \mathbb{P}(\mathd   \omega)$. We will not explicitly mention this point.}, by $\rho^{\eps,N}_t \rightarrow \rho^{N}_t$ in $\mathcal{P}(\mathbb{T}^{2N}),$
we have 
\begin{align*}
	\int_{\mathbb{T}^{2N}}\rho^N_t \log \bar{\rho}_t^N \mathd x^N&=\lim\limits_{\eps\rightarrow 0 } \int_{\mathbb{T}^{2N}}\rho^{\eps,N}_t \log \bar{\rho}_t^N \mathd x^N, \quad \forall t\in[0,T],\quad\tilde{\mathbb{P}}-a.s..
\end{align*}
From \eqref{ptsobol2}, by Sobolev embedding theorem,  we have 
 \begin{align*}
	\int_0^T \|\bar{\rho}(s)\|_{C^2}^2 \mathd s\lesssim	\int_0^T \|\bar{\rho}(s)\|_{H^4}^2 \mathd s<\infty\quad\tilde{\mathbb{P}}-a.s..
\end{align*} 
We then deduce that there exists a null set $\tilde{N}$ such that for every
$\tilde{\omega}=(z^N,\omega)\in\tilde{N}^c,$
	$\frac{\Delta \bar{\rho}_t}{\bar{\rho}_t}\in C(\mathbb{T}^2)$ for $a.e.t\in[0,T]$
and
\begin{align*}
	\lim_{\eps \rightarrow 0}J_1^{\eps,N}(t)&=\sum_{i=1}^N \int_{\mathbb{T}^{2N}}\rho_t^N(x^N)\frac{\Delta \bar{\rho}_t(x_i)}{\bar{\rho}_t(x_i)}\mathd x^N. 
\end{align*}
 for $a.e.t\in[0,T]$. (We omit $z^N,\omega$ in $\rho_t^N$ and $\bar{\rho}_t^N$ for simplicity. )

 Furthermore, since $\rho_t^{\eps,N}$ is a probability measure,
 \begin{align*}
 	\mid J_1^{\eps,N}(t)\mid=\bigg| \int_{\mathbb{T}^{2N}}\rho_t^{  \eps,N}(x^N)\frac{\Delta \bar{\rho}_t(x_i)}{\bar{\rho}_t(x_i)}\mathd x^N \bigg|
 	&\leq \frac{1}{\underset{x\in \mathbb{T}^2}{\inf}\bar{\rho}_0} \| \Delta \bar{\rho}_t\|_{L^\infty}
 	\\&\lesssim\frac{\| \bar{\rho}_t\|_{H^4}}{\underset{x\in \mathbb{T}^2}{\inf}\bar{\rho}_0}.
 \end{align*}
Here we used Sobolev embedding theorem to get the second inequality.
Recall  $\int_0^T \|\bar{\rho}(s)\|_{H^4}^2 \mathd s$ is finite $\tilde{\mathbb{P}}-$almost surely, by the dominated convergence theorem,
we have 
\begin{align*}
	\lim_{\eps \rightarrow 0}\int_{0}^{t}J_1^{\eps,N}(s)\mathd s=\int_{0}^{t}J_1^{N}(s)\mathd s,\quad \forall t\in[0,T],\quad\tilde{\mathbb{P}}-a.s..
\end{align*}

Similiarly, we also have
\begin{align*}
\lim_{\eps \rightarrow 0}\int_{0}^{t}J_2^{\eps,N}(s)\mathd s&=\int_{0}^{t}J_2^{N}(s)\mathd s,\quad \forall t\in[0,T],\quad\tilde{\mathbb{P}}-a.s.,
\\\lim_{\eps \rightarrow 0}\int_{0}^{t}J_3^{\eps,N}(s)\mathd s&=\int_{0}^{t}J_3^{N}(s)\mathd s,\quad \forall t\in[0,T],\quad\tilde{\mathbb{P}}-a.s..
\end{align*}

The most challenging part is the fourth term $J_4^{N,\eps}.$
Recall that
\begin{align}\label{J4}
&\bigg|\int_{0}^{t}J_4^{  \eps,N}(s)\mathd s-\int_{0}^{t}J_4^{N}(s)\mathd s\bigg|\nonumber\\=&\frac{1}{N}\bigg|\sum_{i=1}^{N}\sum_{j\neq i}\int_{0}^{t}\int_{\mathbb{T}^{2N}}\left(\rho^{\eps,N}_s \nabla^{\perp}G_{\eps}(x_i-x_j)\right) \cdot z_j\nabla_{x_i}\log \bar{\rho}_s^N\mathd x^N \mathd s\nonumber\\&-\sum_{i=1}^{N}\sum_{j\neq i}\int_{0}^{t}\int_{\mathbb{T}^{2N}}\left(\rho^{N}_sK(x_i-x_j)\right) \cdot z_j\nabla_{x_i}\log \bar{\rho}_s^N\mathd x^N \mathd s\bigg|.
\end{align}
We define \begin{align*}
	h^j_s(x_i)\assign z_j\nabla_{x_i}\log \bar{\rho}_s^N=z_j\frac{\nabla_{x_i} \bar{\rho}_s(x_i)}{\bar{\rho}_s(x_i)}.
\end{align*}
Recall there exists some $M>0,$ such that $|z_j|\leq M,\quad  \forall j=1,\cdots,N,$ $\tilde{\mathbb{P}}-$ almost surely, given by   Assumption \eqref{initial}. By Sobolev embedding theorem, we have
\begin{align*}
	\|h^j_s\|_{L^{\infty}(\mathbb{T}^2)}\leq \frac{M}{\underset{x\in \mathbb{T}^2}{\inf}\bar{\rho}_0}\|\nabla\bar{\rho}_s\|_{L^{\infty}}\lesssim\|\bar{\rho}_s\|_{H^{2+\frac{1}{2}}}.
\end{align*}
By Theorem 1.52 in \cite{bahouri2011fourier}, we   deduce that 
\begin{align*}
	\int_{0}^{T}	\|h^j_s\|_{L^{\infty}}^8\mathd s&\lesssim \int_{0}^{T}\|\bar{\rho}_s\|^6_{H^2}\|\bar{\rho}_s\|_{H^4}^2\mathd s
	\\ &\lesssim \sup_{t\in [0,T ]}\|\bar{\rho}_t\|_{H^2}^6\int_{0}^{T}\|\bar{\rho}_s\|_{H^4}^2\mathd s.
\end{align*}
Recall that $\sup_{t\in [0,T ]}\|\bar{\rho}_t\|_{H^2}$ and $\int_{0}^{T}\|\bar{\rho}_s\|_{H^4}^2\mathd s$ are finite $\tilde{\mathbb{P}}-$almost surely, given by \eqref{ptsobol2} in Section \ref{sec2dns}, we then have $\int_{0}^{T}	\|h^j_s\|_{L^{\infty}}\mathd s$ and $\int_{0}^{T}	\|h^j_s\|_{L^{\infty}}^8\mathd s$ are finite $\tilde{\mathbb{P}}-$almost surely.

We now denote $\delta$ a constant number  belongs to $(0,1)$ and decompose \eqref{J4} into the following form.
\begin{align*}
	&\bigg|\int_{0}^{t}J_4^{\eps,N}(s)\mathd s-\int_{0}^{t}J_4^{N}(s)\mathd s\bigg|\\=&\frac{1}{N}\bigg|\sum_{i=1}^{N}\sum_{j\neq i}\int_{0}^{t}\int_{\mathbb{T}^{2N}}\left(\rho^{\eps,N}_s \nabla^{\perp}G_{\eps}(x_i-x_j)\right) \cdot h^j_s\mathd x^N \mathd s\\&-\sum_{i=1}^{N}\sum_{j\neq i}\int_{0}^{t}\int_{\mathbb{T}^{2N}}\left(\rho^{N}_sK(x_i-x_j)\right) \cdot h^j_s\mathd x^N \mathd s\bigg|
	\\\leq&\frac{1}{N}\sum_{i=1}^{N}\sum_{j\neq i} \int_{0}^{t}\bigg|\int_{\mathbb{T}^{2N}}\nabla^{\perp}G_{\delta}\cdot h^j_s(\rho^{\eps,N}_s -\rho^{N}_s )\mathd x^N\bigg|\mathd s\\&+\frac{1}{N}\sum_{i=1}^{N}\sum_{j\neq i}\int_{0}^{t}\|h^j_s\|_{L^{\infty}}\int_{\mathbb{T}^{2N}}\rho^{\eps,N}_s |\nabla^{\perp}G_{\delta}-K|\mathd x^N \mathd s\\&+\frac{1}{N}\sum_{i=1}^{N}\sum_{j\neq i}\int_{0}^{t}\|h^j_s\|_{L^{\infty}}\int_{\mathbb{T}^{2N}}\rho^{\eps,N}_s |\nabla^{\perp}G_{\eps}-K|\mathd x^N \mathd s
	\\&+\frac{1}{N}\sum_{i=1}^{N}\sum_{j\neq i}\int_{0}^{t}\|h^j_s\|_{L^{\infty}}\int_{\mathbb{T}^{2N}}\rho^{N}_s |\nabla^{\perp}G_{\delta}-K|\mathd x^N \mathd s.
\end{align*}

Due to the properties of Biot Savart law $K=\nabla^{\perp} G$ and the regularized version of Biot Savart law $\nabla^{\perp} G_{  \eps}$ introduced in Section \ref{sec:sto vortex}, i.e. $|K|\leq\frac{C}{|x|}$, $|\nabla^{\perp} G_{\eps}|\leq\frac{C}{|x|}$ and $K=\nabla^{\perp} G_  \eps$ when $|x|> \eps,$ where $C$ is a positive constant, we have
\begin{align*}
		&\bigg|\int_{0}^{t}J_4^{\eps,N}(s)\mathd s-\int_{0}^{t}J_4^{N}(s)\mathd s\bigg|
		\\&\leq\frac{1}{N} \sum_{i=1}^{N}\sum_{j\neq i}\int_{0}^{t}\bigg|\int_{\mathbb{T}^{2N}}\nabla^{\perp}G_{\delta}\cdot h^j_s(\rho^{\eps,N}_s -\rho^{N}_s )\mathd x^N\bigg|\mathd s\\&+\frac{C}{N}\sum_{i=1}^{N}\sum_{j\neq i}\int_{0}^{t}\|h^j_s\|_{L^{\infty}}\int_{\mathbb{T}^{2}}\frac{1}{|x_i-x_j|}I_{\{|x_i-x_j|\leq \delta\}}\rho^{\eps,N,2}_s\mathd x_i\mathd x_j \mathd s\\&+\frac{C}{N}\sum_{i=1}^{N}\sum_{j\neq i}\int_{0}^{t}\|h^j_s\|_{L^{\infty}}\int_{\mathbb{T}^{2}}\frac{1}{|x_i-x_j|}I_{\{|x_i-x_j|\leq \eps\}}\rho^{\eps,N,2}_s\mathd x_i\mathd x_j \mathd s
		\\&+\frac{C}{N}\sum_{i=1}^{N}\sum_{j\neq i}\int_{0}^{t}\|h^j_s\|_{L^{\infty}}\int_{\mathbb{T}^{2}}\frac{1}{|x_i-x_j|}I_{\{|x_i-x_j|\leq \delta\}}\rho^{N,2}_s\mathd x_i\mathd x_j \mathd s.
\end{align*}
 Here two random measures $\rho^{N,2}_t(\mathd x_i\mathd x_j), \rho^{\eps,N,2}_t(\mathd x_i\mathd x_j)$ on $\mathbb{T}^{4}$ denote the $2$-marginals for the two random measures $\rho^{N}_t(\mathd x^N), \rho^{\eps,N}_t(\mathd x^N)$ on $\mathbb{T}^{2N}$ (see the definition of marginal in Notations) and $C$ is a positive deterministic constant.
 
 The following inequality follows from the property that the  normalized Fisher information of the marginals is bounded by  the total normalized  Fisher information (Lemma 3.7 in \cite{hauray2014kac}), i.e. $I(\rho^{N,k})\leq \frac{k}{N}I(\rho^N), k\leq N,$ and Lemma \ref{fournier3.3}:
 \begin{align*}
 	&\bigg|\int_{0}^{t}J_4^{\eps,N}(s)\mathd s-\int_{0}^{t}J_4^{N}(s)\mathd s\bigg|\\&\leq  \sum_{i=1}^{N}\sum_{j\neq i}\int_{0}^{t}\bigg|\int_{\mathbb{T}^{2N}}\nabla^{\perp}G_{\delta}\cdot h^j_s(\rho^{\eps,N}_s -\rho^{N}_s )\mathd x^N\bigg|\mathd s\\&+\frac{C_{\eta,\beta}}{N}\sum_{i=1}^{N}\sum_{j\neq i}(\eps^{\eta}+\delta^{\eta})\int_{0}^{t}\(I^\beta(\rho^{  \eps,N}(s))+C_{\eta,\beta}\)\|h^j_s\|_{L^{\infty}}\mathd s\\&+\frac{C_{\eta,\beta}}{N}\sum_{i=1}^{N}\sum_{j\neq i}\delta^{\eta}\int_{0}^{t}\(I^\beta(\rho^{N}(s))+C_{\eta,\beta}\)\|h^j_s\|_{L^{\infty}}\mathd s.
 \end{align*}
 Here $\eta$ represent a constant number  belongs to $(0,1),$ $\beta>\frac{1+\eta}{2},$ and $C_{\eta,\beta}$ is a positive deterministic constant depends on $\beta$ and $\eta$.
 We now choose $\eta=\frac{1}{2},\beta=\frac{7}{8}$ and use Holder's inequality to obtain that 
 \begin{align*}
 	&\bigg|\int_{0}^{t}J_4^{\eps,N}(s)\mathd s-\int_{0}^{t}J_4^{N}(s)\mathd s\bigg|\\&\leq  \sum_{i=1}^{N}\sum_{j\neq i}\int_{0}^{t}\bigg|\int_{\mathbb{T}^{2N}}\nabla^{\perp}G_{\delta}\cdot h^j_s(\rho^{\eps,N}_s -\rho^{N}_s )\mathd x^N\bigg|\mathd s\\&+\frac{C}{N}\sum_{i=1}^{N}\sum_{j\neq i}(\eps^{\frac{1}{2}}+\delta^{\frac{1}{2}})\bigg[\(\int_{0}^{t}I(\rho^{\eps,N}(s))\mathd s\)^{\frac{7}{8}}\(\int_{0}^{T}\|h^j_s\|^8_{L^{\infty}}\mathd s\)^{\frac{1}{8}}+C\int_{0}^{T}\|h^j_s\|_{L^{\infty}}\mathd s\bigg]\\&+\frac{C}{N}\sum_{i=1}^{N}\sum_{j\neq i}\delta^{\frac{1}{2}}\bigg[\(\int_{0}^{t}I(\rho^{N}(s))\mathd s\)^{\frac{7}{8}}\(\int_{0}^{T}\|h^j_s\|^8_{L^{\infty}}\mathd s\)^{\frac{1}{8}}+C\int_{0}^{T}\|h^j_s\|_{L^{\infty}}\mathd s\bigg].
 \end{align*}
 Since there exists a positive deterministic constant $C$ such that
 \begin{align*}
 	&\int_{\mathbb{T}^{2N}}\rho_t^{N}\log \rho_t^{N} \mathd x^N>-C,\quad \forall t\in[0,T],\quad \tilde{\mathbb{P}}-a.s.,
 \\	&\int_{\mathbb{T}^{2N}}\rho_t^{\eps,N}\log \rho_t^{\eps,N} \mathd x^N>-C,\quad \forall\eps>0 ,\quad \forall t\in[0,T],\quad \tilde{\mathbb{P}}-a.s..
 \end{align*}
 (see Lemma 3.1 in \cite{hauray2014kac}),
 by  \eqref{eqt:scaentr} and  \eqref{est:unifishereps} in Lemma \ref{lem:entropy solution} with $C(t)=0, \forall t\in[0,T],$ we have
 \begin{align*}
 	\bigg|\int_{0}^{t}J_4^{\eps,N}(s)\mathd s-\int_{0}^{t}J_4^{N}(s)\mathd s\bigg|&\leq  \sum_{i=1}^{N}\sum_{j\neq i}\int_{0}^{t}\bigg|\int_{\mathbb{T}^{2N}}\nabla^{\perp}G_{\delta}\cdot h^j_s(\rho^{\eps,N}_s -\rho^{N}_s )\mathd x^N\bigg|\mathd s\\&+\frac{C_{\rho_0}}{N}\sum_{i=1}^{N}\sum_{j\neq i}\bigg[\(\int_{0}^{T}\|h^j_s\|^8_{L^{\infty}}\mathd s\)^{\frac{1}{8}}+\int_{0}^{T}\|h^j_s\|_{L^{\infty}}\mathd s\bigg](\eps^{\frac{1}{2}}+\delta^{\frac{1}{2}})\\&+\frac{C_{\rho_0}}{N}\sum_{i=1}^{N}\sum_{j\neq i}\bigg[\(\int_{0}^{T}\|h^j_s\|^8_{L^{\infty}}\mathd s\)^{\frac{1}{8}}+\int_{0}^{T}\|h^j_s\|_{L^{\infty}}\mathd s\bigg]\delta^{\frac{1}{2}},
 \end{align*}
where $C_{\rho_0}$ is a positive deterministic constant depends on $\rho_0.$
Let $\eps \rightarrow 0,$  
we can deduce that the first term converges to $0$ by the similar method employed in proving $\lim_{\eps \rightarrow 0}\int_{0}^{t}J_1^{\eps,N}(s)\mathd s=\int_{0}^{t}J_1^{N}(s)\mathd s.$
We then  have
\begin{align*}
	&\limsup_{\eps \rightarrow 0}\bigg|\int_{0}^{t}J_4^{\eps,N}(s)\mathd s-\int_{0}^{t}J_4^{N}(s)\mathd s\bigg|  \leq\frac{C_{\rho_0}}{N}\sum_{i=1}^{N}\sum_{j\neq i}\bigg[\(\int_{0}^{T}\|h^j_s\|^8_{L^{\infty}}\mathd s\)^{\frac{1}{8}}+\int_{0}^{T}\|h^j_s\|_{L^{\infty}}\mathd s\bigg]\delta^{\frac{1}{2}}.
\end{align*}
Recall that $\int_{0}^{T}	\|h^j_s\|_{L^{\infty}}\mathd s$ and $\int_{0}^{T}	\|h^j_s\|_{L^{\infty}}^8\mathd s$ are finite $\tilde{\mathbb{P}}-$almost surely, by the arbitrariness of $\delta$, we finally have \begin{align*}
	\lim_{\eps \rightarrow 0}\int_{0}^{t}J_4^{\eps,N}(s)\mathd s=\int_{0}^{t}J_4^{N}(s)\mathd s,\quad \forall t\in[0,T],\quad\tilde{\mathbb{P}}-a.s..
\end{align*} The proof is then  completed.
\end{proof}

	\begin{proof}[Proof of Theorem \ref{thm:entropy}]
		The proof consists of three steps. 
	
	Step 1: Based on Lemma \ref{lem:entropy solution} and Lemma \ref{lem:entro-2}, we  deduce a rudimentary control of the evolution of relative entropy in this step, i.e. \eqref{inq:reentropy1}, through another formulation of Biot-Savart law and \eqref{con:Frho}. 
	
	 As we mentioned before, we firstly deal with  $H_N(\rho_t^N|\bar{\rho}_t^{ N})$.
			Recall that 
	\begin{align*}
		H_N(\rho_t^N|\bar{\rho}_t^{ N}) = \frac{1}{N}\int_{\mathbb{T}^{2N}} \rho^N_t \log \rho^N_t \mathd x^N - \frac{1}{N}\int_{\mathbb{T}^{2N}}\rho^N_t \log \bar{\rho}_t^N \mathd x^N,
	\end{align*} 
	where the random measures $( \rho_t^{N})_{t\in[0,T]}$ and $( \bar{\rho}_t^{N})_{t\in[0,T]}$  on $\mathbb{T}^{2N}$ are defined in \eqref{eqt:disparti} and \eqref{eqt:disinmean}.
Utilizing the estimates for the first term $\frac{1}{N}\int_{\mathbb{T}^{2N}} \rho^N_t \log \rho^N_t  \mathrm{d} x^N$ in Lemma \ref{lem:entropy solution} and the second term $\frac{1}{N}\int_{\mathbb{T}^{2N}}\rho^N_t \log \bar{\rho}_t^N  \mathrm{d} x^N$ in  Lemma \ref{lem:entro-2},  we can deduce an estimate for the scaled relative entropy $e^{-\int C(\cdot)}H_N(\rho^N|\bar{\rho}^N)$ for any nonnegative  $(\mathcal{F}_t)_{t\in[0,T]}$-progressive measurable process $C(\cdot)\in L^1(\Omega\times[0,T]):$
\begin{align*}
	& \mathd \left( 	e^{- \int_0^{t}C(s)\mathd s}	H_N(\rho_t^N|\bar{\rho}_t^{ N})\right) 
	\\=& \mathd \left( 	e^{- \int_0^{t}C(s)\mathd s}\frac{1}{N}\int_{\mathbb{T}^{2N}} \rho^N_t \log \rho^N_t \mathd x^N \right) - \mathd \left( 	e^{- \int_0^{t}C(s)\mathd s}\frac{1}{N}\int_{\mathbb{T}^{2N}} \rho^N_t \log \bar{\rho}_t^{ N} \mathd x^N \right) 
	\\\leq &- 	C(t)e^{- \int_0^{t}C(s)\mathd s}  	H(\rho_t^N|\bar{\rho}_t^{N}) \mathd t  - 	e^{- \int_0^{t}C(s)\mathd s} (J_1^N(t)+J_2^N(t)) \mathd t,
\end{align*}
where \begin{align*}
	J_1^N(t)=&\sum_{i=1}^N\int_{\mathbb{T}^{2N}} \left(  \frac{|\nabla_{x_i}\rho_t^{N}|^2}{\rho^{N}_t} +\rho^N_t(x^N) \Delta_{x_i}\log \bar{\rho}_t^N (x^N)+ \rho_t^N\frac{\Delta \bar{\rho}_t(x_i)}{\bar{\rho}_t(x_i)}\right) \mathd x^N ,
	\\J_2^N(t)=&	\frac{1}{N}\sum_{i=1}^{N}\sum_{j\neq i}\int_{\mathbb{T}^{2N}}\left(\rho^{N}_t (x^N)z_jK(x_i-x_j)\right)\cdot\nabla_{x_i} \log \bar{\rho}_t^N(x^N) \mathd x^N
	\\&-\sum_{i=1}^N \int_{\mathbb{T}^{2N}} \rho^N_t(x^N) K*  v_t(x_i)\cdot \nabla\log \bar{\rho}_t(x_i) \mathd x^N. 
\end{align*}

In Section \ref{sec:sto vortex}, we introduced the Biot-Savart law $K$, which can be expressed as
$$K=\nabla^{\perp} G.$$ Here $G$ represents the Green function induced by the Laplace operator on  torus $\mathbb{T}^2$.  We now introduce another formulation of Biot-Savart law  given in  \cite{jabin2018quantitative} , \begin{align*}
	K =\div V, \quad V=\begin{pmatrix}
		-\phi \arctan\frac{x_1}{x_2}+ \psi_1&0\\0 &\phi\arctan\frac{x_2}{x_1}+\psi_2
	\end{pmatrix},
\end{align*}which means that for each component $\alpha\in\{1,2\}$, we have $K_{\alpha}=\sum_{\beta}\partial_{\beta}V_{\alpha,\beta}, \beta=1,2.$ Here $\phi $ is a smooth function with compact support in torus $\mathbb{T}^2=[-\pi,\pi]^2$ and $(\psi_1,\psi_2)$ is a corrseponding smooth correction to periodize $V.$ We notice that $2\times2$ matrix  $V=(V_{\alpha\beta})$ satisfying $V_{\alpha\beta}\in L^{\infty},\forall \alpha,\beta=1,2.$ 
 Following the approach  in the proof \footnote{The only distinction between the two cases lies in the intensities $z^N=(z_1, \ldots, z_N)\in \mathbb{R}^N$. However, there is no fundamental disparity between them because our calculations are conducted on $\mathbb{T}^2$ at this stage, and the intensities $z^N=(z_1, \ldots, z_N)$  thus can be considered constants during this process.} 
 of in Theorem 1 in \cite{jabin2018quantitative}, we has the following estimate
 \begin{align*}
 	J_1^N(t)+J_2^N(t)\leq &\frac{1}{N}\sum_{i=1}^N \sum_{\alpha,\beta=1,2}\int_{\mathbb{T}^{2N}}\rho_t^N(x^N) \frac{|\nabla\bar{\rho}_t(x_i) |^2}{(\bar{\rho}_t(x_i))^2} \left| \frac{1}{N} \sum_{j=1}^N\varphi_{\alpha\beta,t}((x_i,z_i),(x_j,z_j))\right| ^2 \mathd x^N
 	\\&  + \frac{1}{N^2}\sum_{i,j=1}^N\int_{\mathbb{T}^{2N}} \rho_t^N(x^N)\varphi_{2,t}((x_i,z_i),(x_j,z_j))\mathd x^N,
 \end{align*}
where 
\begin{align*}
	&\varphi_{\alpha\beta,t}((x_i,z_i),(x_j,z_j))\assign z_jV_{\alpha \beta}(x_i-x_j)-V_{\alpha \beta}*  v_t(x_i),
	\\&\varphi_{2,t}((x_i,z_i),(x_j,z_j))\assign z_jV(x_i-x_j)-V*  v_t(x_i): \frac{\nabla^{2} \bar{\rho}_t(x_i)}{\bar{\rho}_t(x_i)},
\end{align*}
 for $ (x_i,z_i)\in \mathbb{T}^2\times\mathbb{R}, (x_j,z_j)\in \mathbb{T}^2\times\mathbb{R}.$
 
  As a result, we obtain the following estimate
  \begin{align*}
  	& \mathd \left( 	e^{- \int_0^{t}C(s)\mathd s}	H_N(\rho_t^N|\bar{\rho}_t^{N})\right) 
  	\\\leq &- 	C(t)e^{- \int_0^{t}C(s)\mathd s}  	H_N(\rho_t^N|\bar{\rho}_t^{ N}) \mathd t  
  	\\&+ e^{- \int_0^{t}C(s)\mathd s}  \frac{1}{N}\sum_{i=1}^N \sum_{\alpha,\beta=1,2}\int_{\mathbb{T}^{2N}}\rho_t^N(x^N) \frac{|\nabla\bar{\rho}_t(x_i) |^2}{(\bar{\rho}_t(x_i))^2} \left| \frac{1}{N} \sum_{j=1}^N\varphi_{\alpha\beta,t}((x_i,z_i),(x_j,z_j))\right| ^2 \mathd x^N
  	\\&  +e^{- \int_0^{t}C(s)\mathd s}  \frac{1}{N^2}\sum_{i,j=1}^N\int_{\mathbb{T}^{2N}} \rho_t^N(x^N)\varphi_{2,t}((x_i,z_i),(x_j,z_j))\mathd x^N.
  \end{align*}

Furthermore,  by \eqref{con:Frho}, we  integrate the above inequality with respect to $\mathcal{L}({\xi}^{ N})$ and deduce that
 \begin{align*}
	& \mathd \left( 	e^{- \int_0^{t}C(s)\mathd s}	H_N(F_t^N|\bar{F}_t^{ N})\right) 
	\\\leq &- 	C(t)e^{- \int_0^{t}C(s)\mathd s}  	H_N(F_t^N|\bar{F}_t^{ N}) \mathd t  
	\\&+ e^{- \int_0^{t}C(s)\mathd s}  \frac{1}{N}\sum_{i=1}^N \sum_{\alpha,\beta=1,2}\int_{\mathbb{T}^{2N}} \frac{|\nabla\bar{\rho}_t(x_i) |^2}{(\bar{\rho}_t(x_i))^2} \left| \frac{1}{N} \sum_{j=1}^N\varphi_{\alpha\beta,t}((x_i,z_i),(x_j,z_j))\right| ^2F_t^N(\mathd x^N,\mathd z^N)
	\\&  +e^{- \int_0^{t}C(s)\mathd s}  \frac{1}{N^2}\sum_{i,j=1}^N\int_{\mathbb{T}^{2N}} \varphi_{2,t}((x_i,z_i),(x_j,z_j))F_t^N(\mathd x^N,\mathd z^N).
\end{align*}
Since $\varphi_{2,t}=0$ when $\|\nabla^2\bar{\rho}_t\|_{L^{\infty}(\mathbb{T}^2)}=0$ and \eqref{ptsobol1}, we then have
\begin{align}
	 \mathd \left( 	e^{- \int_0^{t}C(s)\mathd s}	H_N(F_t^N|\bar{F}_t^{ N})\right)&\leq  -	C(t)e^{- \int_0^{t}C(s)\mathd s}	H_N(F_t^N|\bar{F}_t^{N})\mathd t\nonumber\\&+ e^{- \int_0^{t}C(s)\mathd s}\left( \mathcal{A}^1_N(t)+\mathcal{A}^2_N(t)\right) \mathd t 
	,\label{inq:reentropy1}
\end{align}
where $\mathcal{A}_N^i$, $i=1,2,$ 
  are defined by 
\begin{align*}
		\mathcal{A}^1_N(t)\assign&	\frac{\|\nabla \bar{\rho}_t\|_{L^{\infty}(\mathbb{T}^2)}^2}{\underset{x\in \mathbb{T}^2 }{\inf}\bar{\rho}_0^2N} \sum_{\alpha,\beta=1,2}\sum_{i=1}^N\int_{\mathbb{R}^N} \int_{\mathbb{T}^{2N}}  \left| \frac{1}{N} \sum_{j=1}^N\varphi_{\alpha\beta,t}((x_i,z_i),(x_j,z_j))\right| ^2 F_t^N(\mathd x^N,\mathd z^N),
	\\ \mathcal{A}^2_N(t)\assign&	I_{\{\|\nabla^2\bar{\rho}_t\|_{L^{\infty}(\mathbb{T}^2)}\neq 0\}}(\omega)
\frac{1}{N^2}\sum_{i,j=1}^N\int_{\mathbb{R}^N}\int_{\mathbb{T}^{2N}} \varphi_{2,t}((x_i,z_i),(x_j,z_j))F_t^N(\mathd x^N,\mathd z^N) .
\end{align*}

Step 2: In the second step,  we aim to estimate $\mathcal{A}^1_N(t)$ and $\mathcal{A}^2_N(t).$  To complete this, we replace the problem of controlling integrals with respect to  the random measure $F_t^N(\mathd x^N,\mathd z^N)$ on $\mathbb{T}^{2N}\times \mathbb{R}^N$ into a problem of controlling exponential integrals with repsect to the random measure $\bar{F}_t^N(\mathd x^N,\mathd z^N)$ on $\mathbb{T}^{2N}\times \mathbb{R}^N$ by Lemma \ref{lemma jw0}. Furthermore, through  Lemma \ref{lemma jw1} and Lemma \ref{lemma jw2}, we finally get the inequality \eqref{inq:laplace}.


We first notice that \begin{align}\label{V}
	\|V_{\alpha \beta}*  v_t\|_{L^{\infty}(\mathbb{T}^2)}\leq\|  v_{t}\|_{L^1(\mathbb{T}^2)}\|V_{\alpha\beta}\|_{L^{\infty}(\mathbb{T}^2)}\lesssim\|  v_{t}\|_{L^2(\mathbb{T}^2)}\|V_{\alpha\beta}\|_{L^{\infty}(\mathbb{T}^2)}\leq\|  v_{0}\|_{L^2}\|V\|_{L^{\infty}(\mathbb{T}^2)},
\end{align}
where we use H\"older inequality and compactness of $\mathbb{T}^{2N}$  to get the second inequality, \eqref{regu-1} to get the final inequality.
In addition, since $\mathcal{L}(\xi)$ has a density with compact support, i.e, there exists $M>0$ such that $\mathcal{L}(\xi)([-M,M])=1.$
 We then find that  there exists a positive deterministic constant $\eta_1$ depending on $\|v_0\|_{L^2(\mathbb{T}^2)}$ and $\mathcal{L}(\xi)$ such that $$\eta_1 \|\varphi^1_{\alpha\beta,t}\|_{L^\infty(\mathbb{T}^{2}\times\mathbb{R})} < \frac{1}{2e},$$ for every $\alpha,\beta=1,2, t\in[0,T] ,$
 where \begin{align*}
 	&\varphi^1_{\alpha\beta,t}((x_i,z_i),(x_j,z_j))\assign \(z_jV_{\alpha \beta}(x_i-x_j)-V_{\alpha \beta}*  v_t(x_i)\)I_{[-M,M]}(z_j), 
 \end{align*}
 for $(x_i,z_i)\in \mathbb{T}^2\times\mathbb{R}, (x_j,z_j)\in \mathbb{T}^2\times\mathbb{R}.$
 
 We then define  
 \begin{align*}
 	\Phi^1_{\alpha\beta,t}(x^N,z^N)\assign \left| \frac{1}{N} \sum_{j=1}^N\varphi^1_{\alpha\beta,t}((x_i,z_i),(x_j,z_j))\right| ^2\in L^{\infty}(\mathbb{T}^{2N}\times\mathbb{R}^N),
 \end{align*}
 and have
\begin{align}
		\mathcal{A}^1_N(t)&=	\frac{\|\nabla \bar{\rho}_t\|_{L^{\infty}(\mathbb{T}^2)}^2}{\underset{x\in \mathbb{T}^2 }{\inf}\bar{\rho}_0^2N} \sum_{\alpha,\beta=1,2}\sum_{i=1}^N\int_{\mathbb{R}^N} \int_{\mathbb{T}^{2N}}  \Phi^1_{\alpha\beta,t}(x^N,z^N) F_t^N(\mathd x^N,\mathd z^N),\nonumber
		\\&\leq \frac{\|\nabla\bar{\rho}_t\|_{L^{\infty}(\mathbb{T}^2)}^2}{\underset{x\in \mathbb{T}^2 }{\inf}\bar{\rho}_0^2}\frac{1}{\eta_1}\sum_{\alpha,\beta=1,2} H_N(F^N_t|\bar{F}^N_t)\nonumber
	+\frac{\|\nabla\bar{\rho}_t\|_{L^{\infty}(\mathbb{T}^2)}^2}{\underset{x\in \mathbb{T}^2 }{\inf}\bar{\rho}_0^2}\frac{1}{\eta_1N}\\&\sum_{\alpha,\beta=1,2}\log \int_{\mathbb{R}^N}\int_{\mathbb{T}^{2N}}   \exp \left(  \frac{\eta_1}{N} \sum_{j_1,j_2=1}\varphi^1_{\alpha\beta,t}((x_1,z_1),(x_{j_1},z_{j_1}))\varphi^1_{\alpha\beta,t}((x_1,z_1),(x_{j_2},z_{j_2})) \right)\bar{F}_t^N(\mathd x^N, \mathd z^N),\nonumber
\end{align}
where we use the Lemma \ref{lemma jw0} and the symmetry of $\bar{F}^N_t$ to get the second inequality.

Notice that  
\begin{align*}
&\int_{\mathbb{R}}\int_{\mathbb{T}^{2}}\varphi^1_{\alpha\beta,t}((x_i,z_i),(x_j,z_j))\bar{F}_t(\mathd x_j, \mathd z_j)	
	\\=&\int_{\mathbb{R}}\int_{\mathbb{T}^{2}}	 (z_jV_{\alpha \beta}(x_i-x_j) -V_{\alpha \beta}*  v_t(x_i))\bar{F}_t(\mathd x_j, \mathd z_j)=0,
\end{align*}
by Lemma \ref{lemma jw1} with $\varphi_{\alpha\beta}$ playing the role of $\phi,$  we have 
\begin{align}
\mathcal{A}^1_N(t)\leq C_1\|\nabla \bar{\rho}_t\|_{L^{\infty}(\mathbb{T}^2)}^2 H_N(F^N_t|\bar{F}^N_t)+ \frac{C_1}{N}\|\nabla \bar{\rho}_t\|_{L^{\infty}(\mathbb{T}^2)}^2,\label{911}
\end{align}
where the random measure $
	\bar{F}_t(\mathd x,\mathd z)\assign \frac{1}{\mathbb{E}[\xi_1]}  v_t(\mathd x)\mathcal{L}(\xi)(\mathd z)$ and $C_1$ is a positive deterministic constant depending on $\|v_0\|_{L^2(\mathbb{T}^2)},\underset{x\in \mathbb{T}^2 }{\inf}\bar{\rho}_0$ and $\mathcal{L}(\xi).$


By using  Lemma \ref{lemma jw0} similarly, we have 
\begin{align*}
		\mathcal{A}^2_N(t)&\leq 	I_{\{\|\nabla^2\bar{\rho}_t\|_{L^{\infty}(\mathbb{T}^2)}\neq 0\}}(\omega)\frac{\|\nabla^2\bar{\rho}_t\|_{L^{\infty}(\mathbb{T}^2)}}{ \underset{x\in \mathbb{T}^2 }{\inf}\bar{\rho}_0}\frac{1}{\eta_2}
	H_N(F^N_t|\bar{F}^N_t)\nonumber+			I_{\{\|\nabla^2\bar{\rho}_t\|_{L^{\infty}(\mathbb{T}^2)}\neq 0\}}(\omega)	\frac{\|\nabla^2\bar{\rho}_t\|_{L^{\infty}(\mathbb{T}^2)}}{ \underset{x\in \mathbb{T}^2 }{\inf}\bar{\rho}_0 }\frac{1}{\eta_2N}\\&\log \int_{\mathbb{R}^N}\int_{\mathbb{T}^{2N}}  \exp \left(\frac{\eta_2}{N}   \sum_{i,j=1}^N\varphi^2_t\((x_i,z_i),(x_j,z_j)\)\right)\bar{F}_t^N(\mathd x^N,\mathd z^N), \nonumber
\end{align*}
where \begin{align}
	\varphi^2_t\((x_i,z_i),(x_j,z_j)\)\assign \frac{\underset{x\in \mathbb{T}^2 }{\inf}\bar{\rho}_0}{\|\nabla^{2} \bar{\rho}_t\|_{L^{\infty}(\mathbb{T}^2)}}I_{[-M,M]}(z_j)\(z_jV(x_i-x_j)-V*  v_t(x_i): \frac{\nabla^{2} \bar{\rho}_t(x_i)}{\bar{\rho}_t(x_i)}\),\nonumber
\end{align} 
for $ (x_i,z_i)\in \mathbb{T}^2\times\mathbb{R}, (x_j,z_j)\in \mathbb{T}^2\times\mathbb{R},$
and $\eta_2$ is a positive deterministic constant depending on $\|v_0\|_{L^2(\mathbb{T}^2)}$ and $\mathcal{L}(\xi),$ such that $$\tilde{C}\eta_2^2\|\varphi^{2}_t\|^2_{L^\infty(\mathbb{T}^2)} < \frac{1}{2}.$$ where $\tilde{C}$ is the universal constant in Lemma \ref{lemma jw2}.

The cancellation with respect to $(x_j,z_j)$ for $\varphi^2_t\((x_i,z_i),(x_j,z_j)\)$ is similar to $\varphi_{\alpha\beta}$, i.e.
\begin{align*}
	&\int_{\mathbb{R}}\int_{\mathbb{T}^{2}}\varphi^2_{t}((x_i,z_i),(x_j,z_j))\bar{F}_t(\mathd x_j, \mathd z_j)	=0,
\end{align*} while the cancellation with respect to $(x_i,z_i)$ follows by 
\begin{align}
	&\int_{\mathbb{R}}\int_{\mathbb{T}^{2}}\varphi^2_{t}((x_i,z_i),(x_j,z_j))\bar{F}_t(\mathd x_i, \mathd z_i)\nonumber
	\\=&\int_{\mathbb{R}}\int_{\mathbb{T}^{2}} \frac{\underset{x\in \mathbb{T}^2 }{\inf}\bar{\rho}_0}{\|\nabla^{2} \bar{\rho}_t\|_{L^{\infty}(\mathbb{T}^2)}}I_{[-M,M]}(z_j)\left(z_jV(x_i-x_j)-V*  v_t(x_i): \frac{\nabla^{2} \bar{\rho}_t(x_i)}{\bar{\rho}_t(x_i)}\right) \bar{\rho}_t(x_i)  \mathd x_i\mathcal{L}(\xi)(\mathd z_i)\nonumber
	\\=&\int_{\mathbb{R}}\int_{\mathbb{T}^{2}}\frac{\underset{x\in \mathbb{T}^2 }{\inf}\bar{\rho}_0}{\|\nabla^{2} \bar{\rho}_t\|_{L^{\infty}(\mathbb{T}^2)}}I_{[-M,M]}(z_j)\left(z_jV(x_i-x_j)-V*  v_t(x_i): {\nabla^{2} \bar{\rho}_t(x_i)}\right)\mathd x_i\mathcal{L}(\xi)(\mathd z_i)\nonumber
	\\=&\int_{\mathbb{R}}\int_{\mathbb{T}^{2}}\frac{\underset{x\in \mathbb{T}^2 }{\inf}\bar{\rho}_0}{\|\nabla^{2} \bar{\rho}_t\|_{L^{\infty}(\mathbb{T}^2)}}I_{[-M,M]}(z_j)\left(z_j\div K_{x_i}(x_i-x_j)-\div_{x_i} K*  v_t(x_i) \right)\bar{\rho}_t(x_i)\mathd x_i\mathcal{L}(\xi)(\mathd z_i)=0\nonumber,
\end{align}
where we  use integration by parts to get the third equality and use $\div K=0$ to get the final equality.  Using Lemma \ref{lemma jw2} with $\varphi^2_{t}$ playing the role of $\phi,$ we find 
\begin{align}
		\mathcal{A}^2_N(t)&\leq I_{\{\|\nabla^2\bar{\rho}_t\|_{L^{\infty}(\mathbb{T}^2)}\neq 0\}}(\omega)\(C_2\|\nabla^2\bar{\rho}_t\|_{L^{\infty}(\mathbb{T}^2)}H_N(F^N_t|\bar{F}^N_t)+\frac{C_2}{N}\|\nabla^2\bar{\rho}_t\|_{L^{\infty}(\mathbb{T}^2)}\)\nonumber
	\\&	\leq C_2\|\nabla^2\bar{\rho}_t\|_{L^{\infty}(\mathbb{T}^2)}H_N(F^N_t|\bar{F}^N_t)+\frac{C_2}{N}\|\nabla^2\bar{\rho}_t\|_{L^{\infty}(\mathbb{T}^2)},\label{911911}
\end{align}
where $C_2$ is a positive deterninsitic constant depending on $\|v_0\|_{L^2(\mathbb{T}^2)},\underset{x\in \mathbb{T}^2 }{\inf}\bar{\rho}_0$ and $\mathcal{L}(\xi).$

Step 3: Through some elementary calculations, we finally get the final result in this step. 

Now we set $C(t)= C_1 \|\nabla\bar{\rho}_t\|_{L^{\infty}}^2   +C_2\|\nabla^2\bar{\rho}_t\|_{L^{\infty}}$, where $C_1, C_2$ are the universal constants in \eqref{911} and \eqref{911911} respectively. We then derive from  \eqref{inq:reentropy1} that 
	\begin{align}
	&	e^{- \int_0^{t}C(s)\mathd s}	H_N(F_t^N|\bar{F}_t^{N})-  H_N(F_0^N|\bar{F}^N_0)\nonumber
	\\\leq & -	\int_0^t C(s)e^{- \int_0^{s}C(r)\mathd r}	H_N(F_s^N|\bar{F}_s^{ N})\mathd s+  \int_0^t\frac{C(s) }{N}e^{- \int_0^{s}C(r)\mathd r}\mathd s\nonumber	\\&+ \int_0^t C(s)e^{- \int_0^{s}C(r)\mathd r}	H_N(F_s^N|\bar{F}_s^{N})\mathd s \nonumber
	\\\leq &  \int_0^t\frac{C(s) }{N}e^{- \int_0^{s}C(r)\mathd r}\mathd s.\label{inq:laplace}
\end{align}
Since $F_0^N(\mathd x^N,\mathd z^N)=\bar{F}^N_0(\mathd x^N,\mathd z^N)=\bar{\rho_0}(\mathd x^N)\mathcal{L}(\xi^N)(\mathd z^N),$
we have $$H_N(F_0^N|\bar{F}^N_0)=0.$$
By Sobolev  embedding theorem and the fact that $\bar{\rho}_t=\frac{1}{\mathbb{E}[\xi_1]}  v_t,   \forall t\in [0,T],$ we then have
\begin{align*}
	\sup_{[0, T]}H_N(F_t^N|\bar{F}^N_t)\leq \frac{1}{N}\exp\{C_{\xi,0}\int_{0}^{T}(\|  v_t\|_{H^4}^2 +1)\mathd t\}\quad\mathbb{P}-a.s.
\end{align*}
where $C_{\xi,0}$ is a positive deterministic constant depending on $\|v_0\|_{L^2(\mathbb{T}^2)},\underset{x\in \mathbb{T}^2 }{\inf}\bar{\rho}_0$ and $\mathcal{L}(\xi).$
 The proof is then completed.
\end{proof}

\section{Stochastic 2-dimensional Navier-Stokes  equation}\label{sec:ns}
In this section, we begin with establishing the well-posedness of  stochastic 2-dimensional Navier-Stokes  equation \eqref{eqt:ns},  as stated in Theorem \ref{thm:spde}. The proof follows a standard  martingale method in the literature on stochastic partial differential equations, e.g.   \cite{dabrock2021existence}. A distinctive feature of our approach is the meticulous utilization of the properties of  Biot-Savart law $K$ when deriving uniform estimates. More precisely, we emphasize the representation
 $K=\div V$ mentioned in Section \ref{sec:relaen}, i.e. $$K_{\alpha}=\sum_{\beta}\partial_{\beta}V_{\alpha,\beta},\alpha,\beta=1,2,$$ where $V$ is a $2\times2$ matrix and every element $V_{\alpha\beta},\alpha,\beta=1,2$ in matrix $V$ belongs to $L^{\infty}.$ For brevity, we omit the proof of Theorem \ref{thm:spde} through the standard martingale argument and focus on presenting the derivation of uniform estimates. 

Moving forward, we will  provide the proof for the regularity of the probabilistically weak solution to stochastic 2-dimensional Navier-Stokes equation \eqref{eqt:ns}, as stated in Lemma \ref{lem:reg}.

\begin{proof}[Proof of Theorem \ref{thm:spde}]
	To begin with, we  establish  existence of the probabilistically weak solution for the stochastic 2-dimensional Navier-Stokes  equation \eqref{eqt:ns}. For fixed $\eps>0,$   given independent 1-dimensional Brownian motions $\{W^{k},k=1,\cdots ,d\}$ on torus $\mathbb{T}^2$ on  probability space $(\Omega,\mathcal{F},(\mathcal{F}_t)_{t\in [0,T]},\mathbb{P}),$  we start with considering the  solution to the regularized stochastic 2-dimensional Navier-Stokes  equation with the regularized initial data $  v_0^  \eps,$ 
	  \begin{equation}\label{eqt:rens}
		\mathd   v=\(\Delta   v-K_{\eps }*  v \cdot \nabla   v\)\mathd t-\sum_{k=1}^{d}\(\sigma_k\cdot\nabla   v \)\circ \mathd W^{k}_t  , \quad (t,x)\in [0,T]\times \mathbb{T}^2,  \quad   v(0,\cdot)=  v_0^\eps.
	\end{equation} 
where $K_{\eps}\assign \div V^{\eps}=\div (V*J_{\eps})$ and $ v_0^{  \eps}\assign  v_0*J_{\eps}.$ \footnote{ Here $J_{\eps}\assign \eps^{-2}J(\eps^{-1}x),$ where $J$ be a nonnegative $C_{c}^{\infty}(\mathbb{R}^2)$ function satisfying $\int_{\mathbb{R}^2} J(x)\mathd x=1$ and $J(x)=0$ if $|x|\geq 1.$ The convolution here should be understood on the whole $\mathbb{R}^2,$ where $V $ (or $v_0$) is extended by periodicity. }

Let $B$ be a standard 2-dimensional Brownian motion on torus $\mathbb{T}^2,$ constructed on an auxilary probability space $(\Omega_0,\mathcal{F}^0,\mathbb{P}^0).$  The probability space  $(\bar{\Omega},\bar{\mathcal{F}},\bar{\mathbb{P}})$ is defined as follows: $$(\bar{\Omega},\bar{\mathcal{F}},\bar{\mathbb{P}})\assign (\Omega_0\times\Omega,\mathcal{F}^0\times\mathcal{F},\mathbb{P}^0\times\mathbb{P}),$$  and let $(\bar{\mathcal{F}}_t)_{t\in[0,T]}$ denote the normal filtration  generated by Brownian motions $B$ and $\{W_1,\cdots,W_d\}.$
The basic idea is to consider the following SDE \eqref{stoflow} on $(\bar{\Omega},\bar{\mathcal{F}},\bar{\mathbb{P}})$ and we can construct a solution to \eqref{eqt:rens} through the solution to \eqref{stoflow}. More precisely,  as presented  in \cite{brzezniak2016existence}, we can construct a measurable map $p^{  \eps}(t,\bar{\omega},x):[0,T]\times\bar{\Omega}\times\mathbb{T}^2\rightarrow\mathbb{T}^2$ (for simplicity, we use the notation $p^{  \eps}_t(x)$ in following content,) satisfying that
\begin{enumerate}
	\item $p^{  \eps}$ is measurable with respect to $\bar{\mathcal{P}}\otimes\mathbb{B}(\mathbb{T}^2),$ where $\bar{\mathcal{P}}$ is the predictable $\sigma$-algebra assssociated with the filtration $(\bar{\mathcal{F}}_t)_{t\in[0,T]}.$
	\item It holds $\bar{\mathbb{P}}$-almost surely that for all $t\in[0,T],$ $p^{  \eps}_t(x)$ is a homeomorphism on torus $\mathbb{T}^2$ and we denote $(p^{  \eps}_t)^{-1}(x)$ denote the inverse map of  $p^{  \eps}_t(x).$
	\item   It holds $\bar{\mathbb{P}}$-almost surely that for all $t\in[0,T],$ $p^{  \eps}_t(x)$ is Lebesgue measure preserveing, i.e. for any integrable function $h$ on torus $\mathbb{T}^2,$ it holds
	\begin{align*}
		\int_{\mathbb{T}^2}h(p^{\eps}_t(x))\mathd x=	\int_{\mathbb{T}^2}h(y)\mathd y.
	\end{align*}
\item  For every $x\in\mathbb{T}^2$,  $X_t=X_t(x)\assign p^{  \eps}_{t}(x)$ solves the SDE 
\begin{align}\label{stoflow}
	X_t=x+ \sqrt{2}B_t+\sum_{k=1}^d\int_0^t\sigma_k(X_s)\circ\mathd W^{k}_s+\int_0^t u^{p }_{  \eps}(t,X_s)\mathd s,
\end{align}
where $$u^{p }_{  \eps}(t,z)\assign \int_{\Omega^0}\int_{\mathbb{T}^{2}}K_{  \eps}(z-p ^{\eps}_{t}(y))  v_0^{  \eps}(y)\mathd y\mathd \mathbb{P}^0, \quad z\in\mathbb{T}^2.$$ 
\end{enumerate}

We then define \begin{align}\label{constr}
	v^{  \eps}_t(x)\assign\int_{\Omega^0}  v_0^{  \eps}((p^{  \eps}_{t})^{-1}(x))\mathd \mathbb{P}^0,\quad\forall t \in[0,T], \quad \forall x\in\mathbb{T}^2.
\end{align}
and  notice that 
\begin{align}\label{bound}
	\|  v^{  \eps}_t\|_{L^{\infty}}\leq \|  v_0\|_{L^{\infty}} \quad \text{and} \quad \underset{x\in \mathbb{T}^2 }{\inf}   v^{  \eps}_t \geq  \underset{x\in \mathbb{T}^2 }{\inf}  v_0,,\quad\forall t \in[0,T].
\end{align}
By It\^o's formula, we deduce that 
	for all $\varphi\in C^{\infty}(\mathbb{T}^2),$  it holds almost surely that for all $t\in [0,T],$ 
	\begin{align}\label{eqregula}
		\left\langle  \varphi,  v^{  \eps}_t\right\rangle  =  \left\langle \varphi,   v_0^{  \eps}\right\rangle +\int_0^t \left\langle \Delta\varphi,   v^{  \eps}_s\right\rangle \mathd s +\int_0^t\left\langle  \nabla \varphi,  K_{  \eps}*  v^{  \eps}_s   v^{  \eps}_s\right\rangle \mathd s+ \sum_{k=1}^{d}\int_0^t  \left\langle \nabla \varphi,    v^{  \eps}_s\sigma_k\right\rangle\circ \mathd W_s^{k} .
	\end{align} 
Therefore, for fixed $\eps>0,$ we construct a solution to \eqref{eqt:rens}. 
Similar to \eqref{barrhot}, $(	v^{  \eps}_t)_{t\in[0,T]}$ is also a solution to linearized  regularized stochastic stochastic 2-dimensional Navier-Stokes  equation with  smooth initial data $  v_0^  \eps$ and smooth coefficients,
\begin{equation}\label{eqt:linearrens}
	\mathd   f=\(\Delta   f-K_{\eps }*  v^{\eps} \cdot \nabla   f\)\mathd t-\sum_{k=1}^{d}\(\sigma_k\cdot\nabla   f \)\circ \mathd W^{k}_t  
\end{equation} 
According to Theorem 5.1  and Theorem 7.1 in \cite{krylov1999analytic}, for fixed $\eps>0,$ we have 
  \begin{align}\label{regregv}
		(  v^{  \eps}(t))_{t\in [0,T]}\in \cap_{k\in \mathbb{N}}  C([0,T];C^k(\mathbb{T}^2)),\quad \mathbb{P}-a.s..
	\end{align} 

	  Based on the uniform estimates in Lemma  \ref{lem:uni01}-Lemma \ref{suph2} for $(v^{  \eps}(t))_{t\in [0,T]},  \eps>0,$ constructed above, 
	   existence of the probabilistically weak solution to  stochastic 2-dimensional Navier-Stokes  equation \eqref{eqt:ns} then follows by  the  standard martingale argument, as outlined in Theorem A.5 in \cite{dabrock2021existence}. 
	  
	 Applying the Yamada-Watanabe theorem \cite[Theorem 1.5]{kurtz2014yamada}, we can directly establish Theorem \ref{thm:spde} by demonstrating existence of probabilistically weak solution  and confirming pathwise uniqueness of solutions. The pathwise uniqueness for stochastic 2-dimensional Navier-Stokes equation \eqref{eqt:ns} can be attained using the same method employed in proving uniqueness for the linearized stochastic 2-dimensional Navier-Stokes equation \eqref{ptsobol3} in Lemma \ref{lem:linearunique}. Consequently, we omit the proof for  pathwise uniqueness here.
\end{proof}

Most of the technical proofs for the following lemmas concerning the uniform estimates rely on the same ideas. Therefore, we have given the
explicit details of some of these proofs, the proof of Lemma \ref{lem:uni01}, the proof of Lemma \ref{lem:uni12} and the proof of Lemma \ref{lem:12high}. In most of the other proofs, these details are then skipped.

Recall that we use $\mathbb{H}^j, j\in \mathbb{N},$ to denote the $L^2$-norm of $j$-th order derivatives for simplicity, i.e, 
\begin{align}
	\|f\|_{\mathbb{H}^j}\assign \|\nabla^{ j}f\|_{L^2(\mathbb{T}^2)}.\nonumber
\end{align}
Recall  that $\|f\|_{H^n}^2= \|f\|_{L^2}^2+ \sum_{j=1 }^{n}\|f\|_{\mathbb{H}^j}^2 , n\in \mathbb{N}$.  

\begin{lemma}\label{lem:uni01}
	For fixed $\eps\in(0,1),$ it holds $\mathbb{P}$-almost surely  that 
	\begin{align}
		\sup_{t\in [0,T]}	\|  v^{\eps}_t\|_{L^2}^2+2\int_0^T\| 	\|  v^{\eps}_t\|_{\mathbb{H}^1}^2\mathd t\leq 2\|  v_0\|_{L^2}^2.\nonumber
	\end{align}
\end{lemma}

\begin{proof}
	For fix $\eps\in(0,1),$
	through the regularity \eqref{regregv} of $(v^{  \eps}(t))_{t\in [0,T]},$ as in Lemma \ref{lem:reg}, we can obtain it holds $\mathbb{P}$-almost surely that  for all $t\in [0,T],x \in \mathbb{T}^2,$ 
		\begin{align}\label{x-regv}
			v_t^{\eps}(x) -   v_0^{\eps}(x) =&\int_0^t  \Delta  v^{\eps}_s(x)\mathd s -\int_0^t  K_{\eps}*  v^{\eps}_s(x) \cdot\nabla   v^{\eps}_s(x)\mathd s\nonumber\\&+\frac{1}{2}\sum_{k=1}^d\int_{0}^{t}\sigma_k\cdot \nabla \[\sigma_k\cdot \nabla   v^{\eps}_s(x)\]\mathd s- \sum_{k=1}^{d}\int_0^t  \sigma_k(x)\cdot\nabla   v^{\eps}_s(x) \mathd W_s^{k}.  
	\end{align}
	where the term $\frac{1}{2}\sum_{k=1}^d\int_{0}^{t}\sigma_k\cdot \nabla \[\sigma_k\cdot \nabla   v^{\eps}_s(x)\]\mathd s$ is tranformed from the Stratonovich integral.
	
	Using Ito's formula  to $(v_t^{\eps}(x))^2$ and integrating on $\mathbb{T}^2$  yields that 
	\begin{align}
		\mathd \|  v^{\eps}\|_{L^2}^2=&-2\|\nabla   v^{\eps}\|_{L^2}^2\mathd t-2\left\langle    v^{\eps}, K_{\eps}*  v^{\eps}\cdot \nabla   v^{\eps}\right\rangle \mathd t-2\sum_{k=1}^d\left\langle   v^{\eps} ,\sigma_k\cdot\nabla   v^{\eps}\right\rangle \mathd W_t^k,\quad \forall t\in [0,T].\nonumber
	\end{align}
	The divergence free  properties  of  $\{\sigma_k, k=1,\cdots, d\}$ and  $K_{\eps}$ 
	imply that 
	\begin{align}
		\left\langle    v^{\eps}, K_{\eps}*  v^{\eps}\cdot \nabla   v^{\eps}\right\rangle =-\left\langle  \nabla   v^{\eps}\cdot K_{\eps}*  v^{\eps},   v^{\eps}\right\rangle =0,\quad 
		\left\langle   v^{\eps},\sigma_k\cdot \nabla   v^{\eps}\right\rangle=- 	\left\langle \nabla   v^{\eps}\cdot \sigma_k,   v^{\eps}\right\rangle=0.\nonumber
	\end{align}
	The result then follows.
	\end{proof}
	
	\begin{lemma}\label{lem:uni12}
		For fixed $\eps\in(0,1),$ 	it holds $\mathbb{P}$-almost surely that for all $t\in[0,T],$
		\begin{align}
			\|  v^{\eps}_t\|_{\mathbb{H}^1}^2+\int_0^t\|  v^{\eps}_{s}\|_{\mathbb{H}^2}^2\mathd s \leq 	\|  v_0\|_{\mathbb{H}^1}^2+ C_{  v_0,\sigma} \int_0^t \|  v^{\eps}\|_{\mathbb{H}^1}^2 \mathd s+M_1(t), \quad\nonumber
		\end{align}
		where \begin{align*}
		M_1(t)=-2\sum_{k=1}^d \int_0^{t} \left\langle \nabla   v^{\eps}, \nabla \[\sigma_k \cdot \nabla   v^{\eps}\]\right\rangle dW^{k}_s
		\end{align*} is a  continuous local martingale with respect to the filtration $(\mathcal{F}^{W}_t)_{t\in[0,T]}$ generated by $W^k,k =1,\cdots,d.$
			Here $C_{  v_0,\sigma}$ denotes some positive constant depends on $\|  v_0\|_{{H}^3}$ and $\sigma_k, k=1,\cdots,d.$
		In particular, 
		\begin{align}\label{inq:lem12-1}
			\sup_{t\in [0,T]}\mathbb{E}\|  v^{\eps}_t\|_{\mathbb{H}^1}^2 +\mathbb{E}\int_0^T\|  v^{\eps}_t\|_{\mathbb{H}^2}^2\mathd t \leq C_{  v_0,\sigma}.\nonumber
		\end{align}
	\end{lemma}
	
	\begin{proof}
		As in Lemma \ref{lem:uni01}  ,we will establish an uniform estimate for $\|   v^{\eps}_t\|_{\mathbb{H}^1}^2$ and $\int_0^t\|  v^{\eps}_{s}\|_{\mathbb{H}^2}^2\mathd s$ by utilizing It\^o's formula and taking advantage of the divergence-free properties of $\{\sigma_k, k=1,\cdots,d\}$ and  $K_{\eps}$. 
		Based on \eqref{x-regv} and the regularity \eqref{regregv} of $(v^{  \eps}(t))_{t\in [0,T]},$ through \cite[Theorem 3.3.3]{kunita1997stochastic}, we conclude that for $i\in\{1,2\},$ we have 	\begin{align*}
			\partial_iv_t^{\eps}(x) -   \partial_iv_0^{\eps}(x) =&\int_0^t  \partial_i\Delta  v^{\eps}_s(x)\mathd s -\int_0^t  \partial_i(K_{\eps}*  v^{\eps}_s(x)\cdot\nabla   v^{\eps}_s(x)) \mathd s\nonumber\\&+\frac{1}{2}\sum_{k=1}^d\int_{0}^{t}\partial_i\(\sigma_k\cdot \nabla \[\sigma_k\cdot \nabla   v^{\eps}_s(x)\]\)\mathd s- \sum_{k=1}^{d}\int_0^t  \partial_i(\sigma_k(x)\cdot\nabla   v^{\eps}_s(x)) \mathd W_s^{k}.  
		\end{align*}
	Next, by applying It\^o's formula to obtain $\|\nabla  v^{  \eps}\|_{L^2}^2$, we obtain the following equation
		\begin{equation}
			\|\nabla  v^{\eps}_t\|_{L^2}^2= 	\|\nabla  v^{\eps}_0\|_{L^2}^2+2\sum_{l=1}^5J_l,\label{5.2-0}
		\end{equation} 
		where $J_l$, $l=1,...,5$, are  
		\begin{align*}
			J_1=& \int_0^t\left\langle \nabla  v^{\eps} ,\nabla( \Delta   v^{\eps})  \right\rangle \mathd s = -\int_0^t\left\langle \nabla^2  v^{\eps} ,\nabla^2 v^{\eps}  \right\rangle\mathd s=-\int_{0}^{t}\|  v^{\eps}_{s}\|_{\mathbb{H}^2}^2\mathd s,
			\\J_2=& -\int_0^t\left\langle \nabla   v^{\eps}, \nabla [K_{\eps}*  v^{\eps}\cdot \nabla   v^{\eps}] \right\rangle\mathd s ,
			\quad \quad \quad \quad \quad J_3=-\sum_{k=1}^d \int_0^t \left\langle \nabla   v^{\eps}, \nabla  \[\sigma_k \cdot \nabla   v^{\eps}\]\right\rangle dW^{k}_s,
			\\J_4=&\frac{1}{2}\sum_{k=1}^d\int_0^t\left\langle \nabla  v^{\eps}, \nabla  \(\sigma_k \cdot \nabla \[\sigma_k\cdot \nabla   v^{\eps}\] \)\right\rangle \mathd s,
			\quad J_5=\frac{1}{2}\sum_{k=1}^d\int_0^t\left\langle \nabla  \[\sigma_k \cdot \nabla   v^{\eps}\] ,\nabla  \[\sigma_k \cdot \nabla   v^{\eps}\]\right\rangle \mathd s .
		\end{align*}
		Here we used integration by parts to get the second equality in $J_1.$ 
		
		Recalling that $K_{  \eps}$ is in $L^1(\mathbb{T}^2)$ and $\sup_{[0, T]}\|  v^{  \eps}_t\|_{L^{\infty}}\leq \|  v^{  \eps}_0\|_{L^{\infty}},$ we can now analyze the second term $J_2$ that involves interactions
		\begin{align*}
			J_2=& \int_0^t\left\langle \Delta  v^{\eps}, K_{\eps}*  v^{\eps}\cdot \nabla   v^{\eps}\right\rangle\mathd s
			\\\leq & \int_0^t \|\Delta v^{\eps}\|_{L^2}\|\nabla   v^{\eps}\|_{L^2}\|K_{\eps}*  v^{\eps}\|_{L^{\infty}}\mathd s 
			\\\leq & \eta\int_0^t \|\Delta v^{\eps}\|_{L^2}^2\mathd s +C_{\eta}\|  v_0\|_{L^{\infty}}^2\int_0^t\|K\|_{L^1}^2 \|  v^{\eps}\|_{\mathbb{H}^1}^2 \mathd s 
			\\\leq & \eta\int_0^t \|\Delta  v^{\eps}\|_{L^2}^2\mathd s +C_{  v_0,\eta}\int_0^t \|  v^{\eps}\|_{\mathbb{H}^1}^2 \mathd s,
		\end{align*}
	where $C_{  v_0,\eta}$ is a positive constant depends on $\|  v_0\|_{H^3},\eta$ and 
	 we applied integration by parts once again to establish the first equality,  H\"older's inequality and \eqref{bound} to derive the second inequality, and  Young's inequality to obtain the third inequality here.
		
		Moving on to analyze $J_4$, we find that
		\begin{align*}
			2J_4=&\sum_{k=1}^d\int_0^t\left\langle  \nabla  v^{\eps},(\nabla\sigma_k)^T \nabla\[\sigma_k\cdot \nabla   v^{\eps}\] \right\rangle + \left\langle \nabla  v^{\eps},\nabla^2  \[\sigma_k\cdot \nabla   v^{\eps}\] \sigma_k \right\rangle \mathd s
			\\=&\sum_{k=1}^d\int_0^t\left\langle  \nabla  v^{\eps},(\nabla\sigma_k)^T \nabla\[\sigma_k\cdot \nabla   v^{\eps}\] \right\rangle -\left\langle \sigma_k, \nabla^2 v^{\eps} \nabla\[\sigma_k\cdot \nabla   v^{\eps}\] \right\rangle \mathd s
			=: J_{4,1}+J_{4,2}.
		\end{align*}
		Here we used $\div \sigma_k=0, k=1,\cdots,d$ to get the second term $J_{4,2}$ at the last line and $\nabla\sigma_k$ denote $\begin{pmatrix}
				\partial_{1} \sigma_k^1 & \partial_{2} \sigma_k^1 \\ \partial_{1} \sigma_k^2 &\partial_{2} \sigma_k^2
			\end{pmatrix},$ $(\nabla\sigma_k)^T$ denote the transpose of $\nabla\sigma_k.$ 
		
		Furthermore, by the chain rule, we can notice that 
		\begin{align*}
			J_{4,2}= -2J_5+\sum_{k=1}^d\int_0^t\left\langle \nabla  v^{\eps},\nabla\sigma_k \nabla\[\sigma_k\cdot \nabla   v^{\eps}\] \right\rangle \mathd s.
		\end{align*}
		This implies that 
		\begin{align*}
			2(J_4+J_5)=&\sum_{k=1}^d\int_0^t\left\langle  \nabla  v^{\eps},(\nabla\sigma_k)^T \nabla\[\sigma_k\cdot \nabla   v^{\eps}\] \right\rangle \mathd s +\sum_{k=1}^d\int_0^t\left\langle \nabla  v^{\eps} , \nabla\sigma_k \nabla\[\sigma_k\cdot \nabla   v^{\eps}\] \right\rangle \mathd s\nonumber
			\\= &\sum_{k=1}^d\int_0^t\left\langle \nabla v^{\eps}, (\nabla\sigma_k)^T(\nabla\sigma_k)^T\nabla v^{\eps} \right\rangle +\|(\nabla\sigma_k)^T\nabla v^{\eps} \|^2_{L^2}  \mathd s\nonumber
			\\&+\sum_{k=1}^d\int_0^t\left\langle ( \sigma_k\cdot\nabla)(\nabla v^{\eps}\otimes(\nabla v^{\eps})^T),(\nabla\sigma_k)^T \right\rangle  \mathd s\nonumber,
		\end{align*}
		where we used the chain rule to get the last equality. 
		
		By making use of $\div \sigma_k=0, k=1,\cdots,d,$ again, we can derive
		\begin{align*}
			2(J_4+J_5)=& \sum_{k=1}^d\int_0^t\left\langle \nabla v^{\eps}, (\nabla\sigma_k)^T(\nabla\sigma_k)^T\nabla v^{\eps} \right\rangle +\|(\nabla\sigma_k)^T\nabla v^{\eps} \|^2_{L^2} \mathd s\nonumber
			\\&-\sum_{k=1}^d\int_0^t\left\langle (\nabla v^{\eps}\otimes(\nabla v^{\eps})^T),( \sigma_k\cdot\nabla)(\nabla\sigma_k)^T \right\rangle  \mathd s\nonumber.
		\end{align*}
	Applying H\"older's inequality, we then obtain the following result 
		\begin{align*}
			2(J_4+J_5)	
			\leq  &C_{\sigma} \int_0^t \| \nabla   v^{\eps}\|_{L^2}^2 \mathd s ,
		\end{align*}
		where  $C_{\sigma}$ is a positive constant depending on $\sigma_k, k=1,\cdots,d.$
		
	Based on the above estimates for $J_2, J_4, J_5,$ by selecting a sufficiently small value for $\eta$ in $J_2$ (We omit the dependence of $C_{\sigma,  v_0}$ on $\eta$ here.), we arrive at the following results.
	\begin{align}\label{ert:5.2-2}
			\|  v^{\eps}_t\|_{\mathbb{H}^1}^2& \leq\|  v^{\eps}_0\|_{\mathbb{H}^1}^2-\int_0^t \|  v^{\eps}\|_{\mathbb{H}^2}^2 \mathd s+ C_{\sigma,  v_0} \int_0^t \|  v^{\eps}\|_{\mathbb{H}^1}^2 \mathd s+M_1(t),
			\\&\leq C_{  v_0,\sigma} -\int_0^t \|  v^{\eps}\|_{\mathbb{H}^2}^2 \mathd s +M_1(t),
	\end{align} 
where we used Lemma \ref{lem:uni01} to get the last inequality.
	
	Define $\tau_R^{  \eps}\assign \inf\{t>0| \sup_{[0,T]}\|  v^{  \eps}_{t}\|_{C^2}\leq R\}\land T.$ Since for fixed $\eps\in(0,1), v^{  \eps}\in C([0,T];C^2)$, we have $\tau_R^{  \eps}\rightarrow T$ almost surely, as $R\rightarrow\infty.$ 
	Taking expectation, we have 
	\begin{align*}
		\mathbb{E}\|  v^{\eps}_{t\land\tau_R^{  \eps}}\|_{\mathbb{H}^1}^2 +\mathbb{E}\int_0^{t\land\tau_R^{  \eps}}\|  v^{\eps}_s\|_{\mathbb{H}^2}^2\mathd s \leq C_{  v_0,\sigma}
	\end{align*}
		Letting $R\rightarrow \infty,$
		the result  then follows.  
	\end{proof}
	
	\begin{lemma}\label{lem:12high}
		For each $p\geq2$, it holds that 
		\begin{align*}
			\sup_{t\in [0,T]}	\mathbb{E}\|  v^{  \eps}_t\|_{\mathbb{H}^1}^{p}+ \mathbb{E}\int_0^T \|  v^{  \eps}_t\|_{\mathbb{H}^1}^{p}
			\|  v^{  \eps}_t\|_{\mathbb{H}^2}^2 \mathd t\leq C_{  v_0,\sigma} ,\quad 	\mathbb{E}\sup_{t\in [0,T]} \|  v^{  \eps}_t\|_{\mathbb{H}^1}^2\leq C_{  v_0,\sigma}.
		\end{align*}
	Here $C_{  v_0,\sigma}$ denotes some positive constant depending on $\|  v_0\|_{{H}^3}$ and $\sigma_k, k=1,\cdots,d.$
	\end{lemma}
	\begin{proof}
		Using It\^o's formula for $\|w^{  \eps}_t\|_{\mathbb{H}^1}^2$ yields that
		\begin{equation*}
			\mathd \|w^{  \eps}_t\|_{\mathbb{H}^1}^{p}=a\|w^{  \eps}_t\|_{\mathbb{H}^1}^{2(\frac{p}{2}-1)}\mathd \|w^{  \eps}_t\|_{\mathbb{H}^1}^2+\frac{p}{4}(\frac{p}{2}-1)\|w^{  \eps}_t\|_{\mathbb{H}^1}^{2(\frac{p}{2}-2)} \[\mathd M_1,\mathd M_1\], \quad \forall p\geq 2.\nonumber
		\end{equation*}
		Notice that $\|w^{  \eps}_t\|_{\mathbb{H}^1}$ is nonnegative, based on  \eqref{ert:5.2-2}, we have 
		\begin{align*}
			\mathd \|w^{  \eps}_t\|_{\mathbb{H}^1}^{p}\leq& \(-\frac{p}{2}\|w^{  \eps}_t\|_{\mathbb{H}^1}^{2(\frac{p}{2}-1)}\|w^{  \eps}_t\|_{\mathbb{H}^2}^2+C_{  v_0,\sigma}\|w^{  \eps}_t\|_{\mathbb{H}^1}^{p}\)\mathd t \\&+a\|w^{  \eps}_t\|_{\mathbb{H}^1}^{2(\frac{p}{2}-1)}\mathd M_1+\frac{p}{4}(\frac{p}{2}-1)\|w^{  \eps}_t\|_{\mathbb{H}^1}^{2(\frac{p}{2}-2)} \[\mathd M_1,\mathd M_1\].
		\end{align*}
	where $C_{  v_0,\sigma}$ is a positive constant depending on $\|  v_0\|_{H^3}$ and $\sigma_k, k=1,\cdots,d,$ and $	\[\mathd M_1,\mathd M_1 \]_{\cdot}$ denotes the quadratic variation process for $M_1.$
	
		Recalling that the continuous local martingale $M_1=-2\sum_{k=1}^d \int_0^t \left\langle \nabla   v^{  \eps}, \nabla \[\sigma_k \cdot \nabla   v^{  \eps}\]\right\rangle dW^{k}_s,$ we notice that 
		\begin{align*}
			\[\mathd M_1,\mathd M_1 \]_{t}=4 \sum_{k=1}^d\left| \left\langle \nabla  v^{  \eps}, \nabla \[\sigma_k \cdot \nabla   v^{  \eps}\]\right\rangle\right|^2  \mathd t.
		\end{align*}

		Furthermore, employing the chain rule leads to
		\begin{align*}
			\left\langle \nabla  v^{  \eps}, \nabla\[\sigma_k \cdot \nabla   v^{  \eps}\]\right\rangle=&\left\langle \nabla v^{  \eps}, \nabla^2   v^{  \eps}\sigma_k\right\rangle+\left\langle \nabla v^{  \eps},( \nabla\sigma_k)^T  \nabla   v^{  \eps}\right\rangle\nonumber
			\\=&-\frac{1}{2}\left\langle  \sigma_k , \nabla\[\tr( \nabla v^{  \eps}\otimes ( \nabla v^{  \eps})^T)\]\right\rangle+\left\langle \nabla v^{  \eps},( \nabla\sigma_k)^T  \nabla   v^{  \eps}\right\rangle,
		\end{align*} 
	 where $\tr A$ denotes the trace of matrix $A.$
		By applying the divergence free properties of $\{\sigma_k,k=1,\cdots, d\}$, we can observe that the first term is $0.$ 
		Consequently, we obtain
		\begin{align*}
			\left\langle \nabla  v^{  \eps}, \nabla\[\sigma_k \cdot \nabla   v^{  \eps}\]\right\rangle&=\left\langle \nabla v^{  \eps},( \nabla\sigma_k)^T  \nabla   v^{  \eps}\right\rangle,
			\\ &\leq \|\nabla \sigma_k\|_{L^{\infty}}\|  v^{  \eps}\|_{\mathbb{H}^1}^2 ,
		\end{align*}
		where we used H\"older's inequality to get the second inequality.
		We then conclude that
		\begin{align}
		\[\mathd M_1,\mathd M_1 \]_{t}\leq C_\sigma \|  v^{  \eps}_t\|_{\mathbb{H}^1}^4\mathd t.\label{estm1}
		\end{align}
		
		Therefore, under the condition   $\sigma_k\in C^{\infty},$ 
		we obtain the following inequality
		\begin{equation}\label{est5.3-1}
			\mathd \|w^{  \eps}_t\|_{\mathbb{H}^1}^{p}\leq \(-\frac{p}{2}\|w^{  \eps}_t\|_{\mathbb{H}^1}^{2(\frac{p}{2}-1)}\|w^{  \eps}_t\|_{\mathbb{H}^2}^2+C_{  v_0,\sigma}\|w^{  \eps}_t\|_{\mathbb{H}^1}^{p}\)\mathd t +\frac{p}{2}\|w^{  \eps}_t\|_{\mathbb{H}^1}^{2(\frac{p}{2}-1)}\mathd M_1.
		\end{equation}
	We now define $\tau_R^{  \eps}\assign \inf\{t>0|  \sup_{[0,T]}\|   v^{  \eps}_{t}\|_{H^1}<R\}\land T.$ Since for fixed $\eps\in(0,1),  v^{  \eps}\in C([0,T];C^1)$, we have $\tau_R^{  \eps}\rightarrow T$ almost surely, as $R\rightarrow\infty.$ 
		 We then obtain the following results by taking  expectation and applying Gronwall's inequality.
		$$\mathbb{E}\|w_{t\land\tau_R^{  \eps}}\|_{\mathbb{H}^1}^{p} \leq C_{  v_0,\sigma},\quad \forall p\geq 2,$$
		and
		\begin{align*}
			\mathbb{E}\int_{0}^{t\land\tau_R^{  \eps}}\frac{p}{2}\|w^{  \eps}_s\|_{\mathbb{H}^1}^{2(\frac{p}{2}-1)}\|w^{  \eps}_s\|_{\mathbb{H}^2}^2ds \leq \|w_0\|_{\mathbb{H}^1}^{p} +C_{  v_0,\sigma}\mathbb{E}\int_{0}^{T}\|w^{  \eps}_s\|_{\mathbb{H}^1}^{p}\mathd s,\quad \forall p\geq 2. 
		\end{align*}
	Let $R\rightarrow \infty,$ we  have
	
	\begin{equation} \label{est:5.3-2}
		\sup_{t\in [0,T]}	\mathbb{E}\|  v^{  \eps}_t\|_{\mathbb{H}^1}^{p}+\frac{p}{2}\mathbb{E}\int_0^T \|  v^{  \eps}_t\|_{\mathbb{H}^1}^{p-2}
		\|  v^{  \eps}_t\|_{\mathbb{H}^2}^2 \mathd t\leq C_{  v_0,\sigma} ,\quad \forall p\geq 2. 
	\end{equation}
		Furthermore, under \eqref{est5.3-1} and \eqref{est:5.3-2}, we choose $p=2$ and have
		\begin{align*}
			\mathbb{E}\sup_{t\in [0,T]}\|w^{  \eps}_t\|_{\mathbb{H}^1}^2\leq C_{  v_0,\sigma}+\mathbb{E}\sup_{t\in [0,T]}M_1.
		\end{align*}
		We apply Burkholder-Davis-Gundy inequality and further find 
		\begin{align*}
			\mathbb{E}\sup_{t\in [0,T]}M_1\leq& C_{  v_0,\sigma}+C \mathbb{E} \(\int_0^T \[\mathd M_1,\mathd M_1\]  \)^{\frac{1}{2}}
			\\\leq& C_{  v_0,\sigma}+C_{\sigma} \mathbb{E} \(\int_0^T\|w_t\|_{\mathbb{H}^1}^{4}\mathd t\)^{\frac{1}{2}}
			\leq C_{  v_0,\sigma}.
		\end{align*}
		In the above inequalities, we used the estimate \eqref{estm1} to derive the first inequality in the last line and the estimate \eqref{est:5.3-2} to derive the second inequality in the last line. The proof is then completed.
	\end{proof}
	
	\begin{lemma}\label{lem:23}
	For fixed $\eps\in(0,1),$ 	it holds $\mathbb{P}$-almost surely that for all $t\in[0,T],$
		\begin{align*}
			\|  v^{  \eps}_t\|_{\mathbb{H}^2}^2+\int_0^t\|  v^{  \eps}_{s}\|_{\mathbb{H}^3}^2\mathd s \leq 	\|  v_0\|_{\mathbb{H}^2}^2+ C_{  v_0,\sigma}\int_0^t \(1+\|  v^{  \eps}_s\|_{\mathbb{H}^2}^2\)\(1+\|  v^{  \eps}_s\|_{\mathbb{H}^1}^2\)\mathd s +M_2(t), 
		\end{align*}
		where \begin{align*}
			M_2(\cdot)=-2\sum_{k=1}^d \int_0^{\cdot} \left\langle \nabla^{2}  v^{  \eps}, \nabla^{2}\[\sigma_k \cdot \nabla   v^{  \eps}\]\right\rangle dW^{k}_s
		\end{align*}  is a continuous local martingale  with respect to the filtration $(\mathcal{F}^{W}_t)_{t\in[0,T]}$ generated by $W^k,k =1,\cdots,d.$ 
      Here $C_{  v_0,\sigma}$ denotes some positive constant depends on $\|  v_0\|_{{H}^3}$ and $\sigma_k, k=1,\cdots,d.$
		In particular,
		\begin{align*}
			\sup_{t\in [0,T]}\mathbb{E}\|  v^{  \eps}_t\|_{\mathbb{H}^2}^2 +\mathbb{E}\int_0^T\|  v^{  \eps}_t\|_{\mathbb{H}^3}^2\mathd t \leq C_{  v_0,\sigma}.
		\end{align*}
	\end{lemma}
	\begin{proof} 
    	As in Lemma \ref{lem:uni12}, we used It\^o's formula and obtian  
		\begin{align*}
			\|\nabla^2  v^{  \eps}_t\|_{L^2}^2	=	 \|\nabla^2  v^{  \eps}_0\|_{L^2}^2+ 2\sum_{l=1}^5J_{l},
		\end{align*}
		where $J_{l}$ are defined by 
		\begin{align*}
			J_1=&\int_0^t\left\langle \nabla^2   v^{  \eps}, \nabla^2\Delta    v^{  \eps}\right\rangle\mathd s = -\int_0^t\| v^{  \eps} \|_{\mathbb{H}^3}^2 \mathd s ,
			\\J_2=&	-\int_0^t\left\langle \nabla^2v^{  \eps}, \nabla^2 [K_{  \eps}*  v^{  \eps}\cdot \nabla   v^{  \eps}] \right\rangle\mathd s ,
			\quad \quad  \quad \quad \quad J_3= -\sum_{k=1}^d \int_0^t \left\langle \nabla^2  v^{  \eps}, \nabla^2\[\sigma_k \cdot \nabla   v^{  \eps}\]\right\rangle dW^{k}_s,
			\\J_4=&\frac{1}{2}\sum_{k=1}^d\int_0^t\left\langle \nabla^2  v^{  \eps}, \nabla^2\(\sigma_k \cdot \nabla \[\sigma_k\cdot \nabla   v^{  \eps}\] \)\right\rangle \mathd s ,
			\quad J_5=\frac{1}{2}\sum_{k=1}^d\int_0^t\left\langle \nabla^2 \[\sigma_k \cdot \nabla   v^{  \eps}\] ,\nabla^2\[\sigma_k \cdot \nabla   v^{  \eps}\]\right\rangle \mathd s .
		\end{align*}
		Here we used integration by parts to get the second term in $J_1.$ 
	For simplicity, we only show the estimate for the interaction part $J_2,$ which is
		the major defference from Lemma \ref{lem:uni12}. Using the chain rule and integration by parts, we have 
		\begin{align*}
			J_2=&	-\int_0^t\left\langle \nabla^2 v^{  \eps}, \nabla^2[K_{  \eps}*  v^{  \eps}\cdot \nabla   v^{  \eps}] \right\rangle\mathd s \\=&\sum_{i=1}^{2}\int_0^t\left\langle \partial_i\Delta  v^{  \eps}, K_{  \eps}*\partial_i  v^{  \eps}\cdot \nabla   v^{  \eps} \right\rangle\mathd s +\int_0^t\left\langle \nabla(\Delta  v^{  \eps}), \nabla^2v^{  \eps}K_{  \eps}*  v^{  \eps} \right\rangle\mathd s .
		\end{align*}

		Recall that $K_{\eps}=\div (V*J_{  \eps}),$ where $V$ is a $2\times2$ matrix and $V_{\alpha\beta}\in L^{\infty}, \forall \alpha,\beta=1,2,$
	we have 
		\begin{align*}
			K_{  \eps}*\partial_i  v^{  \eps}\cdot \nabla   v^{  \eps}=
			\sum_{m,n=1}^2 V^{  \eps}_{m,n}*\partial_{m}\partial_i v^{  \eps} \partial_n  v^{  \eps}. 
		\end{align*}
	We express $J_2$ as follows
			\begin{align*}
			J_2=&\sum_{m,n,i=1}^2\int_0^t\left\langle \partial_i\Delta  v^{\eps }, V^{  \eps}_{m,n}*\partial_{m}\partial_iv^{  \eps} \partial_n  v^{  \eps} \right\rangle\mathd s +\int_0^t\left\langle \nabla(\Delta  v^{  \eps}), \nabla^2v^{  \eps}K_{  \eps}*  v^{  \eps}\right\rangle\mathd s .	
		\end{align*}
			Applying Hölder's inequality and Young's inequality, we obtain
		\begin{align*}
			J_2\leq &\sum_{m,n,i=1}^2\int_0^t\|\partial_i\Delta  v^{  \eps}\|_{L^2}\|V_{m,n}\|_{L^{\infty}}\|\partial_{m}\partial_i  v^{  \eps}\|_{L^1}\|\partial_n  v^{  \eps}\|_{L^2}\mathd s 
			\\&	+\int_0^t\|\nabla(\Delta  v^{  \eps})\|_{L^2}\|K\|_{L^1}\|  v^{  \eps}\|_{L^{\infty}}\|\nabla^2v^{  \eps}\|_{L^2}\mathd s\\ \leq& \eta\int_0^t\|v^{  \eps}\|_{\mathbb{H}^3}^2 \mathd s+ C_{\eta}\int_0^t\|v^{  \eps}\|_{\mathbb{H}^2}^2\(\|v^{  \eps}\|_{\mathbb{H}^1}^2+\|  v^{  \eps}\|_{L^{\infty}}^2\)\mathd s. 
		\end{align*}
		
		The remainder of the argument is analogous to that in Lemma \ref{lem:uni12}. Therefore, the proof will not be reproduced here.
	
	\end{proof}
	
	The next lemma is  in a similar form to Lemma \ref{lem:12high}, but it deals with higher regularities.
	\begin{lemma}\label{lem:23higher}
		It holds that 
		\begin{align*}
			\sup_{t\in [0,T]}	\mathbb{E}\|  v^{  \eps}_t\|_{\mathbb{H}^2}^2\|  v^{  \eps}_t\|_{\mathbb{H}^1}^{2}+\mathbb{E}\int_0^T \|  v^{  \eps}_t\|_{\mathbb{H}^1}^{2}
			\|  v^{  \eps}_t\|_{\mathbb{H}^3}^2 \mathd t+\mathbb{E}\int_0^T \|  v^{  \eps}_t\|_{\mathbb{H}^2}^4\mathd t  \leq  C_{  v_0,\sigma}.
		\end{align*}
	 Here $C_{  v_0,\sigma}$ denotes some positive constant depends on $\|  v_0\|_{{H}^3}$ and $\sigma_k, k=1,\cdots,d.$
	\end{lemma}

			Our last uniform estimates concern on $\mathbb{H}^3$ and $\mathbb{H}^4$-norms. 
			\begin{lemma}\label{lem:34}
			For fixed $\eps\in(0,1),$ 	it holds $\mathbb{P}$-almost surely that for all $t\in[0,T],$
			\begin{align*}
				\|  v^{  \eps}_t\|_{\mathbb{H}^3}^2+\int_0^t\|  v^{  \eps}_{s}\|_{\mathbb{H}^4}^2\mathd s \leq &	\|  v_0\|_{\mathbb{H}^3}^2+ C_{  v_0,\sigma}\int_0^t \(\|  v^{  \eps}_s\|_{\mathbb{H}^1}^2+\|  v^{  \eps}_s\|_{\mathbb{H}^2}^2+\|  v^{  \eps}_s\|_{\mathbb{H}^3}^2\)\mathd s\\&+\int_0^t\(\|  v^{  \eps}_s\|_{\mathbb{H}^2}^4+\|  v^{  \eps}_s\|_{\mathbb{H}^3}^2\|  v^{  \eps}_s\|_{\mathbb{H}^1}^2\)\mathd s +M_3(t),
			\end{align*}
			where \begin{align*}
				M_3(t)=-2\sum_{k=1}^d \int_0^{t} \left\langle \nabla^{3}  v^{  \eps}, \nabla^{3}\[\sigma_k \cdot \nabla   v^{  \eps}\]\right\rangle dW^{k}_s
			\end{align*} is a continuous local martingale  with respect to the filtration $(\mathcal{F}^{W}_t)_{t\in[0,T]}$ generated by $W^k, k =1,\cdots,d.$ 
			Here $C_{  v_0,\sigma}$ denotes some positive constant depending on $\|  v_0\|_{{H}^3}$ and $\sigma_k, k=1,\cdots,d.$
			In particular,
				\begin{align*}
					\sup_{t\in [0,T]}\mathbb{E}\|  v^{  \eps}_t\|_{\mathbb{H}^3}^2 +\mathbb{E}\int_0^T\|  v^{  \eps}_t\|_{\mathbb{H}^4}^2\mathd t \leq C_{  v_0,\sigma}. 
				\end{align*}
			\end{lemma}
			\begin{proof}
			Similar to the previous lemmas, we will evolve $\|\nabla^{3} v^{  \eps}\|_{L^2}^2.$ 
			\begin{align*}
					\|\nabla^{3} v^{  \eps}_t\|_{L^2}^2=\|\nabla^{3} v^{  \eps}_0\|_{L^2}^2+2\sum_{l=1}^5J_{l},
				\end{align*}
			with 
				\begin{align*}
					J_1=&\int_0^t\left\langle \nabla^{3} v^{  \eps}, \nabla^{3} \Delta    v^{  \eps}\right\rangle\mathd s = - \int_0^t\|  v^{  \eps}_{s}\|_{\mathbb{H}^4}^2 \mathd s ,
					\\J_2=&	-\int_0^t\left\langle \nabla^{3}  v^{  \eps}, \nabla^{3}[K_{  \eps}*  v^{  \eps}\cdot \nabla   v^{  \eps}] \right\rangle\mathd s ,
					\quad \quad  \quad \quad \quad J_3= -\sum_{k=1}^d \int_0^t \left\langle \nabla^{3}  v^{  \eps},\nabla^{3}\[\sigma_k \cdot \nabla   v^{  \eps}\]\right\rangle dW^{k}_s,
			\\J_4=&\frac{1}{2}\sum_{k=1}^d\int_0^t\left\langle \nabla^{3}  v^{  \eps}, \nabla^{3}\(\sigma_k \cdot \nabla \[\sigma_k\cdot \nabla   v^{  \eps}\] \)\right\rangle \mathd s, 
				\quad J_5=\frac{1}{2}\sum_{k=1}^d\int_0^t\left\langle \nabla^{3}\[\sigma_k \cdot \nabla   v^{  \eps}\] ,\nabla^{3}\[\sigma_k \cdot \nabla   v^{  \eps}\]\right\rangle \mathd s .
				\end{align*}
			Here we used integration by parts to get the second term in $J_1.$
			
			For simplicity, we only present the estimate for the interaction part $J_2,$ which is the main difference compared to Lemma \ref{lem:uni12}. We have
				\begin{align*}
				J_2=&-\int_{0}^{t}\left\langle \nabla^{3}   v^{  \eps},\nabla^{3} [K_{  \eps}*  v^{  \eps}\cdot \nabla   v^{  \eps}] \right\rangle\mathd s\\=& \int_0^t\left\langle \nabla^{2}  ( \Delta v^{  \eps}), (\nabla^3   v^{  \eps})^TK_{  \eps}*  v^{  \eps}\right\rangle\mathd s +\sum_{i,j=1}^{2}\int_0^t\left\langle  \partial_i\partial_j\Delta  v^{  \eps}, K_{  \eps}* \partial_i v^{  \eps}\cdot  \partial_j\nabla   v^{  \eps} \right\rangle\mathd s 
				\\&+\sum_{i,j=1}^{2}\int_0^t\left\langle   \partial_i\partial_j\Delta  v^{  \eps}, K_{  \eps}*\partial_jv^{  \eps} \cdot  \partial_i\nabla   v^{  \eps} \right\rangle\mathd s +\sum_{i,j=1}^{2}\int_0^t\left\langle   \partial_i\partial_j\Delta  v^{  \eps}, K_{  \eps}*\partial_i\partial_j v^{  \eps}\cdot \nabla   v^{\eps} \right\rangle\mathd s 
					\\=:&J_{2,1}+\sum_{i,j=1}^{2}J_{2,2}^{i,j}+\sum_{i,j=1}^{2}J_{2,3}^{i,j}+\sum_{i,j=1}^{2}J_{2,4}^{i,j}.
				\end{align*}
			where we used integration by parts and chain rule to get the second equality.
				The first term $J_{2,1}$ can be controlled by H\"older's inequality easily, 
				\begin{align*}
					J_{2,1}\leq &\int_0^t\| \nabla^{2}  ( \Delta v^{  \eps})\|_{L^2}\|K\|_{L^1}\|  v^{  \eps}\|_{L^{\infty}}\| \nabla^{3}   v^{  \eps}\|_{L^2}\mathd s \nonumber
				\\\leq &\eta \int_0^t\|v^{  \eps}\|_{\mathbb{H}^4}^2\mathd s+ C_{\eta,  v_0}\int_0^t\| v^{  \eps}\|_{\mathbb{H}^3}^2\mathd s,
				\end{align*}
			where we use Young's inequality and \eqref{bound} to get the second inequality.
				For fixed $i,j\in\{1,2\},$ the terms $J_{2,2}^{i,j}, J_{2,3}^{i,j}, J_{2,4}^{i,j}$ can be controlled in a similar manner. We provide a detailed explanation for the case of $J_{2,2}^{i,j}$ here.
				
				Recall that $K^{  \eps}_{j'}= \sum_{i'}\partial_{i'} V^{  \eps}_{i'j'}$ and 
				\begin{equation}\label{est:5.6-1}
					\|V^{  \eps}_{i'j'}*f\|_{L^{\infty}}\leq \|V^{  \eps}_{i'j'}\|_{L^{\infty}}\|f\|_{L^1}\lesssim \|V_{i'j'}\|_{L^{\infty}} \|f\|_{L^2},
				\end{equation}
			where we used H\"older's inequality  and the compactness of torus $\mathbb{T}^2$ to get the second inequality,
				we then find 
				\begin{align*}
					J_{2,2}^{i,j}=&\sum_{i',j'=1}^2\int_0^t\left\langle \partial_i\partial_j\Delta  v^{  \eps}, V^{  \eps}_{i'j'}*\partial_{i'}\partial_i v^{  \eps} \partial_{j'}\partial_j  v^{  \eps} \right\rangle\mathd s\nonumber
					\\\leq & \sum_{i',j'=1}^2\int_0^t\|\partial_i\partial_j\Delta  v^{  \eps}\|_{L^2}\|V_{i'j'}\|_{L^{\infty}}\|\partial_{i'}\partial_i  v^{  \eps}\|_{L^2}\|\partial_{j'}\partial_j   v^{  \eps}\|_{L^2}\mathd s \nonumber
					\\\leq &\eta \int_0^t\|  v^{  \eps}\|_{\mathbb{H}^4}^2\mathd s+ C_{\eta}\int_0^t\|v^{  \eps}\|_{\mathbb{H}^2}^4\mathd s .
				\end{align*}
			where we used H\"older's inequality and estimate \eqref{est:5.6-1} to get the second inequality and Young's inequality to get the last inequality.
				For  $J_{2,3}^{i,j}$ and $J_{2,4}^{i,j}$, we have 
				\begin{align*}
				J_{2,3}^{i,j}\leq &\eta \int_0^t\|  v^{  \eps}\|_{\mathbb{H}^4}^2\mathd s+ C_{\eta}\int_0^t\|  v^{  \eps}\|_{\mathbb{H}^2}^4\mathd s ,
				\end{align*}
				\begin{align*}
					J_{2,4}^{i,j}\leq \eta \int_0^t\| v^{  \eps}\|_{\mathbb{H}^4}^2\mathd s+ C_{\eta}\int_0^t\| v^{  \eps}\|_{\mathbb{H}^3}^2\|v^{  \eps}\|_{\mathbb{H}^1}^2\mathd s .
				\end{align*}
				The remainder of the argument is analogous to that in Lemma \ref{lem:uni12}. Therefore, the proof will not be reproduced here.
			\end{proof}

	Similar to	Lemma \ref{lem:12high}, we can also deduce the following result.
			\begin{lemma}\label{suph2}
			It holds that 
			\begin{align*}
			&\mathbb{E}[\sup_{t\in [0,T ]}\|  v^{  \eps}_t\|_{\mathbb{H}^2}^2]\leq C_{  v_0,\sigma}.
			\end{align*}
		Here $C_{  v_0,\sigma}$ denotes some positive constant depending on $\|  v_0\|_{{H}^3}$ and $\sigma_k, k=1,\cdots,d.$
		
		\end{lemma}

We now proceed with the proof of Lemma \ref{lem:reg}. In this proof, we leverage the representation of stochastic 2-dimensional Navier-Stokes equation (\ref{eqt:ns}) and exploit the properties of (stochastic) integrals to establish the regularity property outlined in Lemma \ref{lem:reg}.
\begin{proof}[Proof of Lemma \ref{lem:reg}]
The following proof comprises two steps.

Step 1:  In this step, we first establish that 
\begin{align*}
	(  v_t)_{t\in[0,T]}\in C([0,T],H^{2}(\mathbb{T}^2))\quad \mathbb{P}-a.s.,
\end{align*}by utilizing the representation of stochastic 2-dimensional Navier-Stokes equation \eqref{eqt:ns} and the properties of (stochastic) integrals. 

Recall that $\sigma_k, k=1,\cdots,d,$ is smooth given in Assumption \ref{assumption2} and the regularity of $(  v_t)_{t\in[0,T]}$ given in Definition \ref{defnsw}, i.e. $(  v_t)_{t\in[0,T]}$ is a a continuous $L^2(\mathbb{T}^2)$ valued $(\mathcal{G}_t)_{t\in [0,T]}$-adapted stochastic process $(  v_t)_{t\in [0,T]}$  satisfying $\mathbb{E} \int_0^T \|  v_t\|_{H^4}^2\mathd t <\infty,$ by Lemma \ref{lemma triebel}, we have

	\begin{align*}
	&\mathbb{E}	\int_0^T  \|\Delta  v_s\|_{H^2}^2\mathd s <\infty,
	\\&\mathbb{E}	\int_0^T  \|K*  v_s\cdot \nabla   v_s\|_{H^2}^2\mathd s <\infty,
	\\&\mathbb{E}	\int_0^T  \|\sigma_k\cdot\nabla\(\sigma_k\cdot\nabla  v_s\)\|_{H^2}^2\mathd s <\infty,
	\\&\mathbb{E}	\int_0^T  \| \sigma_k\cdot\nabla   v_s\|_{H^2}^2\mathd s <\infty,
	\end{align*}
	We then introduce a continuous $H^{2}$ valued  process $(V_1(t))_{t\in[0,T]}$ defined as follows:
\begin{align*}
	V_1(t)\assign\int_0^t  \Delta  v_s\mathd s -\int_0^t  K*  v_s\cdot \nabla   v_s\mathd s+\frac{1}{2} \sum_{k=1}^{d}\sigma_k\cdot\nabla\(\sigma_k\cdot\nabla  v_s\)\mathd s- \sum_{k=1}^{d}\int_0^t  \sigma_k\cdot\nabla   v_s \mathd W_s^{k} ,
\end{align*}
where the integral with respect to $\mathd s$ is in the sense of Bochner integral and stochastic integral is in the sense of  Hilbert space valued ($H^{2}(\mathbb{T}^2)$) stochastic integral defined in  \cite[Section 2]{liu2015stochastic}.

We aim to show that it holds $\mathbb{P}$-almost surely $V_1=  v$ in the sense of  $C([0,T];L^2(\mathbb{T}^2)).$  Consequently, it holds $\mathbb{P}$-almost surely that $$  v\in C([0,T];H^{2}(\mathbb{T}^2)).$$  

By the properties of Bochner integral and stochastic integral as outlined in \cite[Lemma A.2.2 and Lemma 2.4.1]{liu2015stochastic}, (We regard $\left\langle e_m, \cdot\right\rangle$ as a bounded linear functional on $L^2(\mathbb{T}^2)$ here.) it holds $\mathbb{P}$-almost surely that for each $e_m$ of  the Fourier basis $\{e_m\assign e^{imx}, m\in \mathbb{Z}^2\}$ and $ t\in[0,T],$ 
\begin{align}
	\left\langle V_1(t),e_m\right\rangle&=\left\langle	\int_0^t  \Delta  v_s\mathd s -\int_0^t  K*  v_s \cdot\nabla   v_s\mathd s+\frac{1}{2} \sum_{k=1}^{d}\sigma_k\cdot\nabla\(\sigma_k\cdot\nabla  v_s\)\mathd s- \sum_{k=1}^{d}\int_0^t  \sigma_k\cdot\nabla   v_s \mathd W_s^{k} ,e_m\right\rangle\nonumber\\&=\int_0^t \left\langle   \Delta  v_s,e_m\right\rangle \mathd s -\int_0^t\left\langle    K*  v_s \cdot\nabla  v_s,e_m\right\rangle \mathd s\nonumber
	\\&+\frac{1}{2}\sum_{k=1}^{d}\int_0^t\left\langle \sigma_k\cdot\nabla(\sigma_k\cdot\nabla  v_s),e_m\right\rangle \mathd s- \sum_{k=1}^{d}\int_0^t  \left\langle    \sigma_k\cdot\nabla  v_s,e_m\right\rangle \mathd W_s^{k}.\nonumber
\end{align}
In addition, from Definition \ref{defnsw}, we deduce that it holds $\mathbb{P}$-almost surely that for each $e_m$ of  the Fourier basis $\{e_m\assign e^{imx}, m\in \mathbb{Z}^2\}$ and $ t\in[0,T],$ 
\begin{align}\label{representns}
	\left\langle    v_t-  v_0,e_m\right\rangle&=\int_0^t \left\langle   \Delta  v_s,e_m\right\rangle \mathd s-\int_0^t\left\langle    K*  v_s \cdot\nabla  v_s,e_m\right\rangle \mathd s\nonumber
	\\&+\frac{1}{2}\sum_{k=1}^{d}\int_0^t\left\langle \sigma_k\cdot\nabla(\sigma_k\cdot\nabla  v_s),e_m\right\rangle \mathd s- \sum_{k=1}^{d}\int_0^t  \left\langle   \sigma_k\cdot\nabla  v_s,e_m\right\rangle \mathd W_s^{k}.
\end{align}
Here we expressed the Stratonovich integral into It\^o's form and $\frac{1}{2}\sum_{m\in \mathbb{Z}^2}\int_0^t\left\langle \sigma_k\cdot\nabla(\sigma_k\cdot\nabla  v_s),e_m\right\rangle \mathd s$ is the term tranformed from the Stratonovich integral.
Consequently,
it holds $\mathbb{P}$-almost surely that 	$\forall t\in[0,T],$
$	  v_t -   v_0$
is equivalent with $V_1(t)$
in the $L^2(\mathbb{T}^2)$ sense. We then have $$(  v_t)_{t\in[0,T]}\in C([0,T],H^{2}(\mathbb{T}^2))\quad \mathbb{P}-a.s..$$
By Sobolev embedding theorem, we then deduce that 
\begin{align}\label{v-continoustx}
	(  v_t)_{t\in[0,T]}\in C([0,T],C(\mathbb{T}^2))\quad \mathbb{P}-a.s..
\end{align}
Consequently, by  the regularity of $(  v_t)_{t\in[0,T]}$ given in Definition \ref{defnsw}, i.e.  for all $t\in [0,T]$,  $$\|  v_t\|_{L^{\infty}(\mathbb{T}^2)}\leq \|  v_0\|_{L^{\infty}(\mathbb{T}^2)}\quad\mathbb{P}-a.s.,$$   $$ \underset{x\in \mathbb{T}^2 }{ess\inf}   v_t\geq  \underset{x\in \mathbb{T}^2 }{ess\inf}   v_0\quad\mathbb{P}-a.s.,$$ we have that it holds $\mathbb{P}$-almost surely that 
\begin{align*}
	\underset{x\in \mathbb{T}^2 }{\sup}  v_t\leq \underset{x\in \mathbb{T}^2 }{\sup}  v_0,\quad    \underset{x\in \mathbb{T}^2 }{\inf}   v_t\geq  \underset{x\in \mathbb{T}^2 }{\inf}   v_0,\quad \forall t\in[0,T].
\end{align*}

Step2: Subsequently, based on \eqref{v-continoustx}, we then deduce \eqref{eq:re2} in Lemma \ref{lem:reg} by utilizing the representation of stochastic 2-dimensional Navier-Stokes equation \eqref{eqt:ns} again in the second step.

By Sobolev embedding theorem, we conclude that 
\begin{align*}
 \mathbb{E} \int_0^T \|  v_t\|_{C^2}^2\mathd t\lesssim\mathbb{E} \int_0^T \|  v_t\|_{H^4}^2\mathd t<\infty.
\end{align*}
We then  have
\begin{align*}
	\Delta  v +K*  v\cdot \nabla   v++\frac{1}{2} \sum_{k=1}^{d}\sigma_k\cdot\nabla\(\sigma_k\cdot\nabla  v\)&\in L^2([0,T]\times\Omega,\mathd t\otimes\mathbb{P},C(\mathbb{T}^2)),
\end{align*}
and
\begin{align*}
\sigma_k\cdot\nabla   v &\in L^2([0,T]\times\Omega,\mathd t\otimes\mathbb{P},C^1(\mathbb{T}^2)),\quad \forall k=1,\cdots,d.
\end{align*}
We then define $V_2(t,x)$ as follows
\begin{align*}
	V_2(t,x)\assign&\int_0^t  \Delta  v_s(x)\mathd s -\int_0^t  K*  v_s(x)\cdot \nabla   v_s(x)\mathd s
	\\&+\frac{1}{2} \sum_{k=1}^{d}\sigma_k(x)\cdot\nabla\(\sigma_k(x)\cdot\nabla  v_s(x)\)\mathd s- \sum_{k=1}^{d}\int_0^t  \sigma_k(x)\cdot\nabla   v_s(x) \mathd W_s^{k} .
\end{align*}
Through Theorem 3.1.1 in \cite{kunita1997stochastic}, we find that $V_2(t,x)$ is a continuous $C(\mathbb{T}^2)$ valued process. We further deduce that $V_2\in C([0,T];L^2(\mathbb{T}^2))\quad \mathbb{P}-a.s..$ 

 We now show that show that $V_2(\cdot)=  v(\cdot)-  v_0, \mathbb{P}-a.s.$ in the sense of $C([0,T],L^2(\mathbb{T}^2)).$ Since $(  v_t)_{t\in[0,T]}\in C([0,T],L^2)$, by \eqref{representns}, we can express $  v_t$ as follows.
 \begin{align*}
 	v_t-  v_0&=\sum_{m\in \mathbb{Z}^2}	\left\langle    v_t-  v_0,e_m\right\rangle e_m
 	\\&=\sum_{m\in \mathbb{Z}^2}\int_0^t \left\langle   \Delta  v_s,e_m\right\rangle \mathd se_m -\sum_{m\in \mathbb{Z}^2}\int_0^t\left\langle    K*  v_s \cdot\nabla  v_s,e_m\right\rangle \mathd se_m
 	\\&+\frac{1}{2}\sum_{k=1}^{d}\sum_{m\in \mathbb{Z}^2}\int_0^t\left\langle \sigma_k\cdot\nabla(\sigma_k\cdot\nabla  v_s),e_m\right\rangle \mathd se_m- \sum_{k=1}^{d}\sum_{m\in \mathbb{Z}^2}\int_0^t  \left\langle   \sigma_k\cdot\nabla  v_s,e_m\right\rangle \mathd W_s^{k}e_m.
 \end{align*}
By (stochastic) Fubini theorem, we have 
\begin{align*}
	v_t-  v_0&=\sum_{m\in \mathbb{Z}^2}	\left\langle    v_t-  v_0,e_m\right\rangle e_m
	\\&=\sum_{m\in \mathbb{Z}^2}\left\langle  V_2(t),e_m\right\rangle e_m.
\end{align*}
Then, $V_2(\cdot)=  v(\cdot)-  v_0,\mathbb{P}-a.s.$ in the sense of $C([0,T],L^2(\mathbb{T}^2)).$
Recall that $$(v_t)_{t\in[0,T]}\in C([0,T],C(\mathbb{T}^2))\quad  \mathbb{P}-a.s.,$$ and $$(V_2(t))_{t\in[0,T]}\in C([0,T],C(\mathbb{T}^2))\quad \mathbb{P}-a.s.,$$  Consequently, it holds  $\mathbb{P}$-almost surely that $\forall t\in[0,T], x\in\mathbb{T}^2,$
\begin{align*}
	v_t(x)-  v_0(x)=V_2(t,x).
\end{align*}
\end{proof}
\newpage
\appendix\section{Well-posedness of particle system}\label{appendix:vortex} 

In this section, we aim to prove well-posedness of \eqref{eqt:vortex} by similar methods  in \cite{fontbona2007paths} and \cite{flandoli2011full}. The basic idea is  to show that  the singularity of the kernel will never  be visited almost surely.  Specifically, we first consider the regularized vortex systems and show that   $X_i(t)\neq X_j(t)$ for all $t\in [0,T]$ and $i\neq j$ almost surely by a uniform estimate in regularizing parameter.
\setcounter{equation}{0}
\renewcommand{\theequation}{A.\arabic{equation}}

\begin{proof}[Proof of Proposition \ref{thm:vortex}]
	The proof consists of three steps.
	
	Step 1: Our goal  is to show well-posedness of \eqref{eqt:vortex} with random initial data.  To complete this, we first construct the solution to the vortex systems with regularized interaction kernels and the initial data satisfying a strictly  positive minimal mutual distance $\eta>0,$ i.e. $$\underset{i,j=1,\cdots,N}{\min}|X_i(0)-X_j(0)|>\eta,$$
	in this step.

	Recall  that  $G(x)$ is the Green function induced by Laplace operator  on $\mathbb{T}^{2}$.  Let $G_{\eps}$  be some symmetric smooth periodic funtion such that 
	$G_{\eps}(x)=G(x)$ for $|x|>\eps$ and 
	\begin{equation*}
		C_1\log(|x|\vee \eps)-C_3\leq G_{\eps} \leq  C_2\log(|x|\vee \eps)+C_3,\quad 
		|\nabla G_{\eps}|\leq C_4\frac{1}{(|x|\vee \eps)},\quad|\nabla^{ 2}G_{\eps}|\leq C_4 \frac{1}{(|x|\vee \eps)^2},
	\end{equation*}
	where $C_i$, $i=1,2,3,4,$ are   positive universal  constants and $  \eps\in(0,1).$  We now consider the regularized systems  \eqref {eqt:apppro} with deterministic initial data $x^N=(x_1,\cdots,x_N)\in \mathbb{T}^{2N}$:
	\begin{equation}\label{eqt:apppro}
		X_{i,\eps}(t)=x^N_i+ \frac{1}{N}\sum_{j\neq i}\int_0^t \xi_j\nabla^{\perp}G_{\eps}(X_{i,\eps}(s)-X_{j,\eps}(s))\mathd s+\sqrt{2}B^i_t+\sum_{k=1}^d\int_0^t\sigma_k(X_{i,\eps}(s))\circ\mathd W^k_s.
	\end{equation} 
	We define $(\varphi^{\eps}_t(x^N))_{t\in[0,T]}$ as the unique probabilistically strong solution to the regularized vortex system \eqref{eqt:apppro} with deterministic initial data $x^N\in \mathbb{T}^{2N}.$ The regularity properties of $\nabla^{\perp}G_{\eps}$ and standard results on stochastic flow \cite[Theorem 4.6.5]{kunita1997stochastic} imply existence of a continuous version of the two parameter processes $(t,x^N)\rightarrow \varphi^{\eps}_t(x^N):[0,T]\times\mathbb{T}^{2N}\rightarrow\mathbb{T}^{2N},$ which is moreover continuously  differentiable in $x^N.$ 
	
	 We now enlarge the probability space. 
	For fixed $\eta>0$, we define the following subset of torus $\mathbb{T}^{2N}$ \begin{align*}
		\mathbb{T}^{2N}_{\eta}\assign\{x^N=(x_1,\cdots,x_N)\in\mathbb{T}^{2N}:\underset{i,j=1,\cdots,N}{\min}|x_i-x_j|>\eta\}
	\end{align*}  and let $l^{\eta}(\mathd x^N)$ be the normalized Lebesgue measure on $\mathbb{T}^{2N}_{\eta},$ i.e.\begin{align}\label{eqt:leta}
		l^{\eta}(d x^N)\assign\frac{1}{\int_{T_{\eta}^{2\mathbb{N}}}1\mathd x^N}I_{\mathbb{T}^{2N}_{\eta}}(x^N)\mathd x^N.
	\end{align}
	Let the space $\Omega_{1}\assign\mathbb{T}^{2N}\times\Omega$ be endowed with the natural complete product $\sigma$-algebra  with respect to the product probability measure on $\Omega_{1},$
	\begin{align*}
		\mathbb{P}^{\eta} (\mathd x^N,\mathd\omega)\assign	l^{\eta}(\mathd x^N)\times\mathbb{P}(\mathd \omega).
	\end{align*} We write $\mathbb{E}^{\eta}$ for the associated expectation and $(\bar{\mathcal{F}}_t)_{t\in[0,T]}$  stands for the normal filtration generated by $(\mathcal{F}_t\otimes \mathbb{B}(\mathbb{T}^{2N}))_{t\in[0,T]}.$ 
	For each $\eps>0,$ consider now on  $\Omega_{1}$ the process $\{  \bar{X}^{\eps,N}_t,t\in[0,T]\}$ defined as 
	\begin{align}\label{eqt:barx}
		\bar{X}^{\eps,N}_t(x^N,\omega)&\assign \varphi^{\eps}_t(x^N)(\omega),\quad \forall t\in[0,T].
	\end{align}
	and the $(\bar{\mathcal{F}}_t)_{t\in[0,T]}$ stopping time
	\begin{equation}\label{eqt:bartau}
		\bar{\tau}_{\eps}=\inf\{t\in [0,T], \exists i\neq j, |  \bar{X}_{\eps,i}(t)-  \bar{X}_{\eps,j}(t)|\leq \eps\}.
	\end{equation}
	Notice that the random variable $  \bar{X}^{\eps,N}_0(x^N,\omega)=x^N$ has the law $l^{\eta}$ defined in \eqref{eqt:leta}.
	
	Step 2:  We provide uniform estimates in Lemma \ref{lem:vortexito} and Lemma \ref{lem:zxcvbnn} which are uniform in the regularizing parameter in this step.
	
	More precisely, as in \cite{flandoli2011full}, we first show a uniform estimate on the singular potential $	\left| 	\mathbb{E}^{\eta}	\Phi_{\eps}(    \bar{X}^{\eps,N}_{t\wedge \bar{\tau}_{\eps}})\right| $  with respect to $\eps,t,$ i.e. \eqref{eqt:uniforpoten}.
	Due to the introduction of $\{B_i,i=1,\cdots,N\},$ we have to deal with the term evolving $\Delta G_{  \eps}$ (see $K_1^{\eps}(t)$ below), which does not appear in \cite{flandoli2011full}. As a result, we can only obtain a uniform estimate for $	\left| 	\mathbb{E}^{\eta}	\Phi_{\eps}(    \bar{X}^{\eps,N}_{t\wedge \bar{\tau}_{\eps}})\right| ,$ rather than $\mathbb{E}^{\eta}	\sup_{[0, T]}\Phi_{\eps}(    \bar{X}^{\eps,N}_{t})$ in \cite{flandoli2011full}. However, it is sufficient for us to arrive at the final result.

	\begin{lemma}\label{lem:vortexito}For each $t\in [0,T]$ and $\eps \in(0,\eta),$ we have 
		\begin{equation}
			\left| 	\mathbb{E}^{\eta}	\Phi_{\eps}(    \bar{X}^{\eps,N}_{t\wedge 	\bar{\tau}_{\eps}})\right| \leq C_{N,T,\sigma,\eta} +C_N\mathbb{E}^{\eta}\int_0^{t\wedge \tau_{\eps}}\phi(    \bar{X}^{\eps,N}_{s})\mathd s,\nonumber
		\end{equation}
		where $\Phi_{\eps}(x^N)=\sum_{i\neq j,i,j=1,\cdots,N}G_{\eps}(x_{i}-x_{j}),\quad x^N\in \mathbb{T}^{2N},$ and $\phi:\mathbb{T}^{2N}\rightarrow \mathbb{R}^{+}$  is defined by 
		\begin{equation}\label{parti-phi}
			\phi(x^N)=\sum_{i\neq j\neq l,i,j,l=1,\cdots,N}\frac{1}{|x_i-x_l||x_i-x_j|},\quad x^N\in \mathbb{T}^{2N}.
		\end{equation}  
		Here $C_{N}$  denotes  a positive constant depending on $N,$ $C_{N,T,\sigma,\eta}$ denotes a positive constant depending on $N,T,\eta,$ $\sigma_k, k=1,\cdots,d$
	\end{lemma}
	\begin{proof}

			By Ito's formula, we have
			\begin{equation}
				\Phi_{\eps}(    \bar{X}^{\eps,N}_{t}) =  \Phi_{\eps}(    \bar{X}^{\eps,N}_0)+K_1^{\eps}(t)+K_2^{\eps}(t)+K_3^{\eps}(t)+K_4^{\eps}(t)+M_t^{\eps},\nonumber
			\end{equation}
			where  $M_t^{\eps}$ is a continuous martingale with zero expectation and $K_i$, $i=1,2,3,4$, denote 
			\begin{align*}
				K_1^{\eps}(t)=&\sum_{i\neq j,i,j=1,\cdots,N }\int_{0}^{t}\Delta G_{\eps}\(    \bar{X}_{i,\eps}(s)-    \bar{X}_{j,\eps}(s)\)\mathd s,\nonumber
				\\ K_2^{\eps}(t)=&\frac{1}{N} \int_{0}^{t}\sum_{i \neq j,i,j=1,\cdots,N}\nabla G_{\eps}\(    \bar{X}_{i,\eps}(s)-    \bar{X}_{j,\eps}(s)\)
				\\&\cdot \left[\sum_{i'\neq i}\xi_{i'}\nabla^{\perp}G_{\eps}\(    \bar{X}_{i,\eps}(s)-    \bar{X}_{i',\eps}(s)\)- \sum_{j'\neq j}\xi_{j'}\nabla^{\perp}G_{\eps}\(    \bar{X}_{j,\eps}(s)-    \bar{X}_{j',\eps}(s)\)\right]\mathd s ,\nonumber
				\\K_3^{\eps}(t)=&\frac{1}{2} \sum_{i\neq j,i,j=1,\cdots,N} \sum_{k=1}^d \int_0^t \(\sigma_k(    \bar{X}_{i,\eps}(s))\cdot\nabla \sigma_k(    \bar{X}_{i,\eps}(s))-\sigma_k(    \bar{X}_{j,\eps}(s))\cdot\nabla \sigma_k(    \bar{X}_{j,\eps}(s))\)
				\\&\cdot \nabla G_{\eps}\(    \bar{X}_{i,\eps}(s)-    \bar{X}_{j,\eps}(s)\)\mathd s,\nonumber
				\\K_4^{\eps}(t)=&\frac{1}{2} \sum_{i\neq j,i,j=1,\cdots,N} \sum_{k=1}^d \int_{0}^{t}\nabla^{2 }G_{\eps}\(    \bar{X}{i,\eps}(s)-    \bar{X}_{j,\eps}(s)\)
				\\&\cdot 
				\left[ \( \sigma_k(    \bar{X}_{i,\eps}(s))-\sigma_k(    \bar{X}_{j,\eps}(s))\)\otimes \(\sigma_k(    \bar{X}_{i,\eps}(s))-\sigma_k(    \bar{X}_{j,\eps}(s)) \)\right] \mathd s.\nonumber\end{align*}
			
			Recall $  \bar{X}^{\eps,N}_0$ has the law $l^{\eta}$ defined in \eqref{eqt:leta}, based on the definition of $l^{\eta},$	we first find that there exists a positive constant $C_{\eta}$ depends on $\eta$ such that
			\begin{align*}
				\bigg|\mathbb{E}^{\eta}\Phi_{\eps}(    \bar{X}^{\eps,N}_0)\bigg|\leq C_{\eta},\quad \forall \eps \in(0,\eta).
			\end{align*}
			Since  $\Delta G(x)=0$ for any $x\neq 0$,  we have $K_1^{\eps}(t\wedge\bar{\tau}_{\eps})=0.$  As to  $K_2^{\eps}(t)$, we find 
			
			\begin{align*}
				|K_2^{\eps}(t\wedge \bar{\tau}_{\eps})|\leq &C_N\int_{0}^{t\wedge \bar{\tau}_{\eps}}\sum_{i\neq j\neq  i',i,j,i'=1,\cdots,N}|\xi_{i'}|\(|    \bar{X}_{i,\eps}(s)-    \bar{X}_{j,\eps}(s)|\vee \eps\)^{-1}\(|    \bar{X}_{i,\eps}(s)-    \bar{X}_{i',\eps}(s)|\vee \eps\)^{-1}\mathd s.
			\end{align*}
			Here we use the fact that $\nabla^{\perp}G_{\eps}\cdot \nabla G_{\eps}=0$ and  $|\nabla G_{\eps}|\lesssim\frac{1}{(|x|\vee \eps)}.$
			Since $\{\xi_i\}$ are uniformly bounded, (for simplicity, we  omit the dependence of constants on $\mathcal{L}(\xi),$) we have 
			\begin{align*}
				\mathbb{E}^{\eta} |	K_2^{\eps}(t\wedge 	\bar{\tau}_{\eps})|\leq C_N  \mathbb{E}^{\eta}\int_{0}^{t\wedge \bar{\tau}_{\eps}} \phi(    \bar{X}^{\eps,N}(s))\mathd s,
			\end{align*}
			where $\phi$ is defined by  $$\phi:     x^N\in \mathbb{T}^{2N}\rightarrow\phi(    x^N)=\sum_{i\neq j\neq l}\frac{1}{|x_i-x_l||x_i-x_j|}.$$ 
			
			Furthermore, since the vector fields $\{\sigma_k,k\in 1, \cdots, d\}$ are smooth on torus $\mathbb{T}^2,$ we obtain 
			\begin{align*}
				\sup_{\eps>0}	|K_3^{\eps}(t)|\leq&C_{\sigma}\sup_{\eps>0} \sum_{i\neq j,i,j=1,\cdots,N}\int_0^T \(|    \bar{X}_{i,\eps}(s)-    \bar{X}_{j,\eps}(s)|\vee \eps\)^{-1}\left|     \bar{X}_{i,\eps}(s)-    \bar{X}_{j,\eps}(s) \right| \mathd s\leq C_{N,T,\sigma},
				\\\sup_{\eps>0}	|K_4^{\eps}(t)|\leq&C_{\sigma}\sup_{\eps>0} \sum_{i\neq j,i,j=1,\cdots,N}\int_0^T \(|    \bar{X}_{i,\eps}(s)-    \bar{X}_{j,\eps}(s)|\vee \eps\)^{-2}\left|     \bar{X}_{i,\eps}(s)-    \bar{X}_{j,\eps}(s) \right|^2 \mathd s\leq C_{N,T,\sigma}.
			\end{align*}
	
		Therefore, we conclude that 
		\begin{equation*}
			\left| 	\mathbb{E}^{\eta}	\Phi_{\eps}(    \bar{X}^{\eps,N}_{t\wedge 	\bar{\tau}_{\eps}})\right| \leq  C_{N,T,\sigma,\eta}+C_N\mathbb{E}^{\eta}\int_0^{t\wedge	\bar{\tau}_{\eps}}\phi(    \bar{X}^{\eps,N}(s))\mathd s,
		\end{equation*}
		where $C_N$ is a universal positive  constant depending on $N$ and $C_{N,T,\sigma,\eta}$ is a universal positive constant depending on $T,N,\eta, \{\sigma_k,k =1, \cdots, d\}.$
	\end{proof}
	
	\begin{lemma}\label{lem:zxcvbnn}For $\phi$ defined in \eqref{parti-phi}, each $t\in [0,T]$ and $\eps\in(0,\eta),$ it holds that 
		\begin{equation}
			\mathbb{E}^{\eta}\int_0^{t\wedge 	\bar{\tau}_{\eps}}\phi(    \bar{X}^{\eps,N}(s))\mathd s\leq C_{N,T}.\nonumber
		\end{equation}
		Here $C_{N,T}$ is a positive constant depending on $N,T.$ 
		
		In particular, \begin{align}\label{eqt:uniforpoten}
			\left| 	\mathbb{E}^{\eta}	\Phi_{\eps}(    \bar{X}^{\eps,N}_{t\wedge 	\bar{\tau}_{\eps}})\right| \leq C_{N,T,\sigma,\eta},\quad \forall t\in[0,T],\forall   \eps\in(0,\eta).
		\end{align}
		Here $C_{N,T,\sigma}$ is a positive constant depending on $N,T,\eta$  and $\{\sigma_k,k=1,\cdots,d\}.$ 
	\end{lemma}
	\begin{proof}
		
		\begin{align*}
			\mathbb{E}^{\eta}\int_0^{t\wedge 	\bar{\tau}_{\eps}}\phi(    \bar{X}^{\eps,N}_{s})\mathd s=&\mathbb{E}\int_0^{t\wedge 	\bar{\tau}_{\eps}}\int_{T^{2\mathbb{N}}}\phi(    \varphi^{\eps}_s(x^N)) l^{\eta}(\mathd x^N)\mathd s
			\\\lesssim&\mathbb{E}\int_0^t\int_{\mathbb{T}^{2N}}\phi(\varphi^{\eps}_s(x^N))\mathd x^N\mathd s.\nonumber
		\end{align*}
		where we used \eqref{eqt:leta} to get the second inequality.
		Since the coefficients of \eqref{eqt:approvortex}  are divergence free,   $\varphi^{\eps}$ is   Lebesgue measure preserving, see \cite[Theorem 4.3.2]{kunita1997stochastic}. We then arrive at 
		\begin{align*}
			\mathbb{E}^{\eta}\int_0^{t\wedge 	\bar{\tau}_{\eps}}\phi(    \bar{X}^{\eps,N}_{s})\mathd s&\lesssim \mathbb{E}\int_0^t\int_{\mathbb{T}^{2N}}\phi(y^N)\mathd y^N\mathd s\leq C_{N,T},
		\end{align*} 
		where we used  the fact that $\phi$ is an integrable function on  torus $\mathbb{T}^{2N}$ to establish the final inequality.
	\end{proof}
Step 3:	Based on the uniform estimate for $	\left| 	\mathbb{E}^{\eta}	\Phi_{\eps}(    \bar{X}^{\eps,N}_{t\wedge 	\bar{\tau}_{\eps}})\right| ,$ we conduct the analysis  for the potential $\mathbb{E}^{\eta}	\Phi_{\eps}(    \bar{X}^{\eps,N}_{t\wedge \tau_{\eps}})$ and establish well-posedness of the particle system \eqref{eqt:vortex} following  the similar idea  in \cite{fontbona2007paths}  in this step.
	
	We decompose the expectation of the potential energy $\mathbb{E}^{\eta} \Phi_{\eps}(    \bar{X}^{\eps,N}_{T\wedge \bar{\tau}_{\eps}})$ as follows:
	\begin{align*}
		\mathbb{E} ^{\eta}\Phi_{\eps}(    \bar{X}^{\eps,N}_{T\wedge \bar{\tau}_{\eps}})= \mathbb{E}^{\eta}\[ \Phi_{\eps}(    \bar{X}^{\eps,N}_{ \bar{\tau}_{\eps}})I_{\bar{\tau}_{\eps}\leq T}\]+\mathbb{E}^{\eta}\[ \Phi_{\eps}(    \bar{X}^{\eps,N}_{T})I_{\bar{\tau}_{\eps}> T}\]\nonumber.
	\end{align*}
 Notice that  $\bar{\tau}_\eps\leq T$ means that there exists random indexs $i'\neq j'$  that $|    \bar{X}_{\eps,i'}(\bar{\tau}_{\eps})-    \bar{X}_{\eps,j'}(\bar{\tau}_{\eps})|=\eps$, we then have 
\begin{align*}
\mathbb{E}^{\eta} \Phi_{\eps}(    \bar{X}^{\eps,N}_{ \bar{\tau}_{\eps}\wedge T}) &=\mathbb{E}^{\eta}G_{\eps}(    \bar{X}^{i'}_{\bar{\tau}_{\eps}}-    \bar{X}_{\bar{\tau}_{\eps}}^{j'})I_{\bar{\tau}_{\eps}\leq T}+\mathbb{E}^{\eta}\sum_{i\neq j,i\neq i',j\neq j' i,j=1,\cdots,N}G_{\eps}(    \bar{X}^{i}_t-    \bar{X}_t^{j})I_{\bar{\tau}_{\eps}\leq T}
\\&+\mathbb{E}^{\eta}\sum_{i\neq j, i,j=1,\cdots,N}G_{\eps}(    \bar{X}^{i}_t-    \bar{X}_t^{j})I_{\bar{\tau}_{\eps}>  T},
\end{align*}
	Due to the fact  that $G_{\eps} \leq  C\log(|x|\vee \eps)+C$  and the compactness of $\mathbb{T}^2$,  we notice \begin{align*}
		\mathbb{E}^{\eta}\sum_{i\neq j,i\neq i',j\neq j' i,j=1,\cdots,N}G_{\eps}(    \bar{X}^{i}_t-    \bar{X}_t^{j})I_{\bar{\tau}_{\eps}\leq T}+\mathbb{E}^{\eta}\sum_{i\neq j, i,j=1,\cdots,N}G_{\eps}(    \bar{X}^{i}_t-    \bar{X}_t^{j})I_{\bar{\tau}_{\eps}>  T}\leq C
	\end{align*}
where $C$ is a positive constant. 

	Using the fact that $G_{\eps} \leq  C\log(|x|\vee \eps)+C$ again and the fact that $|    \bar{X}_{\eps,i'}(\bar{\tau}_{\eps})-    \bar{X}_{\eps,j'}(\bar{\tau}_{\eps})|=\eps,$
	 we then have \begin{align*}
	 	\mathbb{E}^{\eta} \Phi_{\eps}(    \bar{X}^{\eps,N}_{	\bar{\tau}_{\eps}\wedge T}) &\leq C\log\eps\mathbb{P}^{\eta}(\bar{\tau}_{\eps}\leq T)+C,
	 \end{align*}
	 where $C$ is a positive constant. 

	By Lemma \ref{lem:vortexito} and Lemma \ref{lem:zxcvbnn}, we know
	  \begin{align*}
		\left| 	\mathbb{E}^{\eta}	\Phi_{\eps}(    \bar{X}^{\eps,N}_{t\wedge \bar{\tau}_{\eps}})\right| \leq C_{N,T,\sigma,\eta},\quad \forall t\in[0,T],\quad\forall \eps\in(0,\eta),
	\end{align*} 
we then have
\begin{align*}
		-C_{N,T,\sigma,\eta}\leq \mathbb{E}^{\eta} \Phi_{\eps}(    \bar{X}^{\eps,N}_{ \bar{\tau}_{\eps}\wedge T}) \leq C\log\eps\mathbb{P}^{\eta}(\bar{\tau}_{\eps}\leq  T)+C.
\end{align*}
Let $\eps\downarrow 0,$ we have
	\begin{align*}
		\mathbb{P}^{\eta}(\bar{\tau}_{\eps}\leq T)\leq  \frac{C_{N,T,\sigma,\eta}}{-\log\eps}\rightarrow 0,
	\end{align*}
Recall that $	\mathbb{P}^{\eta} (\mathd x^N,\mathd\omega)=	l^{\eta}(\mathd x^N)\times\mathbb{P}(\mathd \omega) $ is a probability measure on $\Omega_{1}=\mathbb{T}^{2N}\times\Omega$, we then conclude that
\begin{align*}
\mathbb{P}(\lim_{\eps\rightarrow 0 }\bar{\tau}_{\eps}\leq T)=0\quad l^{\eta}(\mathd x^N)-a.s..
\end{align*}
Consequently, $\eta>0$ being arbitrary, we then deduce that
\begin{align}\label{inte}
	\mathbb{P}(\lim_{\eps\rightarrow 0 }\bar{\tau}_{\eps}\leq T)=0,\quad a.e..
\end{align}

	Recall  that $(X^{\eps}_t)_{t\in[0,T]}$ is the unique probabilistically strong solution to the regularized vortex system \eqref{eqt:apppro} with random initial data $X^N(0)$  of law $\mathcal{L}^{\otimes N}(X(0))$ given in Assumption \ref{initial} satisfying $X^{\eps}_t=\varphi^{\eps}_t(X^N(0)),$ where $(\varphi^{\eps}_t(x^N))_{t\in[0,T]}$ is the unique probabilistically strong solution to the regularized vortex system \eqref{eqt:apppro} with deterministic initial data $x^N,$ and  the stopping time  $\tau_{\eps}$:
\begin{equation}
	\tau_{\eps}=\inf\{t\in [0,T], \exists i\neq j, |X_{\eps,i}(t)-X_{\eps,j}(t)|\leq \eps\}\nonumber.
\end{equation}
Integrating \eqref{inte} over $\mathbb{T}^{2N}$ with respect to $\bar{\rho}_0^{\otimes N}\mathd x^N,$ where $\bar{\rho_0}$ is the density of $\mathcal{L}(X(0))$ given in Assumption \ref{initial},
we conclude that 
\begin{align*}
	\mathbb{P}(\lim_{\eps\rightarrow 0 }\tau_{\eps}\leq T)=0.
\end{align*}
It is then straightforward to find that $  X$ defined by 
\begin{align*}
	  X^N(t)= X^{\eps,N}(t),\quad \text{ for  }  t\leq \tau_{\eps}, \quad \eps>0,
\end{align*}
is the desired unique probabilistically strong solution to \eqref{eqt:vortex}.
\end{proof}
\section{Disintegration}\label{sec:disinte}
We  show a disintegration result including a short proof for the sake of completeness.
\setcounter{equation}{0}
\renewcommand{\theequation}{B.\arabic{equation}}
\begin{lemma}\label{lem:disinte}
	Under Assumption \ref{assumptiona},  for fixed $k\geq 1,$ given a continuous $\mathbb{T}^{2k}$ valued $(\mathcal{F}_{t})_{t\in[0,T]}$- adapted stochastic process $\{  Z_k(t),t\in[0,T]\},$  for the random measures on $\mathbb{T}^{2k},$ \footnote{The random measure $$\mathcal{L}(Z_k(t)| \mathcal{F}^{\xi^k,W}_T)(\mathd x^k)$$  on $\mathbb{T}^{2k}$ can be written as $Q(\mathd x^k, z^k, y)|_{(z^k,y)=(\xi^k,W)}.$ Here $Q(\mathd x^k, z^k,y)$ is a function of two variables $(z^k, y)\in \mathbb{R}^k\times C^{\otimes d}([0,T];\mathbb{T}) $ and $B\in \mathbb{B}(\mathbb{T}^{2k}),$ such that  $Q(B, z^k,y)$ is $ \mathbb{B}( \mathbb{R}^k)\otimes \mathbb{B}(C^{\otimes d}([0,T];\mathbb{T}))$-measurable in  $(z^k, y)$ for fixed $B$ and a measure in $B$ for fixed $(z^k, y).$ We use $$\mathcal{L}(Z_k(t)| \mathcal{F}^{\xi^k,W}_T)(\mathd x^k,z^k,  W(\omega))$$ to denote $Q(\mathd x^k, z^k, y)|_{y=W}.$ }
	\begin{align*}
		\rho_t^Z(\mathd x^k, z^k,\omega)\assign\mathcal{L}((  Z_k(t))| \mathcal{F}^{\xi^k,W}_T)(\mathd x^k,z^k,W(   \omega)),\quad t\in[0,T],
	\end{align*}
	where $(z^k,\omega)$ belongs to  the extended probability space $(\mathbb{R}^k\otimes\Omega,\overline{\mathbb{B}(\mathbb{R}^k)\otimes \mathcal{F}},\mathcal{L}(\xi^k)\otimes\mathbb{P})$ and $\xi^k$ denotes intensities $(\xi_1,\cdots,\xi_k),$ we have
	\begin{enumerate}
		\item  For every $\varphi\in C(\mathbb{T}^{2k}),$ 
			$\{	\left\langle \varphi , \rho_t^Z\right\rangle ,t\in[0,T]\}$ is a continuous $(\overline{\mathbb{B}(\mathbb{R}^k)\otimes \mathcal{F}^{W}_{t}})_{t\in[0,T]}$ adapted process on the extended probability space $(\mathbb{R}^k\otimes\Omega,\overline{\mathbb{B}(\mathbb{R}^k)\otimes \mathcal{F}},\mathcal{L}(\xi^k)\otimes\mathbb{P})$.
	In particular, $$(\rho_t^Z)_{t\in[0,T]} \in C([0,T];\mathcal{P}(\mathbb{T}^{2k}))\quad \mathcal{L}(\xi^k)\otimes\mathbb{P}-a.s..$$ 
		
		\item It holds that
		\begin{align*}
			\mathcal{L}(  Z_k(t),\xi^k)| \mathcal{F}^{W}_T)(\mathd x^k,\mathd z^k,\omega)=\rho_t^Z(\mathd x^k, z^k,\omega)\mathcal{L}(\xi^k)(\mathd z^k),\quad \forall t\in[0,T],\quad \mathbb{P}-a.s..
		\end{align*}
	\end{enumerate}

	  Given a $\sigma$-algebra $\mathcal{G}$ on $\mathbb{R}^k\otimes\Omega,$ the $\sigma$-algebra  $\overline{\mathcal{G}}$  means the completion of $\sigma$-algebra $\mathcal{G}$  with respect to $\mathcal{L}(\xi^k)\otimes\mathbb{P}.$  
	See relevant definitions for product space in Notations.
\end{lemma}

\begin{proof}[Proof of Lemma \ref{lem:disinte}]
	We first prove (1).
	The continuity of $\{	\left\langle \varphi , \rho_t^Z\right\rangle ,t\in[0,T]\}$  and $\{	\rho_t^Z,t\in[0,T]\}$  is obtained by the continuity of $Z_k(\cdot).$
	We mainly prove the adaptness of $\{	\left\langle \varphi , \rho_t^Z\right\rangle ,t\in[0,T]\}.$ For fixed $t\in[0,T],$
	using  the fact that $W_k(t+\cdot)-W_k(t) \perp \mathcal{F}_t , \forall k\in\{1,\cdots,d\},$ we have $\mathcal{F}_t\perp \mathcal{F}^{\xi^k,W}_T|\mathcal{F}^{\xi^k,W}_t,$  which implies that  
	\begin{align}\label{2.1}
		\mathbb{E}[ \varphi(Z_k(t)) |\mathcal{F}^{\xi^k,W}_T]=\mathbb{E}[ \varphi(Z_k(t)) | \mathcal{F}^{\xi^k,W}_t],\quad \mathbb{P}-a.s..
	\end{align} 
	We then have \footnote{We use $\tilde{\mathbb{E}}[\cdot]$ to denote taking expectation under $\mathcal{L}(\xi^k)\otimes\mathbb{P}.$ }
	\begin{align*}
		&\tilde{\mathbb{E}}\bigg[\bigg|\left\langle \varphi , \mathcal{L}((Z_k(t))| \mathcal{F}^{\xi^k,W}_T)(\mathd x^k,z^k,W(   \omega))\right\rangle-	\left\langle \varphi , \mathcal{L}((Z_k(t))| \mathcal{F}^{\xi^k,W}_t)(\mathd x^k,z^k,W|_{[0,t]}(   \omega))\right\rangle\bigg|\bigg]
		\\=&\int_{\Omega}\int_{\mathbb{R}^k}\bigg[\bigg|\left\langle \varphi , \mathcal{L}((Z_k(t))| \mathcal{F}^{\xi^k,W}_T)(\mathd x^k,z^k,W(   \omega))\right\rangle\\&-	\left\langle \varphi , \mathcal{L}((Z_k(t))| \mathcal{F}^{\xi^k,W}_t)(\mathd x^k,z^k,W|_{[0,t]}(   \omega))\right\rangle\bigg|\bigg]\mathcal{L}(\xi^k)(\mathd z^k)\mathd \mathbb{P}
		\\=&\int_{\Omega}\bigg[\bigg|	\left\langle \varphi , \mathcal{L}((Z_k(t))| \mathcal{F}^{\xi^k,W}_T)(\mathd x^k,\xi^k(\omega),W(   \omega))\right\rangle-	\left\langle \varphi , \mathcal{L}((Z_k(t))| \mathcal{F}^{\xi^k,W}_t)(\mathd x^k,\xi^k(\omega),W|_{[0,t]}(   \omega))\right\rangle\bigg|\bigg]\mathd \mathbb{P}
		\\=&\mathbb{E}\bigg[\bigg|\mathbb{E}[ \varphi(Z_k(t)) |\mathcal{F}^{\xi^k,W}_T]-\mathbb{E}[ \varphi(Z_k(t)) | \mathcal{F}^{\xi^k,W}_t]\bigg|\bigg]=0,
	\end{align*}
	where we used the independence $\xi^k\perp \mathcal{F}^{W}_T$ to get the second inequality and \eqref{2.1} to get the forth equality. Since $\left\langle \varphi , \mathcal{L}((Z_k(t))| \mathcal{F}^{\xi^k,W}_t)(\mathd x^k,z^k,W|_{[0,t]}(   \omega))\right\rangle$ is $\overline{\mathbb{B}(\mathbb{R}^k)\otimes \mathcal{F}^{W}_{t}}$ measurable, $$\left\langle \varphi , \rho_t^Z\right\rangle=	\left\langle \varphi , \mathcal{L}((Z_k(t))| \mathcal{F}^{\xi^k,W}_T)(\mathd x^k,z^k,W(   \omega))\right\rangle$$ is $\overline{\mathbb{B}(\mathbb{R}^k)\otimes \mathcal{F}^{W}_{t}}$ measurable.

	We now prove (2).
	By the monotone class theorem, it is sufficient for us to prove that 
	for any bounded $\mathbb{B}(\mathbb{T}^{2k})$ measurable function $f_1(x^k)$ and bounded $\mathbb{B}(\mathbb{R}^{k})$ measurable function  $f_2(z^k),$  it holds $\mathbb{P}$-almost surely that
	\begin{align*}
		&\int_{\mathbb{R}^k}\int_{\mathbb{T}^{2k}}f_1(x^k)f_2(z^k)\mathcal{L}(Z_k(t),\xi^k)| \mathcal{F}^{W}_T)(\mathd x^k,\mathd z^k,W(   \omega))\\=&\int_{\mathbb{R}^k}\int_{\mathbb{T}^{2k}}f_1(x^k)f_2(z^k)\mathcal{L}((Z_k(t))| \mathcal{F}^{\xi^k,W}_T)(\mathd x^k,z^k,W(   \omega))\mathcal{L}(\xi^k)(\mathd z^k).
	\end{align*}
	By the properties of conditional expectation, we have 
	\begin{align*}
		&\int_{\mathbb{R}^k}\int_{\mathbb{T}^{2k}}f_1(x^k)f_2(z^k)\mathcal{L}(Z_k(t),\xi^k)| \mathcal{F}^{W}_T)(\mathd x^k,\mathd z^k,W(   \omega))\\=
		&\mathbb{E}\[f_1(Z_k(t))f_2(\xi^k)|\mathcal{F}^{W}_T\]
		\\=&\mathbb{E}\[\mathbb{E}\[f_1(Z_k(t))f_2(\xi^k)|\mathcal{F}^{W}_T\]|\mathcal{F}_T^{W}\vee\mathcal{F}^{\xi ^k}\]
		\\=&\mathbb{E}\[\int_{\mathbb{T}^{2}}f_1(x^k)f_2(\xi^k)\mathcal{L}((Z_k(t))| \mathcal{F}^{\xi^k,W}_T)(\mathd x^k,\xi^k(\omega),W(   \omega))|\mathcal{F}^{W}_T\],
	\end{align*}
	Due to the independence  $\xi^k\perp\mathcal{F}^{W}_T$ given in Assumption \ref{assumptiona},
	\begin{align*}
		&\mathbb{E}\[\int_{\mathbb{T}^{2}}f_1(x^k)f_2(\xi^k)\mathcal{L}((Z_k(t))| \mathcal{F}^{W}_T)(\mathd x^k,\xi^k(\omega),W(   \omega))|\mathcal{F}^{W}_T\]
		\\&=	\int_{\mathbb{R}^k}\int_{\mathbb{T}^{2k}}f_1(x^k)f_2(z^k)\mathcal{L}((Z_k(t))| \mathcal{F}^{\xi^k,W}_T)(\mathd x^k,z^k,W(   \omega))\mathcal{L}(\xi^k)(\mathd z^k).
	\end{align*}
	
	The proof is then completed.
	
\end{proof}
\paragraph{\textbf{Acknowledgement:}}The authors would like to thanks Rongchan Zhu and Zhenfu Wang for quite useful discussion and advice.

\bibliographystyle{alpha}
\bibliography{vortexgeneral}

\end{document}